\documentclass[preprint,3p,10pt,a4paper]{elsarticle}
\usepackage[dvipsnames]{xcolor}
\usepackage[ruled, linesnumbered]{algorithm2e}

\usepackage[utf8]{inputenc}
\usepackage[english]{babel}
\usepackage{bm}
\usepackage{placeins}
\usepackage{xfrac}
\usepackage{epstopdf,cmap,amsmath,amssymb,amsfonts,mathtext,enumerate,float,natbib,indentfirst,graphicx,multirow,color,setspace,subfigure,amsthm,chngcntr,url}
\usepackage{lmodern}
\usepackage{mathrsfs}
\usepackage{hyperref}
\usepackage{cleveref}
\usepackage{booktabs}
\usepackage{array}
\usepackage{bm}
\usepackage{float}
\usepackage{siunitx}

\newtheorem{proposition}{Proposition}
\newtheorem{remark}{Remark}

\usepackage{tikz}

\usetikzlibrary{shadows}
\usetikzlibrary{calc,spy}
\usepackage{pgfplots}
\pgfplotsset{compat=1.18}
\usetikzlibrary{arrows.meta,positioning,calc,fit,backgrounds,patterns}
\usetikzlibrary{intersections}

\usepackage{pgfplotstable}
\usepgfplotslibrary{fillbetween}

\definecolor{myPurple}{RGB}{142,124,195}   
\definecolor{myYellow}{RGB}{230,184,74}    

\usepackage{lineno} 

\hypersetup{
    colorlinks=true,
    linkcolor=blue,
    citecolor=blue,
    urlcolor=blue,
    pdftitle={Parameterization-driven arbitrary Lagrangian–Eulerian method for large-deformation isogeometric fluid–structure interaction},
    pdfauthor={Ye Ji}
}

\graphicspath{{fig/}{data/}}

\begin{document}
\baselineskip11pt

\begin{frontmatter}

\title{Parameterization-driven arbitrary Lagrangian–Eulerian method for large-deformation isogeometric fluid–structure interaction}

\author[label_tud]{Jingya Li}
\ead{j.li-9@tudelft.nl}
\author[label_tud]{Ye Ji\corref{cor1}}
\ead{y.ji-1@tudelft.nl}
\author[label_tue]{Hugo Verhelst}
\ead{h.m.verhelst@tue.nl}
\author[label_tud_me]{Henk den Besten}
\ead{Henk.denBesten@tudelft.nl}
\author[label_tud]{Matthias M\"oller}
\ead{m.moller@tudelft.nl}

\cortext[cor1]{Corresponding author}
\address[label_tud]{Delft Institute of Applied Mathematics, Delft University of Technology, Delft, 2628 CD, the Netherlands}
\address[label_tue]{Department of Mechanical Engineering, Eindhoven University of Technology, Eindhoven, 5612 AE, the Netherlands}
\address[label_tud_me]{Maritime and Transport Technology Department, Delft University of Technology, Delft, 2628 CD, the Netherlands}

\begin{abstract}
    Body-fitted arbitrary Lagrangian–Eulerian (ALE) methods provide a sharp representation of the fluid–structure interface but rely on mesh-update strategies that incrementally deform a reference configuration. Because each update is applied on top of the previous one, distortion accumulates over time, and large structural motion eventually leads to loss of mesh validity. This path dependence is inherent to the mesh-motion formulation itself: it follows from updating each mesh from the previous one, and cannot be eliminated by refining the underlying extension operator. To address this issue, we reformulate the ALE mesh-motion problem in the isogeometric setting as a sequence of independent domain parameterization problems. At each time step, a multi-patch spline parameterization of the fluid domain is constructed from the current interface geometry; its validity and quality therefore depend on the current configuration alone, not on the cumulative history of mesh deformation. Mesh motion follows directly from this per-step parameterization, rather than being the primary unknown of an auxiliary mesh-extension problem driven by interface displacement, so the cumulative distortion characteristic of classical ALE is absent by construction. Three technical components realize this framework: (i) a barrier-function-based spline parameterization that enforces a strictly positive Jacobian at every time step; (ii) a tangential-slip reparameterization that handles sustained large rotations of closed interfaces, where no fixed boundary-to-parameter correspondence is admissible; and (iii) a constant-preserving quasi-interpolation operator for solution transfer between consecutive parameterizations, which preserves uniform states in the adopted advective ALE update. The framework is formulated in a partitioned setting, so that the fluid and structural solvers remain independent. Our numerical examples start with a purely kinematic example motivating the use of higher-order spline-based mesh formulations over conventional piecewise-linear methods in terms of the maximum allowed deformation the mesh can handle. Thereafter, we validate our method in combination with higher-order isogeometric mesh descriptions on two two-dimensional FSI benchmarks and a prescribed-rotation flow benchmark, covering standard and large-rotation regimes, and on three-dimensional falling-sphere and rotor problems. On the rotating-square benchmark, the tangential-slip strategy sustains continuous rotation beyond the admissible range of conventional mesh-deformation ALE methods with a fixed boundary correspondence. On the Turek–Hron and perpendicular-flap benchmarks, the method reproduces published reference displacements and responses, while the minimum scaled Jacobian remains bounded away from zero throughout the simulation. A prescribed-translation sphere test isolates mesh robustness under a $40\,\mathrm{m}$ displacement, and a fully coupled free-fall test reproduces the wall-corrected terminal velocity to within $0.29\%$. A three-dimensional rotor example further demonstrates that the framework extends naturally to volumetric spline parameterizations. Finally, we show that the per-step spline parameterizations can be used directly within a standard finite element solver, yielding results consistent with the isogeometric solution on the same sequence of geometries. This decoupling of geometry construction from field discretization positions the proposed framework as a solver-agnostic geometry component, compatible with both isogeometric and classical finite-element FSI pipelines. The focus is therefore on the geometry-construction module, rather than on claiming general accuracy or efficiency advantages of isogeometric analysis over finite-element remeshing strategies.
\end{abstract}


\end{frontmatter}


\section{Introduction}
\label{sec:introduction}

Many problems in computational fluid dynamics are posed on domains
whose geometry evolves in time. These include free-surface and
multiphase flows~\citep{Katopodes2018FreeSurface}, flows past moving
or rotating bodies, and fluid--structure interaction (FSI) problems~\citep{Tezduyar2001MovingBoundaries,FSI_Book}. In all of these settings, the computational mesh must be updated at every time step to conform to the moving boundary, and the validity of this update determines the range of deformations that the simulation can sustain. FSI is particularly demanding in this respect: the fluid domain is not prescribed a priori but follows from the coupled fluid--structure response, so the mesh must track an interface that can accumulate large displacement or rotation over the course of the simulation.

FSI methods can be broadly classified by how they represent the fluid--structure interface \citep{Hou2012FSIReview,Takagi2012EulerianFSIReview}. \textit{Non-body-fitted} approaches discretize the fluid equations on a background grid that does not conform to the interface, avoiding the difficulties of a deforming mesh. This class includes immersed boundary methods \citep{Peskin2002IBM}, CutFEM and related unfitted formulations \citep{Schott2019CutFEM,Burman2025CutFEM}, and their isogeometric counterparts \citep{Marussig2018TrimmedIGA,Kamensky2015Immersogeometric}. Their flexibility, however, comes at the price of a non-conforming interface representation, which typically requires special techniques to recover near-wall quantities such as interface tractions, the very quantities that drive the FSI coupling. \textit{Body-fitted} approaches align the computational mesh with the interface and provide a sharp geometric representation of the boundary; for FSI, this sharpness is not an optional refinement but a direct enabler of accurate interface loads \citep{Hosters2018NURBSCoupling,li2026isogeometric}.

Within the body-fitted class, the arbitrary Lagrangian--Eulerian (ALE) formulation \citep{Hirt1974ALE,Donea2004ALE} is the most widely used and also the oldest framework. In ALE, mesh nodes on the interface move with the structure, while interior nodes are repositioned by an auxiliary mesh-motion strategy, most commonly a harmonic \citep{Johnson1994HE}, biharmonic \citep{Helenbrook2003BHE}, or linear-elastic extension of the interface displacement \citep{Tezduyar1992LE, Blom2000ElasticMesh}. Because the mesh itself evolves in time, any ALE discretization must additionally comply with the discrete geometric conservation law, a sufficient condition for the nonlinear stability of the time integration \citep{Farhat2001DGCL,Formaggia2004DGCL}.

This framework has since been transferred to the isogeometric setting \citep{HUGHES20054135}, giving rise to a family of IGA-ALE formulations in which spline and NURBS bases replace piecewise-polynomial shape functions for both geometry representation and solution approximation \citep{Bazilevs2008IGAFSI,Bazilevs2006IGAArterial,Bazilevs2013ALEVMS,Bazilevs2023IGAFSI}. Beyond the familiar benefits of geometric exactness at curved interfaces and higher-order inter-element continuity, IGA-ALE offers one feature that is especially pertinent to FSI coupling: the interface tractions emerge as naturally smooth fields, obviating the post-processing usually required to reconstruct them from piecewise-polynomial stresses \citep{Hosters2018NURBSCoupling,li2026isogeometric}. The present work is set within this IGA-ALE framework and adopts a partitioned coupling strategy: the fluid and structural subproblems are solved by independent solvers that exchange only interface displacements and tractions at each coupling step.

The principal limitation of classical body-fitted ALE methods lies in the mesh-motion problem itself. In standard mesh-deformation formulations, the mesh at time $t^{n+1}$ is obtained from the mesh at time $t^{n}$ by applying an incremental displacement field, typically the solution of an auxiliary extension problem driven by the interface displacement. The resulting update is path-dependent: the mesh at any instant is the composition of all previous updates, and a local distortion introduced at one step is carried forward to every subsequent step. For moderate interface motion this accumulation remains bounded, but under large cumulative displacement or rotation it is not: the Jacobian determinant of the ALE mapping eventually passes through zero, and the mesh becomes invalid. The top row of Figure~\ref{fig:paradigm_comparison} illustrates this behaviour for a rotating-square configuration.

Three concrete consequences follow. First, element distortion degrades the interface tractions that drive the coupled problem. Second, the Jacobian determinant of the ALE mapping is not guaranteed to remain positive: once it changes sign at any point of the mesh, the mapping is no longer a bijection and the simulation has to be stopped. Third, on distorted meshes the discrete geometric conservation law can no longer be enforced, and the associated nonlinear stability guarantees~\citep{Farhat2001DGCL} are lost. Remeshing and mesh-smoothing strategies can mitigate these effects \citep{Takizawa2021LENoAccum}, but the underlying cause persists: as long as the mesh at $t^{n+1}$ is constructed from the mesh at $t^{n}$, any refinement of the extension operator stays within the same path-dependent construction.

Several strategies have been proposed to push body-fitted ALE beyond the regime of small-to-moderate interface motion. Fixed-mesh ALE formulations \citep{Codina2009FMALE,BaigesCodina2010FMALE} avoid mesh motion altogether by solving the governing equations on a fixed background mesh, retaining the ALE mapping only in a virtual sense; this sidesteps the distortion problem but sacrifices the sharp body-fitted interface that motivates the ALE formulation itself. A more direct line of work keeps the body-fitted character of ALE but replaces mesh deformation with per-step parameterization: rather than updating the existing mesh, a new conforming mesh is generated at every time step from the current interface configuration. Recent work by Schwentner and Fries \citep{Schwentner2025MeshGenFSI} and by Schalk and Fries \citep{mesh_regeneration_Graz} shows that this strategy is viable for large-deformation FSI in the finite element setting, and Di Cristofaro et al.~\citep{DiCristofaro2025ALEFlying} demonstrate a related ALE approach with frequent remeshing and field remapping for three-dimensional FSI with large boundary motion, including flapping-wing and falling-sphere configurations. Within this line of work, the transfer of discrete fields between successive meshes emerges as a central concern; constant-preserving transfer operators \citep{Farrell2009ConservativeInterpolation,Pont2017Interpolation} are of particular interest because they preserve uniform states accross the transfer. For conservative ALE formulations, this transfer property must be complemented by a discrete geometric identity for evolving Jacobian and mesh flux \citep{Farhat2001DGCL}.

In the isogeometric setting, ALE-FSI formulations have so far been developed primarily within the classical mesh-deformation approach \citep{Bazilevs2008IGAFSI,Bazilevs2023IGAFSI}. While these formulations benefit from the geometric exactness and higher-order continuity of spline discretizations, they inherit the structural limitation identified in the previous section: the fluid mesh at any time instant is still the product of a cumulative sequence of incremental updates.

An alternative approach becomes available when the fluid discretization is based on spline representations of the geometry. 

The present work is aligned with this mesh-regeneration viewpoint, but its emphasis is different from conventional finite-element remeshing. The regenerated object is a spline parameterization of the fluid domain. This makes it possible to control validity directly through the positivity of the parameterization Jacobian, to introduce tangential slip at the level of boundary pre-images, and to reuse the same geometry update in both isogeometric and finite element flow solvers.

In isogeometric analysis, the physical domain is represented not by a collection of nodal coordinates, but by a spline- or NURBS-based parameterization that maps a fixed parametric domain onto the physical domain.
This distinction is more than representational. A finite element mesh is tied to its construction history: its quality depends on how it was obtained, and subsequent updates can only reposition its existing nodes. A spline parameterization, in contrast, is a self-contained geometric object that can be rebuilt from scratch at any time instant, using the current interface configuration as its sole input. The bottom row of Figure~\ref{fig:paradigm_comparison} illustrates this property on the same rotating-square configuration: at every snapshot, the mesh is an independently constructed parameterization rather than the accumulated result of all previous updates.

\begin{figure}[H]
    \centering
    \includegraphics[width=\linewidth]{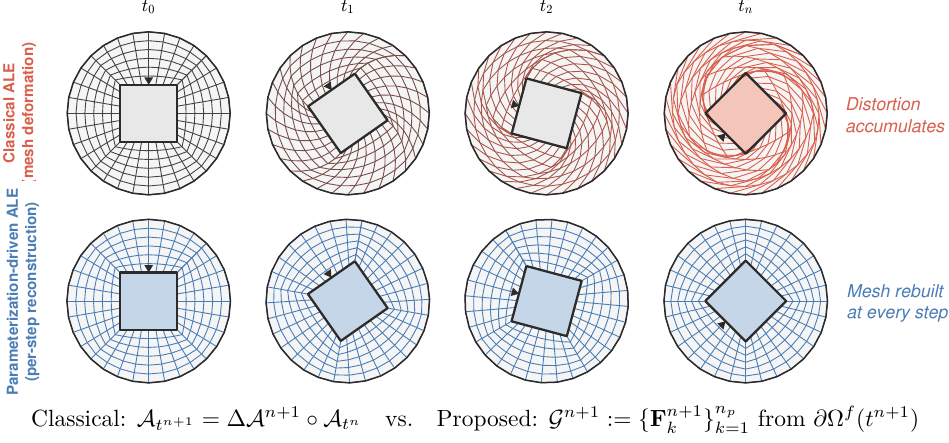}
\caption{Two approaches to ALE mesh motion under large interface displacement, illustrated on a rotating-square configuration. \textbf{Top:} in the classical mesh-deformation ALE approach, the mapping $\mathcal A_{t^{n+1}}$ is obtained by composing the preceding mapping $\mathcal A_{t^n}$ with an incremental update $\Delta\mathcal A^{n+1}$; the resulting path-dependent distortion may accumulate until the mesh becomes invalid. \textbf{Bottom:} in the proposed parameterization-driven approach, the spline geometry $\mathcal G^{n+1}:=\{\mathbf F_k^{n+1}\}_{k=1}^{n_p}$ is reconstructed directly from the current interface configuration $\partial\Omega^f(t^{n+1})$, independently of $\mathcal G^n$. The small triangular marker on the inner square is used to track its cumulative rotation.}
    \label{fig:paradigm_comparison}
\end{figure}

We therefore reformulate the ALE mesh-motion problem. Instead of updating the fluid mesh incrementally from one time step to the next, we construct it at each step as an independent spline parameterization of the current fluid domain, using the current interface configuration as input. The past motion of the interface is recorded by the structural solver alone; the fluid mesh at time $t^n$ does not depend on the one at $t^{n-1}$. The mesh velocity that couples the fluid equations to the moving domain is then recovered from successive parameterizations rather than prescribed through an auxiliary mesh-extension problem.

This reformulation has consequences that reach beyond robustness under large deformations. The cumulative distortion typical of classical ALE does not arise at all since each parameterization is built independently. The solution transfer between consecutive parameterizations is performed by a constant-preserving quasi-interpolation operator, which preserves uniform states in the adopted advective ALE update. While this transfer property alone does not establish the fully discrete geometric conservation law, the latter is recovered when it is combined with a consistent discrete relation between the evolving Jacobian and the mesh flux. Consequently, the discrete geometric conservation law holds algebraically, and the associated nonlinear stability guarantees~\citep{Farhat2001DGCL} are retained. Finally, because the parameterization depends only on the interface geometry and not on the fluid discretization, the same sequence of parameterizations can drive either an isogeometric or a standard finite element fluid solver. The purpose of the present work is therefore not to claim a universal accuracy or efficiency advantage of IGA over finite-element remeshing strategies. Instead, the contribution is a spline-parameterization-driven ALE geometry module whose validity is controlled at the level of the parameterization and which can be reused by different downstream flow discretizations. To the best of our knowledge, this view of ALE mesh motion has not been pursued in the isogeometric FSI literature, even though the required spline parameterization machinery for complex domains is now available~\citep{Ji2021BarrierNURBS}. 

Our main contributions are as follows.
\begin{itemize}
    \item We reformulate the ALE mesh-motion problem in the isogeometric setting as a sequence of independent domain parameterization problems. At every time step, we build a new multi-patch spline parameterization of the fluid domain directly from the current interface, rather than updating the previous mesh. This replaces the path-dependent mesh-deformation approach, and therefore the cumulative distortion it carries, with an effectively history-independent construction.
    \item For closed domains that rotate through large cumulative angles, a fixed boundary-to-parameter correspondence simply cannot be maintained. Once the body has rotated far enough, it becomes incompatible with any valid mesh. We address this by letting the parametric pre-images slide tangentially along the interface. The resulting scheme handles rotation regimes that, to our knowledge, cannot be reached by any mesh-deformation ALE formulation.
    \item We use a constant-preserving quasi-interpolation operator to transfer the discrete solution between successive parameterizations.  In the adopted advective ALE update, exact reproduction of constants yields free-stream preservation. 
    \item Because each parameterization is built as a pure geometric object, nothing in the construction is tied to the field solver. We exploit this by feeding the same sequence of parameterizations into a standard finite element solver and recovering results consistent with the isogeometric computation.
\end{itemize}

The remainder of the paper is organized as follows. \Cref{sec:governing-equations} states the governing equations of the fluid and structural subproblems together with the corresponding ALE formulation. \Cref{sec:numerical-framework} develops the numerical framework: the isogeometric discretization of the fluid equations, the per-step spline parameterization of the fluid domain together with the tangential-slip reparameterization for closed configurations, the recovery of the ALE mesh velocity from two consecutive parameterizations, and the constant-preserving solution-transfer operator used for free-stream preservation in the adopted advective ALE update. \Cref{sec:numerical-examples} reports two- and three-dimensional numerical experiments: standard FSI benchmarks, a rotating-square configuration that drives the inner body through a full revolution, prescribed-translation and freely falling spheres, a three-dimensional rotor problem, and a test in which the regenerated parameterizations are coupled to a classical finite element solver. \Cref{sec:conclusions} summarizes the main findings and outlines directions for future work.

\section{Governing equations}
\label{sec:governing-equations}

In this paper, we consider a fluid--structure interaction problem in which the fluid occupies a time-dependent domain $\Omega^f(t) \subset \mathbb{R}^d$ ($d \in \{2, 3\}$), separated from the surrounding structure $\Omega^s(t)$ by a moving interface $\Gamma_{fs}(t) = \partial \Omega^f(t) \cap \partial \Omega^s(t)$. Throughout, superscripts $f$ and $s$ distinguish fluid and structural quantities, and the subscript $0$ denotes quantities in the reference configuration.

\subsection{Incompressible flow}
\label{sec:incompressible-flow}

The fluid motion is governed by the incompressible Navier--Stokes equations in the time-dependent domain $\Omega^f(t)$:
\begin{subequations}
    \begin{align}
    \rho^f \left( \frac{\partial \mathbf{u}^f}{\partial t} 
    + (\mathbf{u}^f \cdot \nabla) \mathbf{u}^f \right) 
    &= \nabla \cdot \boldsymbol{\sigma}^f + \mathbf{f}^f \quad \text{in } \Omega^f(t), \label{eq:momentum} \\
    \nabla \cdot \mathbf{u}^f &= 0 \quad \text{in } \Omega^f(t),
    \label{eq:continuity}
    \end{align}
\end{subequations}
where $\mathbf{u}^f : \Omega^f(t) \times [0,T] \to \mathbb{R}^d$ is 
the fluid velocity, $\rho^f$ is the (constant) density, 
$\mathbf{f}^f$ is a body force, and $\boldsymbol{\sigma}^f$ is the 
Cauchy stress tensor. 

For an incompressible Newtonian fluid, the stress tensor is given by
\begin{equation}
    \boldsymbol{\sigma}^f = -p^f \mathbf{I} + 2\mu^f \mathbb{D}(\mathbf{u}^f),
\end{equation}
with pressure $p^f$, dynamic viscosity $\mu^f = \rho^f \nu^f$, and symmetric strain-rate tensor $\mathbb{D}(\mathbf{u}^f) = \tfrac{1}{2}(\nabla \mathbf{u}^f + \nabla \mathbf{u}^{f \top})$. 

The system is closed by an initial condition $\mathbf{u}^f(\cdot, 0) = \mathbf{u}_0^f$ in $\Omega^f(0)$ together with Dirichlet and Neumann conditions on the inflow and outflow boundaries; the interface conditions on $\Gamma_{fs}(t)$ are deferred to \Cref{sec:coupling-conditions}.

\subsection{Arbitrary Lagrangian–Eulerian formulation}
\label{sec:ale-formulation}

The time-dependence of $\Omega^f(t)$ requires a formulation of the moving-grid. To this end, we adopt the arbitrary Lagrangian--Eulerian (ALE) framework. A reference configuration $\Omega_0^f$ is mapped onto the current fluid domain at each time instant through the application of the following:
\begin{equation}
    \mathcal{A}: \Omega_0^f \times [0, T] \to \mathbb{R}^d,
    \qquad
    \mathbf{x} = \mathcal{A}(\mathbf{X}, t) 
    = \mathbf{X} + \mathbf{u}^{\mathrm{ALE}}(\mathbf{X}, t),
\end{equation}
where $\mathbf{u}^{\mathrm{ALE}}$ is the mesh displacement field, and $\Omega^f(t) := \mathcal{A}(\Omega_0^f, t)$. The mesh velocity is the time derivative taken at fixed reference point $\mathbf{X}$,
\begin{equation}
    \hat{\mathbf{w}}(\mathbf{X}, t) 
    := \left. \frac{\partial \mathcal{A}}{\partial t} 
    \right|_{\mathbf{X}}
    = \left. \frac{\partial \mathbf{u}^{\mathrm{ALE}}}{\partial t} 
    \right|_{\mathbf{X}},
\end{equation}
which pushes forward to the current configuration as $\mathbf{w}(\mathbf{x}, t) = \hat{\mathbf{w}}(\mathcal{A}^{-1}(\mathbf{x}, t), t)$.

Under the ALE mapping, the time derivative in the momentum equation is evaluated at fixed $\mathbf{X}$ rather than fixed $\mathbf{x}$. The fluid equations become

\begin{subequations}
\label{eq:ale-ns}
\begin{align}
    \rho^f \left( 
    \left. \frac{\partial \mathbf{u}^f}{\partial t} \right|_{\mathbf{X}} 
    + \big( (\mathbf{u}^f - \mathbf{w}) \cdot \nabla \big) \mathbf{u}^f 
    \right)
    &= \nabla \cdot \boldsymbol{\sigma}^f + \mathbf{f}^f
    && \text{in } \Omega^f(t), 
    \label{eq:ale-momentum} \\
    \nabla \cdot \mathbf{u}^f 
    &= 0
    && \text{in } \Omega^f(t).
    \label{eq:ale-continuity}
\end{align}
\end{subequations} 
The quantity $\mathbf{u}^f - \mathbf{w}$ is the relative velocity of the fluid with respect to the moving mesh.

On the interface, consistency of the ALE mapping with the structural motion requires
\begin{equation}
    \mathbf{u}^{\mathrm{ALE}}(\mathbf{X}, t) 
    = \mathbf{d}^s(\mathbf{X}, t) 
    \quad \text{on } \Gamma_{fs,0} \times [0, T],
\end{equation}
where $\mathbf{d}^s$ is the structural displacement defined in \Cref{sec:structural-mechanics}. Within the interior of the fluid domain, $\mathbf{u}^{\mathrm{ALE}}$ is not uniquely determined and must be constructed such that the ALE mapping is locally invertible. Denoting the Jacobian matrix of $\mathcal{A}$ by
\begin{equation}
    \bm{\mathcal{J}}^{\mathrm{ALE}}(\mathbf{X}, t) 
    := \nabla_{\mathbf{X}} \mathcal{A}(\mathbf{X}, t),
    \label{eq:jacobian-ale}
\end{equation}
this reads
\begin{equation}
    \det \bm{\mathcal{J}}^{\mathrm{ALE}}(\mathbf{X}, t) > 0 
    \quad \text{for all } (\mathbf{X}, t) \in \Omega_0^f \times [0, T].
    \label{eq:bijectivity-continuous}
\end{equation}
The discretization strategy used to ensure \eqref{eq:bijectivity-continuous} is developed in \Cref{sec:barrier-patch-method}; the corresponding discrete criterion is stated in \Cref{sec:bijectivity-criterion}.

Figure~\ref{fig:fsi_configurations} summarizes the reference and current configurations together with the mappings that connect them on the fluid and structural sides.

\begin{figure}[H]
    \centering
    \includegraphics[width=\linewidth]{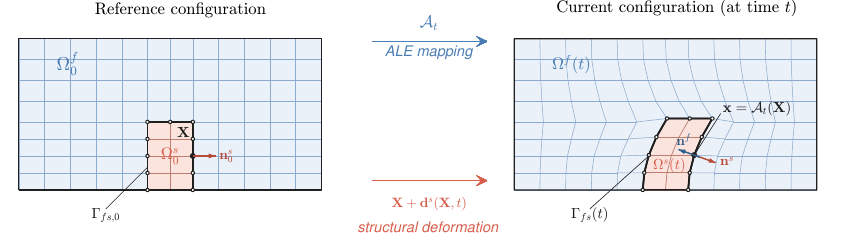}
    \caption{Configurations and mappings in the ALE--FSI framework, shown on a representative conforming parameterization in which every node on the interface belongs to both the fluid and the structural parameterization (open circles). \textbf{Left:} reference configurations of the fluid ($\Omega_0^f$) and the structure ($\Omega_0^s$), sharing the reference interface $\Gamma_{fs,0}$; $\mathbf{X}$ denotes a reference interface point, and $\mathbf{n}_0^{s}$ is the outward unit normal on the solid. \textbf{Right:} current configurations at time $t$, obtained from the reference by the ALE mapping $\mathcal{A}_t$ on the fluid side and by the structural deformation $\mathbf{X}\mapsto\mathbf{X}+\mathbf{d}^{s}(\mathbf{X},t)$ on the solid side. The two mappings agree on $\Gamma_{fs,0}$, so the current interface $\Gamma_{fs}(t)$ is the same curve from the fluid and structural sides. The point $\mathbf{x}=\mathcal{A}_t(\mathbf{X})$ lies on $\Gamma_{fs}(t)$ with outward normals $\mathbf{n}^{s}$ and $\mathbf{n}^{f}=-\mathbf{n}^{s}$.}
    \label{fig:fsi_configurations}
\end{figure}

\subsection{Structural mechanics}
\label{sec:structural-mechanics}

The structural response is described in a total Lagrangian framework with respect to a reference configuration $\Omega_0^s \subset \mathbb{R}^d$. The unknowns are the displacement field $\mathbf{d}^s : \Omega_0^s \times [0,T] \to \mathbb{R}^d$ and the corresponding velocity $\mathbf{v}^s := \partial \mathbf{d}^s / \partial t$. Balance of linear momentum reads
\begin{equation}
    \rho_0^s \frac{\partial^2 \mathbf{d}^s}{\partial t^2}
    = \nabla_0 \cdot \mathbf{P}^s + \rho_0^s \mathbf{b}^s
    \quad \text{in } \Omega_0^s \times [0,T],
    \label{eq:structure-motion}
\end{equation}
where $\rho_0^s$ is the reference mass density, $\mathbf{b}^s$ is a body force per unit mass, $\nabla_0$ denotes the gradient with respect to reference coordinates, and $\mathbf{P}^s := \mathbf{F} \mathbf{S}^s$ is the first Piola--Kirchhoff stress tensor.

The deformation gradient is defined by $\mathbf{F} := \mathbf{I} + \nabla_0 \mathbf{d}^s$, and the Green--Lagrange strain tensor by
\begin{equation}
    \mathbf{E} := \tfrac{1}{2}(\mathbf{F}^{\top}\mathbf{F} - \mathbf{I}),
\end{equation}
which captures geometric nonlinearity under large deformations. The structure is modeled as a St.~Venant--Kirchhoff material, for which the second Piola--Kirchhoff stress is
\begin{equation}
    \mathbf{S}^s = \lambda^s \operatorname{tr}(\mathbf{E})\,\mathbf{I} 
    + 2 \mu^s \mathbf{E},
\end{equation}
where $\lambda^s, \mu^s$ are the Lam\'{e} parameters. The associated Cauchy stress, used in the interface traction condition below, is obtained by the standard push-forward,
\begin{equation}
    \boldsymbol{\sigma}^s = J^{-1} \mathbf{F} \mathbf{S}^s \mathbf{F}^{\top},
    \qquad J := \det \mathbf{F}.
\end{equation}

The reference boundary $\partial \Omega_0^s$ is partitioned into three non-overlapping parts: a Dirichlet part $\Gamma_D^s$, a Neumann part $\Gamma_N^s$, and the fluid--structure interface $\Gamma_{fs,0} \subset \partial \Omega_0^s$. On $\Gamma_{fs,0}$, the traction is prescribed by the fluid and is treated in \Cref{sec:coupling-conditions}. On the remaining parts,
\begin{subequations}
\label{eq:structure-bc}
\begin{align}
    \mathbf{d}^s &= \mathbf{g}^s 
    \quad \text{on } \Gamma_D^s \times [0,T], \\
    \mathbf{P}^s \mathbf{n}_0^s &= \mathbf{h}^s 
    \quad \text{on } \Gamma_N^s \times [0,T],
\end{align}
\end{subequations}
where $\mathbf{g}^s$ is a prescribed displacement, $\mathbf{h}^s$ is a prescribed reference traction, and $\mathbf{n}_0^s$ is the outward unit normal on the reference configuration. Initial data $\mathbf{d}^s(\cdot, 0) = \mathbf{d}_0^s$ and $\mathbf{v}^s(\cdot, 0) = \mathbf{v}_0^s$ close the structural subproblem.

\subsection{Coupling conditions}
\label{sec:coupling-conditions}

At each time $t$, the fluid and structural subproblems are linked through conditions on the interface $\Gamma_{fs}(t)$. The partitioned formulation solves each subproblem independently and uses these conditions to transfer information at each coupling step \citep{li2026isogeometric}.

\paragraph{Kinematic coupling} Continuity of velocity at the interface requires
\begin{equation}
    \mathbf{u}^f(\mathbf{x}, t) = \mathbf{v}^s(\mathbf{X}, t)
    \quad \text{for } \mathbf{x} = \mathbf{X} + \mathbf{d}^s(\mathbf{X}, t), 
    \; \mathbf{X} \in \Gamma_{fs,0},
    \label{eq:velocity-continuity}
\end{equation}
where $\mathbf{v}^s = \partial \mathbf{d}^s / \partial t$ is the structural velocity. Consistently, the ALE mapping tracks the structural displacement on the interface:
\begin{equation}
    \mathbf{u}^{\mathrm{ALE}}(\mathbf{X}, t) = \mathbf{d}^s(\mathbf{X}, t)
    \quad \text{on } \Gamma_{fs,0} \times [0, T].
    \label{eq:ale-consistency}
\end{equation}
Condition \eqref{eq:ale-consistency} determines the ALE mesh displacement on $\Gamma_{fs,0}$; its extension into the interior is the subject of \Cref{sec:barrier-patch-method}.

\paragraph{Dynamic coupling} Equilibrium of tractions across the interface reads
\begin{equation}
    \boldsymbol{\sigma}^f(\mathbf{x}, t) \mathbf{n}^f 
    + \boldsymbol{\sigma}^s(\mathbf{x}, t) \mathbf{n}^s = \mathbf{0}
    \quad \text{on } \Gamma_{fs}(t),
    \label{eq:traction-continuity}
\end{equation}
where $\mathbf{n}^f$ and $\mathbf{n}^s = -\mathbf{n}^f$ are the outward unit normals on the current fluid and structural interfaces. Because the structural subproblem is posed in the reference configuration, \eqref{eq:traction-continuity} is recast using Nanson's formula $J\, \mathbf{F}^{-\top} \mathbf{n}_0^s\, \mathrm{d} A_0 = \mathbf{n}^s\, \mathrm{d} A$, which yields the reference-configuration form
\begin{equation}
    \mathbf{P}^s \mathbf{n}_0^s 
    = -\, J\, \boldsymbol{\sigma}^f \mathbf{F}^{-\top} \mathbf{n}_0^s
    \quad \text{on } \Gamma_{fs,0} \times [0, T].
    \label{eq:traction-reference}
\end{equation}
Equation \eqref{eq:traction-reference} acts as a Neumann boundary condition for the structural subproblem: the traction applied by the fluid is evaluated on the deformed interface and pulled back to the reference configuration.

\section{Parameterization-driven ALE framework}
\label{sec:numerical-framework}

This section develops the numerical framework that realizes the parameterization-driven ALE formulation. Four technical components make up the core of the framework: the isogeometric discretization of the fluid equations (\Cref{sec:iga}); the per-step spline parameterization of the fluid domain (\Cref{sec:mesh-generation}); the ALE mesh velocity recovered from successive parameterizations (\Cref{sec:ale-mesh-velocity}); and the constant-preserving solution transfer that secures the discrete geometric conservation law (\Cref{sec:solution-transfer}). A discrete criterion for the bijectivity of the resulting patch mappings is collected in \Cref{sec:bijectivity-criterion}.

\subsection{Isogeometric discretization for fluid subproblem}
\label{sec:iga}

We discretize the fluid subproblem using isogeometric analysis (IGA) \citep{HUGHES20054135}, with the specific choices detailed below: a multi-patch NURBS representation of the geometry and a Taylor--Hood-type B-spline pair for velocity and pressure.

\paragraph{B-spline and NURBS basis}
For a given open knot vector $\Xi = \{\xi_1 \leq \xi_2 \leq \cdots \leq \xi_{n+p+1}\}$, the univariate B-spline basis $\{B_{i,p}\}_{i=1}^{n}$ of degree $p$ is constructed by the Cox--de~Boor recursion \citep{HUGHES20054135,piegl2012nurbs}. Tensor-product B-spline bases on the parametric domain $\mathcal{P} = [0,1]^d$ are obtained as products of univariate bases in each parametric direction. Introducing positive weights yields the corresponding NURBS basis functions $\{N_i(\boldsymbol{\xi})\}_{i=1}^{n}$, and the geometry of each patch is represented as
\begin{equation}
    \mathbf{x}(\boldsymbol{\xi}) 
    = \sum_{i=1}^{n} N_i(\boldsymbol{\xi})\,\mathbf{x}_i,
    \qquad \boldsymbol{\xi} \in \mathcal{P},
\end{equation}
where $\{\mathbf{x}_i\}_{i=1}^{n} \subset \mathbb{R}^{d}$ are the control points.

\paragraph{Multi-patch discretization}
The fluid domain $\Omega^f(t)$ is represented as a collection of $n_p$ conforming patches (see \Cref{sec:mesh-generation}), with patch $k$ parameterized by a time-dependent geometry map
\begin{equation}
    \mathbf{F}_k(\boldsymbol{\xi}, t) 
    = \sum_{i=1}^{n_k} N_i(\boldsymbol{\xi})\,\mathbf{x}_{k,i}(t),
    \qquad \mathbf{F}_k(\cdot, t): \mathcal{P} \to \Omega_k(t),
\end{equation}
where $\{\mathbf{x}_{k,i}(t)\}$ are the control points of patch $k$ at time $t$. The discrete solution spaces are constructed by composition with $\mathbf{F}_k$, and global spaces are assembled by identifying control points shared by adjacent patches.
At time $t^n$, the multi-patch spline geometry is denoted by
\begin{equation}
    \mathcal G^n:=\{\mathbf F_k^n\}_{k=1}^{n_p}.
    \label{eq:spline-geometry-collection}
\end{equation}

\paragraph{Taylor--Hood discretization}
To satisfy the inf-sup stability condition required for a stable discretization of the incompressible Navier--Stokes equations, we adopt a Taylor--Hood-type isogeometric velocity--pressure pair~\citep{Hooseini2015TaylorHood}. The velocity space $V_h$ consists of B-splines of degree $p$ and regularity $C^{p-1}$, while the pressure space $Q_h$ consists of B-splines of degree $p-1$ and regularity $C^{p-2}$, both defined on the same physical element partition. This pair is inf-sup stable and delivers optimal-order approximation.

For $p \geq 2$, the $C^{p-1}$ regularity of the velocity space makes the velocity gradient globally continuous on each patch. The viscous stress field $\boldsymbol{\sigma}^f$ is therefore continuous within each patch, and interface tractions can be evaluated directly at the patch boundary without the projection or recovery procedures required on $C^0$ Lagrange elements \citep{Hosters2018NURBSCoupling,li2026isogeometric}.

\subsection{Per-step spline domain parameterization}
\label{sec:mesh-generation}

Using the multi-patch setting of \Cref{sec:iga}, the patch mappings $\{\mathbf{F}_k(\cdot, t)\}_{k=1}^{n_p}$ are constructed anew at each time step from the current boundary data, without reference to a deforming background mesh. This construction has two ingredients: a barrier-patch optimization procedure that produces a bijective, well-shaped parameterization from given boundary curves (\Cref{sec:barrier-patch-method}); and two application patterns—one for the open fluid domains arising in FSI coupling (\Cref{sec:open-domain}), and one for the closed annular domains used in our prescribed-motion benchmark (\Cref{sec:closed-domain}). In both settings the ingredient of \Cref{sec:barrier-patch-method} is applied patch-wise; the two settings differ only in how the patch boundaries are assembled from the interface and from fixed geometry.

\subsubsection{Barrier-patch method}
\label{sec:barrier-patch-method}

The barrier-patch method \citep{Ji2021BarrierNURBS, Ji2022PenaltyParam} builds a bijective and well-shaped NURBS parameterization of a single patch from prescribed boundary control points. Let $\mathbf{F}_k$ be the patch map of patch $k$, with boundary control points fixed by the surrounding geometry (specified separately for open and closed domains in \Cref{sec:open-domain,sec:closed-domain}). The interior control points $\{\mathbf{x}_{k,i}\}_{i \in \mathcal{I}_k}$ are determined by a two-stage optimization, where $\mathcal{I}_k$ indexes the interior degrees of freedom of patch $k$.

\paragraph{Stage 1: Bijectivity at quadrature points}
Let
\begin{equation}
    \bm{\mathcal{J}}_k(\boldsymbol{\xi}) 
    := \frac{\partial \mathbf{F}_k}{\partial \boldsymbol{\xi}}(\boldsymbol{\xi}),
    \qquad \boldsymbol{\xi} \in \mathcal{P},
    \label{eq:jacobian-patchwise}
\end{equation}
denote the Jacobian matrix of the patch map, and let $\{\boldsymbol{\xi}_q\}_{q=1}^{n_q} \subset \mathcal{P}$ be a set of sampling points (taken as the Gauss quadrature points of the element-wise integration rule). We first eliminate fold-over configurations by minimizing
\begin{equation}
    E_{\mathrm{fold}}(\mathbf{F}_k) 
    = \sum_{q=1}^{n_q} 
    \bigl[\,\delta - \det \bm{\mathcal{J}}_k(\boldsymbol{\xi}_q)\,\bigr]_{+}^{2},
    \label{eq:barrier-foldover}
\end{equation}
where $[x]_+ := \max(x, 0)$ denotes the positive part and $\delta > 0$ is a small threshold. The functional \eqref{eq:barrier-foldover} is strictly positive whenever the discrete bijectivity condition $\det \bm{\mathcal{J}}_k(\boldsymbol{\xi}_q) \geq \delta$ fails at any sampling point, and vanishes otherwise. Minimization is therefore halted once a configuration with $E_{\mathrm{fold}} = 0$ is reached.

\paragraph{Stage 2: Quality improvement}
Once bijectivity is secured, element quality is improved by minimizing the Winslow-plus-area functional
\begin{equation}
    E_{\mathrm{qual}}(\mathbf{F}_k) 
    = \lambda_1 \int_{\mathcal{P}} 
    \frac{\|\bm{\mathcal{J}}_k\|_F^2}{\det \bm{\mathcal{J}}_k} \, 
    d\boldsymbol{\xi} 
    + \lambda_2 \int_{\mathcal{P}} 
    \bigl(\det \bm{\mathcal{J}}_k\bigr)^2 \, d\boldsymbol{\xi},
    \label{eq:barrier-quality}
\end{equation} 
where $\|\cdot\|_F$ denotes the Frobenius norm. The first term is the classical Winslow functional, which penalizes angular distortion and acts as a smoother on the control net; the second term promotes uniform cell areas. The positivity of $\det \bm{\mathcal{J}}_k$ established in Stage~1 is preserved throughout Stage~2 because the Winslow integrand becomes singular as $\det \bm{\mathcal{J}}_k \to 0^+$. The weights $\lambda_1, \lambda_2 > 0$ are taken as $\lambda_1 = 1$ and $\lambda_2 = 1/A^2$, where $A = \int_{\mathcal{P}} \det\bm{\mathcal{J}}_k,d\boldsymbol{\xi}$ is the (scaled) physical area.  \Cref{sec:numerical-examples}.

Both stages are minimized over the interior control points using L-BFGS, with the boundary control points held fixed. The integrals in \eqref{eq:barrier-quality} are evaluated by the same element-wise Gauss rule used in \eqref{eq:barrier-foldover}. The two-stage procedure returns a patch map $\mathbf{F}_k$ whose discrete bijectivity is certified by Stage~1 and whose element quality has been improved by Stage~2; the same procedure is applied independently to each patch at each time step.

\subsubsection{Application to open FSI domains}
\label{sec:open-domain}

For fluid domains arising in fluid--structure interaction, the fluid occupies the region exterior to a deforming body and interior to fixed walls. We partition the fluid domain into two layers:
\begin{itemize}
    \item a \emph{near-body layer} of $n_s$ patches, each carrying one segment of the fluid--structure interface $\Gamma_{fs}(t)$ as its inner boundary;
    \item an \emph{extension layer} of $n_{\mathrm{ext}}$ patches that connects the near-body layer to the fixed outer walls (inflow, outflow, and far-field boundaries).
\end{itemize}
Only the near-body patches are reparameterized at each time step; the extension patches are fixed once and for all at initialization. The two layers are joined conformally, and the control points on the shared interface between the layers are held fixed throughout the simulation.

\paragraph{Near-body patch boundaries}
For near-body patch $k$, the boundary control points are assembled from three sources:
\begin{itemize}
    \item the \emph{inner} boundary 
    $\boldsymbol{\gamma}_k^{\mathrm{in}}(\cdot, t^n)$ tracks the 
    corresponding segment of $\Gamma_{fs}(t^n)$ and is supplied 
    by the structural subproblem through the ALE consistency 
    condition \eqref{eq:ale-consistency};
    \item the \emph{outer} boundary 
    $\boldsymbol{\gamma}_k^{\mathrm{out}}$ lies on the fixed 
    interface with the extension layer and is therefore 
    time-independent;
    \item the \emph{side} boundaries are shared with the two 
    neighboring near-body patches and are constructed once at initialization, together with the rest of the near-body layer, so that only their inner boundaries evolve with the fluid--structure interface.
\end{itemize}
Given these four boundary curves, the interior control points of $\mathbf{F}_k$ are obtained by the barrier-patch procedure of \Cref{sec:barrier-patch-method}.

\paragraph{Interface transfer}
At each time step, the structural solver returns the updated positions of interface control points on $\Gamma_{fs}(t^n)$. In the spline-based coupling setting adopted in literature \citep{li2026isogeometric,Hosters2018NURBSCoupling}, the fluid and structural solvers exchange information through NURBS/spline representation of the interface. The inner boundary curves $\{\boldsymbol{\gamma}_k^{\mathrm{in}}(\cdot, t^n)\}$ are assembled from these control points and passed to the barrier-patch optimization as fixed boundary data.

\subsubsection{Application to closed domains: ruled surface and tangential slip}
\label{sec:closed-domain}

For closed annular configurations, in which a simply connected body is fully surrounded by fluid, the fluid domain is discretized using $n_p = n_s$ conforming patches without extension layer. The inner and outer boundaries of the annulus are represented as closed B-spline curves $\boldsymbol{\gamma}^{\mathrm{in}}(\cdot, t)$ and $\boldsymbol{\gamma}^{\mathrm{out}}(\cdot)$, both parameterized over $s \in [0, 1]$ with periodic identification $s = 0 \sim s = 1$.

\paragraph{Ruled surface initialization}
An initial global parameterization of the annular domain is constructed as the ruled surface between the two boundary curves,
\begin{equation}
    \mathbf{F}^{\mathrm{init}}(s, \eta, t) 
    = (1 - \eta)\, \boldsymbol{\gamma}^{\mathrm{out}}(s) 
    + \eta\, \boldsymbol{\gamma}^{\mathrm{in}}(s, t),
    \qquad (s, \eta) \in [0, 1]^2,
    \label{eq:ruled-surface}
\end{equation}
which is then partitioned into $n_s$ conforming patches by subdividing the parameter domain along $s$. The barrier-patch procedure of \Cref{sec:barrier-patch-method} is applied patch-wise to improve element quality; boundary control points on $\boldsymbol{\gamma}^{\mathrm{in}}(\cdot, t)$ and $\boldsymbol{\gamma}^{\mathrm{out}}(\cdot)$ remain fixed, while interior control points are updated.

\paragraph{Tangential slip}
The ruled-surface construction \eqref{eq:ruled-surface} imposes a fixed parametric correspondence $s \mapsto s$ between inner and outer boundaries. Under prescribed rigid-body rotation of the inner body, this correspondence accumulates over time: a material point on the inner boundary rotates with the body, while the corresponding outer point remains stationary, producing patches that shear more and more with each step. After a finite number of steps, the patch maps become highly skewed and eventually non-bijective, and the barrier stage of \Cref{sec:barrier-patch-method} can no longer recover a valid parameterization \citep{ji2024boundary}.

We remove this accumulation by reparameterizing the inner boundary at each time step before constructing the ruled surface. Let $\delta^n \in [0, 1)$ denote a normalized tangential shift chosen (below) to align the inner boundary with the outer boundary. The shifted inner boundary is
\begin{equation}
    \tilde{\boldsymbol{\gamma}}^{\mathrm{in}}(s, t^n) 
    := \boldsymbol{\gamma}^{\mathrm{in}}\bigl((s + \delta^n) 
    \bmod 1,\; t^n\bigr), 
    \qquad s \in [0, 1],
    \label{eq:shifted-net}
\end{equation}
so that the physical curve is unchanged, while the location of the parametric seam is moved. In the implementation this is done by splitting the closed spline curve at the shifted parameter value and then merging the two pieces in the opposite order. The ruled surface \eqref{eq:ruled-surface} is then constructed with $\tilde{\boldsymbol{\gamma}}^{\mathrm{in}}$ instead of $\boldsymbol{\gamma}^{\mathrm{in}}$.

For the prescribed rigid-body rotations considered in this work, the shift is chosen directly from the cumulative rotation angle. Let $\theta^n$ be the rotation angle at time $t^n$. We use
\begin{equation}
    \delta^n =
    \left(-\frac{\theta^n}{2\pi}\right) \bmod 1 .
    \label{eq:shift-closed-form}
\end{equation}
The sign is chosen so that the parameterization slips in the direction opposite to the physical rotation. In this way, the correspondence between the inner and outer boundaries remains close to the initial one instead of accumulating tangential shear.

For piecewise-smooth inner boundaries with $n_c$ geometric corners distributed evenly in parameter, the corner parameters are
\begin{equation}
    \mathcal{C} = \left\{\, \tfrac{j}{n_c} \,:\, j = 0, 1, \ldots, n_c - 1 \,\right\}
    \label{eq:corner-parameters}
\end{equation}
($n_c = 4$ in the rotating-square and rotor examples of \Cref{sec:flow-past-square,sec:3d-rotor}). When $\delta^n$ falls close to an element of $\mathcal{C}$, the split-and-merge step cuts $\boldsymbol{\gamma}^{\mathrm{in}}(\cdot, t^n)$ very near a genuine corner, producing a reparameterized curve with a spurious $C^0$ point next to the real corner. We avoid this by displacing $\delta^n$ away from corners. Let $d^n$ denote the signed periodic distance from $\delta^n$ to the nearest corner parameter. When $|d^n| < \varepsilon_c$, we apply the smooth local offset
\begin{equation}
    \delta^n \leftarrow \delta^n + \varepsilon_s\left[1 - \left(\frac{d^n}{\varepsilon_c}\right)^2 \right]^2,
    \qquad
    \varepsilon_c = 0.04,\quad
    \varepsilon_s = 0.025,
    \label{eq:corner-bump}
\end{equation}
followed by reduction modulo $1$. The values satisfy $\varepsilon_c < 1/(2 n_c)$ and $\varepsilon_s < \varepsilon_c$, so that bump supports of adjacent corners do not overlap and the correction cannot push $\delta^n$ past the nearest corner. The bump $[1-(d^n/\varepsilon_c)^2]^2$ is the lowest-order polynomial that vanishes to first order at $|d^n| = \varepsilon_c$, ensuring that $\delta^n$ depends continuously on $\theta^n$.

As a final safeguard, when $\delta^n$ is within machine precision of $0$ or $1$, the split-and-merge step reduces to the identity up to round-off and is skipped. The reparameterization \eqref{eq:shifted-net}--\eqref{eq:corner-bump} enables the construction of well-shaped meshes under arbitrarily large cumulative tangential motion; this is demonstrated on the rotating-square benchmark in \Cref{sec:flow-past-square}.

\textit{Scope of the tangential-slip strategy.} The tangential-slip construction used here assumes that the closed interface admits a stable cyclic boundary correspondence. This is the case for convex or nearly convex bodies such as the rotating square and the rotor geometry considered later. For strongly non-convex interfaces, very slender structures, or sharp concave corners, a global tangential shift may lead to a non-uniform redistribution of the parameter pre-images and may no longer be sufficient to maintain patch quality. In such cases, we expect that the single global shift would have to be replaced by a localized redistribution of the pre-images (for instance weighted by local arc length and curvature, or applied per patch rather than globally), which is left to future work. The minimum scaled Jacobian $\mathcal{J}_{\min}$ monitored throughout the simulations provides a direct indicator of whether any such extension preserves bijectivity and mesh quality.

\subsection{ALE mesh velocity}
\label{sec:ale-mesh-velocity}

Because the patch maps are constructed anew at each time step, the mesh velocity is not defined by incremental node displacement but by the difference between the new and previous patch maps at the same parametric coordinate. For patch $k$, we define the parameter-space mesh velocity
\begin{equation}
    \hat{\mathbf{w}}_k^n(\boldsymbol{\xi}) 
    := \frac{\mathbf{F}_k^n(\boldsymbol{\xi}) 
    - \mathbf{F}_k^{n-1}(\boldsymbol{\xi})}
    {\Delta t},
    \qquad \boldsymbol{\xi} \in \mathcal{P}.
    \label{eq:mesh-velocity-parametric}
\end{equation}
In the NURBS representation $\mathbf{F}_k^{n}(\boldsymbol{\xi}) 
= \sum_{i=1}^{n_k} N_i(\boldsymbol{\xi})\, 
\mathbf{x}_{k,i}^{n}$, this reduces to
\begin{equation}
    \hat{\mathbf{w}}_k^n(\boldsymbol{\xi}) 
    = \sum_{i=1}^{n_k} N_i(\boldsymbol{\xi}) \, 
    \frac{\mathbf{x}_{k,i}^n - \mathbf{x}_{k,i}^{n-1}}{\Delta t},
\end{equation}
i.e., the NURBS interpolant of the control-point displacements divided by $\Delta t$.

The mesh velocity used in the ALE momentum equation \eqref{eq:ale-momentum} is the push-forward of $\hat{\mathbf{w}}_k^n$ to the current fluid domain $\Omega^f(t^n)$:
\begin{equation}
    \mathbf{w}^n(\mathbf{x}) 
    := \hat{\mathbf{w}}_k^n\!\bigl((\mathbf{F}_k^n)^{-1}
    (\mathbf{x})\bigr),
    \qquad \mathbf{x} \in \Omega_k(t^n).
    \label{eq:mesh-velocity-physical}
\end{equation}

In practice, the inverse $(\mathbf{F}_k^n)^{-1}$ is never formed explicitly: \eqref{eq:mesh-velocity-physical} is evaluated at quadrature points whose parametric coordinates $\boldsymbol{\xi}_q$ are known by construction, so that $\mathbf{w}^n(\mathbf{F}_k^n(\boldsymbol{\xi}_q)) = \hat{\mathbf{w}}_k^n(\boldsymbol{\xi}_q)$ is computed directly from the control-point increments.

\subsection{Solution transfer and the discrete geometric conservation law}
\label{sec:solution-transfer}

Since the geometric parameterizations are constructed at every time step, each fluid solve is posed on a new parameterization $\Omega_h^{f,n}$ that differs from the previous one $\Omega_h^{f,n-1}$. The discrete fluid fields, therefore, have to be transferred across meshes between steps. This transfer problem is distinct from the geometric consistency of a conservative ALE time discretization: the former concerns the representation of the old field on the new parameterization, whereas the latter concerns the discrete evolution of the geometric Jacobian and its consistency with the mesh flux.

\paragraph{Transfer operator}
For each patch $k = 1, \ldots, n_p$ and each point $\mathbf{x} \in \Omega_k(t^n)$, the composition
\begin{equation}
    \Phi_k^{n}(\mathbf{x}) 
    := \mathbf{F}_k^{n-1}\!\bigl((\mathbf{F}_k^n)^{-1}(\mathbf{x})\bigr)
    \in \Omega_k(t^{n-1})
    \label{eq:patch-pullback}
\end{equation}
identifies the point on the previous patch that shares the same parametric coordinate as $\mathbf{x}$. Although \eqref{eq:patch-pullback} is written in physical coordinates, the inverse map $(\mathbf{F}_k^n)^{-1}$ is not formed explicitly. Instead, the correspondence is evaluated directly in parameter space: for a point $\mathbf{x}^n=\mathbf{F}_k^n(\boldsymbol{\xi})$ on the current parameterization, the corresponding point on the previous parameterization is 
\begin{equation}
        \Phi_k^n(\mathbf{x}^n)=\mathbf{F}_k^{n-1}(\boldsymbol{\xi}).
        \label{eq:parametric-transfer}
\end{equation}
Thus, no point-wise Newton solve is required to evaluate the same-parametric-coordinate correspondence between consecutive parameterizations.

Given the discrete velocity $\mathbf{u}_h^{f,n-1}$ on $\Omega_h^{f,n-1}$, the transferred field on the new mesh is defined by the quasi-interpolation
\begin{equation}
    \tilde{\mathbf{u}}_h^{f,n} 
    := \mathcal{Q}_h^n\!\bigl(\mathbf{u}_h^{f,n-1} 
    \circ \Phi^{n}\bigr),
    \label{eq:solution-transfer}
\end{equation}
where $\Phi^n$ denotes the collection $\{\Phi_k^n\}_{k=1}^{n_p}$ applied patch-wise, and $\mathcal{Q}_h^n$ is a quasi-interpolation operator onto the discrete velocity space on $\Omega_h^{f,n}$. The pressure field is transferred in the same manner with $\mathcal{Q}_h^n$ replaced by its analogue on the pressure space.

\paragraph{Quasi-interpolation operator}
For a scalar function $f$ on $\Omega_h^{f,n}$, the quasi-interpolant takes the form
\begin{equation}
    \mathcal{Q}_h^n f 
    = \sum_{i=1}^{n_{\mathrm{dof}}} \lambda_i^n(f)\, N_i^n,
    \label{eq:quasi-interpolant}
\end{equation}
where $\{N_i^n\}$ is the global basis on $\Omega_h^{f,n}$ and each $\lambda_i^n$ is a local linear functional supported in a small neighborhood of the $i$-th control point. In this work we use de~Boor--Fix-type functionals \citep{buffa2016quasi}, for which $\mathcal{Q}_h^n$ is local, linear, and constant-preserving:
\begin{equation}
    \mathcal{Q}_h^n c = c 
    \qquad \text{for every constant } c \in \mathbb{R}.
    \label{eq:constant-preserving}
\end{equation}
Because each $\lambda_i^n$ depends only on the values of $f$ in a small neighborhood, the cost of \eqref{eq:quasi-interpolant} scales linearly with the number of degrees of freedom and is negligible compared with the cost of the fluid solve. For vector-valued fields, \eqref{eq:quasi-interpolant} is applied component-wise.

\paragraph{Discrete geometric conservation law}
The DGCL \citep{Farhat2001DGCL,Guillard2000GCL} is commonly expressed through free-stream preservation: a fully discrete ALE scheme should preserve a spatially uniform velocity field under arbitrary mesh motion, 
\begin{equation}
    \mathbf{u}_h^{f,0} = \mathbf{u}_0 \in \mathbb{R}^d
    \quad \Longrightarrow \quad
    \mathbf{u}_h^{f,n} = \mathbf{u}_0 
    \qquad \forall\, n \geq 1.
    \label{eq:dgcl-condition}
\end{equation}
Preserving \eqref{eq:dgcl-condition} is widely regarded as a minimal consistency requirement for ALE schemes on moving meshes and is often linked to their nonlinear stability \citep{Farhat2001DGCL,Formaggia2004DGCL}; see also \citep{Bazilevs2008IGAFSI} for a discussion in the IGA context.

\begin{proposition}
    Let the ALE momentum equation \eqref{eq:ale-momentum} be discretized in time by backward Euler \footnote{The proof extends with essentially no change to any consistent single-step integrator: for a $\theta$-method with $\theta \in (0, 1]$, all spatial terms in \eqref{eq:ale-momentum}--\eqref{eq:ale-continuity} vanish pointwise for a constant velocity field, and the velocity update reduces to $\mathbf{u}_h^{f,n} = \tilde{\mathbf{u}}_h^{f,n}$. In \Cref{sec:numerical-examples}, $\theta = 1$ is used for the perpendicular-flap, rotating-square, and three-dimensional rotor examples, and $\theta = 1/2$ (Crank--Nicolson) for the Turek--Hron FSI2 benchmark,} with mesh velocity $\mathbf{w}^n$ defined by \eqref{eq:mesh-velocity-physical}. In the absence of external forcing and with boundary data compatible with a constant state,  the per-step parameterization and transfer procedure of \Cref{sec:mesh-generation,sec:solution-transfer} preserves a spatially uniform velocity field in the sense of ~\eqref{eq:dgcl-condition}.
    \label{prop:dgcl}
\end{proposition}

\begin{proof}
Suppose $\mathbf{u}_h^{f,n-1} = \mathbf{u}_0$ on $\Omega_h^{f,n-1}$ for some spatially uniform $\mathbf{u}_0 \in \mathbb{R}^d$.

Since $\mathbf{u}_0$ is a constant field, its pullback by any map equals itself; in particular,
\begin{equation*}
    \mathbf{u}_0 \circ \Phi^{n} = \mathbf{u}_0 
    \qquad \text{on } \Omega_h^{f,n}.
\end{equation*}
Applying $\mathcal{Q}_h^n$ to both sides and using the constant-preserving property \eqref{eq:constant-preserving} component-wise yields
\begin{equation*}
    \tilde{\mathbf{u}}_h^{f,n} 
    = \mathcal{Q}_h^n(\mathbf{u}_0 \circ \Phi^{n}) 
    = \mathcal{Q}_h^n \mathbf{u}_0 
    = \mathbf{u}_0.
\end{equation*}

For a spatially uniform velocity field $\mathbf{u}_0$, the strain rate and the convective transport in \eqref{eq:ale-momentum} vanish identically:
\begin{equation*}
    \mathbb{D}(\mathbf{u}_0) = 0,
    \qquad
    \bigl((\mathbf{u}_0 - \mathbf{w}^n) \cdot \nabla\bigr)\mathbf{u}_0 
    = \mathbf{0},
    \qquad
    \nabla \cdot \mathbf{u}_0 = 0.
\end{equation*}
Consequently, the viscous term vanishes and the continuity equation is trivially satisfied; the pressure field adjusts to balance boundary data and body forces, but does not enter the velocity update. The backward Euler momentum balance therefore reduces to
\begin{equation*}
    \frac{\rho^f}{\Delta t}
    \bigl(\mathbf{u}_h^{f,n} - \tilde{\mathbf{u}}_h^{f,n}\bigr) 
    = \mathbf{0},
\end{equation*}
which gives $\mathbf{u}_h^{f,n} = \tilde{\mathbf{u}}_h^{f,n} = \mathbf{u}_0$. The claim follows by induction on $n$.
\end{proof}

\paragraph{Fully discrete matrix form}
Proposition~\ref{prop:dgcl} concerns the advective form of the convective term, which is the form used in all computations reported in \Cref{sec:numerical-examples}. In that form the constant-preserving transfer is by itself sufficient, because every spatial term vanishes pointwise on a uniform state and the parameterization Jacobian never appears. The discrete geometric conservation law is, however, classically stated for the conservative form, where the Jacobian enters the discrete time derivative explicitly and does not drop out; there a constant-preserving transfer alone does not establish free-stream preservation. We therefore give the corresponding matrix-level statement, both to show that the present construction is DGCL-consistent in this stronger sense and because it is what is required should the geometry module be combined with a conservative ALE discretization. At the fully discrete level, the time-dependent contribution is assembled as
\begin{equation}
    \frac{M^{n+1}\mathbf{U}^{n+1}
          -M^n\widetilde{\mathbf{U}}^n}{\Delta t}
    -C_{\mathrm{GCL}}^{n+\frac12}\mathbf{U}^{n+1},
    \label{eq:conservative-time-block}
\end{equation}
where $\widetilde{\mathbf{U}}^n$ is the coefficient vector transferred from the preceding parameterization. The matrices below are defined for a scalar basis and applied separately to each velocity component.

Let
\begin{equation*}
    \widehat{\Omega}:=\bigsqcup_{k=1}^{n_p}\mathcal P_k,
    \qquad
    J^m\big|_{\mathcal P_k}
    =J_k^m
    :=\det\bm{\mathcal J}_k^m
    =\det\!\left(\frac{\partial\mathbf F_k^m}
                        {\partial\boldsymbol\xi}\right),
    \qquad m\in\{n,n+1\},
\end{equation*}
denote the disjoint union of the patch parameter domains and the corresponding parameterization Jacobian determinants. The functions $\phi_i$ are the reference-domain analysis basis functions, understood patch-wise on $\widehat\Omega$. We define $C_{\mathrm{GCL}}^{n+\frac12}$ as the interval-averaged Jacobian-rate correction matrix over $[t^n,t^{n+1}]$:
\begin{equation}
    \bigl(M^m\bigr)_{ij}:=\int_{\widehat\Omega}J^m\phi_i\phi_j
    \,\mathrm d\boldsymbol\xi,
    \quad m\in\{n,n+1\},
    \qquad
    \bigl(C_{\mathrm{GCL}}^{n+\frac12}\bigr)_{ij}:=
    \int_{\widehat\Omega}\frac{J^{n+1}-J^n}{\Delta t}
    \phi_i\phi_j\,\mathrm d\boldsymbol\xi.
    \label{eq:mass-dgcl-matrix}
\end{equation}
For the linear-in-time parameterization between $t^n$ and $t^{n+1}$, the metric identity $\partial_tJ=J\nabla\!\cdot\mathbf w$ identifies the second integral with the time-averaged mesh-divergence contribution. Its evaluation through the endpoint Jacobians enforces that relation algebraically.

\begin{proposition}[Fully discrete DGCL identity]
Suppose that $M^n$, $M^{n+1}$, and $C_{\mathrm{GCL}}^{n+\frac12}$ are assembled with the same analysis basis and quadrature functional on the common knot refinement formed from both consecutive parameterizations. Then
\begin{equation}
    C_{\mathrm{GCL}}^{n+\frac12}
    =\frac{M^{n+1}-M^n}{\Delta t}.
    \label{eq:discrete-gcl-identity}
\end{equation}
Consequently, if the transfer operator reproduces constants as in \eqref{eq:constant-preserving}, then \eqref{eq:conservative-time-block} vanishes for $\widetilde{\mathbf U}^n=\mathbf U^{n+1}=\mathbf U_0$. Hence, the fully discrete update satisfies the DGCL condition \eqref{eq:dgcl-condition} for any pair of consecutive admissible parameterizations.
\label{prop:fully-discrete-dgcl}
\end{proposition}

\begin{proof}
Let $\mathscr{I}$ denote the numerical quadrature rule shared by the three matrices, using the same integration cells, quadrature points, and weights on the common knot refinement. Applying its linearity entrywise to \eqref{eq:mass-dgcl-matrix} gives
\begin{align}
    \bigl(C_{\mathrm{GCL}}^{n+\frac12}\bigr)_{ij}
    &=\frac{1}{\Delta t}\,
      \mathscr{I}\!\left[(J^{n+1}-J^n)\phi_i\phi_j\right] \notag\\
    &=\frac{1}{\Delta t}\left(
      \mathscr{I}\!\left[J^{n+1}\phi_i\phi_j\right]
      -\mathscr{I}\!\left[J^n\phi_i\phi_j\right]\right)
      =\frac{(M^{n+1})_{ij}-(M^n)_{ij}}{\Delta t},
\end{align}
which proves \eqref{eq:discrete-gcl-identity}. Substitution into \eqref{eq:conservative-time-block} yields
\begin{equation}
    \frac{M^n\bigl(\mathbf{U}^{n+1}
          -\widetilde{\mathbf{U}}^n\bigr)}{\Delta t}.
    \label{eq:reduced-time-block}
\end{equation}
For a constant free stream $\mathbf u_0$, let $\mathbf U_0$ denote its coefficient vector. The transfer property gives $\widetilde{\mathbf{U}}^n=\mathbf{U}_0$. Setting $\mathbf{U}^{n+1}=\mathbf{U}_0$ therefore makes \eqref{eq:reduced-time-block} vanish. Under compatible boundary data and zero forcing, the remaining spatial terms vanish as in Proposition~\ref{prop:dgcl}. The update therefore leaves $\mathbf U_0$ unchanged to machine precision.
\end{proof}

\begin{remark}
The proof uses the advective form of the convective term,
\begin{equation*}
    \bigl((\mathbf{u}^f - \mathbf{w}) \cdot \nabla\bigr)\mathbf{u}^f,
\end{equation*}
which vanishes identically for any spatially uniform 
$\mathbf{u}^f$. In the conservative form
\begin{equation*}
    \nabla \cdot \bigl((\mathbf{u}^f - \mathbf{w}) \otimes 
    \mathbf{u}^f\bigr),
\end{equation*}
a constant $\mathbf{u}^f$ leaves a residual proportional to $\nabla \cdot \mathbf{w}$. This term need not vanish separately; it is balanced by the temporal change of the parameterization Jacobian when the discrete identity \eqref{eq:discrete-gcl-identity} is satisfied.
\end{remark}

\begin{remark}
Proposition~\ref{prop:fully-discrete-dgcl} is algebraic and does not require exact integration of the time-dependent Jacobian fields. It requires the old mass matrix, new mass matrix, and Jacobian-rate correction matrix to use the same quadrature functional; the union integration basis enforces this common partition in the implementation.

The identity \eqref{eq:discrete-gcl-identity} should not be read as a tautology, since it holds for the particular discretization \eqref{eq:mass-dgcl-matrix} of the correction term and not for the discretization one would write down first. Assembling $C_{\mathrm{GCL}}^{n+\frac12}$ directly from the metric identity $\partial_t J = J\,\nabla\!\cdot\mathbf{w}$, that is, as $\int_{\widehat\Omega} J^{n+\frac12}\bigl(\nabla\!\cdot\mathbf{w}^{n+\frac12}\bigr)\phi_i\phi_j\,\mathrm d\boldsymbol\xi$ with the Jacobian and the mesh divergence evaluated at the midpoint, reproduces $(M^{n+1}-M^n)/\Delta t$ only up to a time-interpolation and quadrature error. The free stream would then be preserved to that accuracy rather than to machine precision, and the error would enter at every step. What makes the present construction exact is that the correction is assembled from the two endpoint Jacobians themselves, on the common knot refinement of the two parameterizations and with the same quadrature functional as the two mass matrices, so that the cancellation takes place at the level of the quadrature sums rather than of the integrals they approximate. The requirement is therefore a genuine constraint on the implementation, and it is the reason the union integration basis is used.
\end{remark}

\begin{remark}
The quasi-interpolation operator reproduces constants exactly~\eqref{eq:constant-preserving}; hence a spatially uniform state is transferred without error. This establishes free-stream preservation for the adopted advective update, but does not make the transfer conservative for integral quantities such as mass or momentum.
\end{remark}

\begin{figure}[H]
\centering
\resizebox{\linewidth}{!}{%
\begin{tikzpicture}[>=Stealth,
    panel img/.style={anchor=south west, inner sep=0pt, outer sep=0pt},
    flow arrow/.style={draw=black!55, line width=0.8pt, ->, >=Stealth},
    panel title/.style={font=\small\bfseries, black!85},
    panel caption/.style={font=\footnotesize, black!80},
    step note/.style={font=\scriptsize, black!60, text width=2.6cm, align=center}
  ]
 
\def\ps{2.2}
\def\hsep{1.4}
\def\xA{0}     \def\xB{3.6}   \def\xC{7.2}
\def\xD{10.8}  \def\xE{14.4}

\begin{scope}[shift={(\xA,0)}]
  \node[panel img] at (0,0)
    {\includegraphics[width=\ps cm, height=\ps cm]{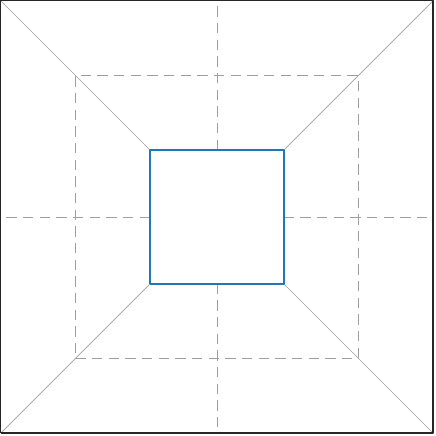}};
  \node[panel title, above] at (1.1,\ps+0.15) {Input};
  \node[panel caption, below] at (1.1,-0.3)
    {$(\mathbf{u}_h^{f,n\!-\!1},p_h^{f,n\!-\!1})$ on $\Omega_h^{f,n\!-\!1}$};
\end{scope}

\draw[flow arrow] (\xA+\ps+0.15,1.1)--
  node[above,font=\scriptsize, black!70]{Step 1} (\xB-0.15,1.1);
 
\begin{scope}[shift={(\xB,0)}]
  \node[panel img] at (0,0)
    {\includegraphics[width=\ps cm, height=\ps cm]{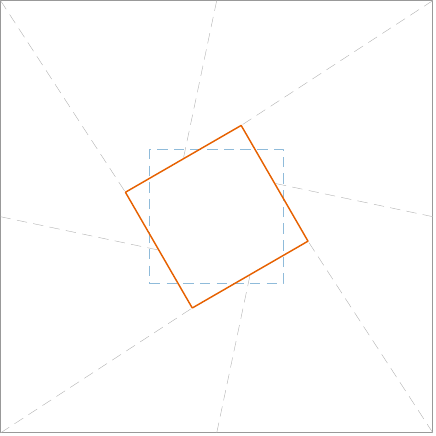}};
  \node[panel title, above] at (1.1,\ps+0.15) {Steps 1--2};
  \node[panel caption, below] at (1.1,-0.3) {Boundary curves};
\end{scope}
 
\draw[flow arrow] (\xB+\ps+0.15,1.1)--
  node[above,font=\scriptsize, black!70]{Step 3} (\xC-0.15,1.1);
 
\begin{scope}[shift={(\xC,0)}]
  \node[panel img] at (0,0)
    {\includegraphics[width=\ps cm, height=\ps cm]{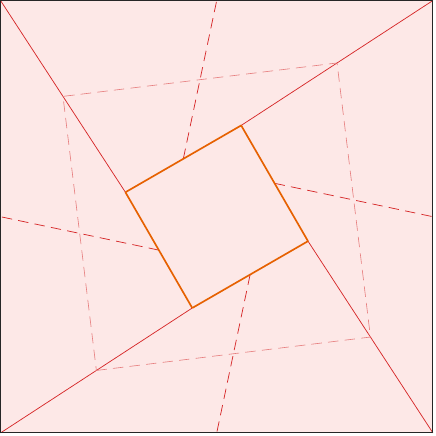}};
  \node[panel title, above] at (1.1,\ps+0.15) {Step 3};
  \node[panel caption, below] at (1.1,-0.3) {Ruled surface};
\end{scope}
 
\draw[flow arrow] (\xC+\ps+0.15,1.1)--
  node[above,font=\scriptsize, black!70]{Step 4} (\xD-0.15,1.1);
 
\begin{scope}[shift={(\xD,0)}]
  \node[panel img] at (0,0)
    {\includegraphics[width=\ps cm, height=\ps cm]{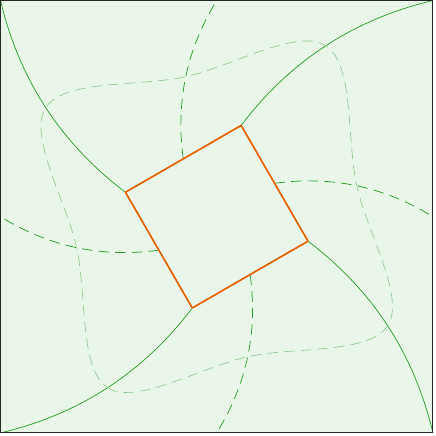}};
  \node[panel title, above] at (1.1,\ps+0.15) {Step 4};
  \node[panel caption, below] at (1.1,-0.3) {$\Omega_h^{f,n}$ (optimized)};
\end{scope}
 
\draw[flow arrow] (\xD+\ps+0.15,1.1)--
  node[above,font=\scriptsize, black!70]{Steps 5--7} (\xE-0.15,1.1);
 
\begin{scope}[shift={(\xE,0)}]
  \node[panel img] at (0,0)
    {\includegraphics[width=\ps cm, height=\ps cm]{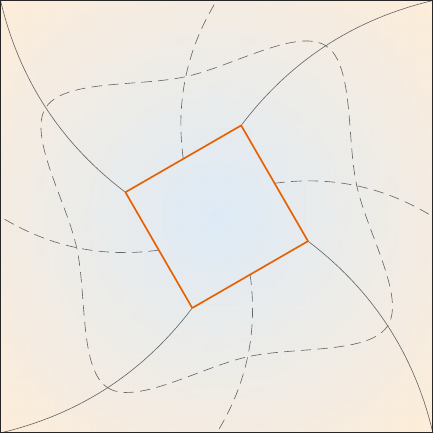}};
  \node[panel title, above] at (1.1,\ps+0.15) {Steps 5--7};
  \node[panel caption, below] at (1.1,-0.3)
    {$(\mathbf{u}_h^{f,n},p_h^{f,n})$ on $\Omega_h^{f,n}$};
\end{scope}
 
\node[step note] at (\xB+1.1,-1.45)
  {from structural solver or kinematics;\\
   tangential slip for closed domains};
\node[step note] at (\xC+1.1,-1.45) {Eq.~\eqref{eq:ruled-surface}};
\node[step note] at (\xD+1.1,-1.45)
  {Eqs.~\eqref{eq:barrier-foldover}--\eqref{eq:barrier-quality}};
\node[step note] at (\xE+1.1,-1.45)
  {$\mathcal{Q}_h^n$ \eqref{eq:solution-transfer},
   $\mathbf{w}^n$ \eqref{eq:mesh-velocity-physical},
   solve \eqref{eq:ale-momentum}--\eqref{eq:ale-continuity}};
 
\end{tikzpicture}%
}
 
\caption{%
  Schematic of one time step of Algorithm~\ref{alg:timestep}, illustrated on the rotating-square geometry. \textbf{Input:} solution at $t^{n-1}$ on the previous mesh $\Omega_h^{f,n-1}$.
  \textbf{Steps 1--2:} new interface curves obtained from the structural solver or prescribed kinematics, after tangential reparameterization (closed domains).
  \textbf{Step 3:} initial ruled-surface parameterization before quality optimization.
  \textbf{Step 4:} optimized parameterization $\Omega_h^{f,n}$ after the barrier-patch optimization.
  \textbf{Steps 5--7:} the solution on the optimized mesh, obtained by solution transfer, mesh-velocity evaluation, and the fluid solve.}
\label{fig:ale-schematic}
\end{figure}

The complete time-stepping procedure of the proposed framework is summarized in Algorithm~\ref{alg:timestep}; an illustration on the rotating-square geometry is given in Figure~\ref{fig:ale-schematic}.

\begin{algorithm}[H]
\caption{One time step of the parameterization-driven ALE procedure.}
\label{alg:timestep}
\KwIn{$(\mathbf{u}_h^{f,n-1}, p_h^{f,n-1})$ on $\Omega_h^{f,n-1}$}
\KwOut{$(\mathbf{u}_h^{f,n}, p_h^{f,n})$ on $\Omega_h^{f,n}$}

Obtain boundary curves $\{\boldsymbol{\gamma}_k^{\mathrm{in}}(\cdot,t^n)\}$ from prescribed kinematics or from the structural solver\;

\If{closed-domain configuration}{
    Apply the tangential reparameterization \eqref{eq:shifted-net} to shift the inner boundary\label{algline:tangslip}\;
}

Construct the initial patch mappings $\{\mathbf{F}_k^n\}$ from the boundary curves: ruled surface \eqref{eq:ruled-surface} for closed domains, transfinite interpolation from $\{\boldsymbol{\gamma}_k^{\mathrm{in}}, \boldsymbol{\gamma}_k^{\mathrm{out}}\}$ for open domains (\Cref{sec:open-domain})\label{algline:ruled}\;

Refine the interior control points by minimizing the barrier-patch functionals \eqref{eq:barrier-foldover}--\eqref{eq:barrier-quality}\label{algline:barrier}\;

Compute $(\tilde{\mathbf{u}}_h^{f,n}, \tilde{p}_h^{f,n})$ using the transfer operator $\mathcal{Q}_h^n$ in \eqref{eq:solution-transfer}\label{algline:transfer}\;

Evaluate the mesh velocity $\mathbf{w}^n$ from \eqref{eq:mesh-velocity-physical}\label{algline:meshvel}\;

Solve \eqref{eq:ale-momentum}--\eqref{eq:ale-continuity} on $\Omega_h^{f,n}$ to obtain $(\mathbf{u}_h^{f,n}, p_h^{f,n})$\label{algline:fluid}\;
\end{algorithm}

\subsection{Bijectivity Criterion}
\label{sec:bijectivity-criterion}

We now specify the discrete condition used to enforce the bijectivity requirement \eqref{eq:bijectivity-continuous} on the patch maps $\{\mathbf{F}_k^n\}$ produced at each time step by the barrier-patch procedure of \Cref{sec:barrier-patch-method}.

The patch-wise Jacobian matrix
\begin{equation}
    \bm{\mathcal{J}}_k^n(\boldsymbol{\xi}) 
    := \frac{\partial \mathbf{F}_k^n}{\partial \boldsymbol{\xi}}(\boldsymbol{\xi})
    \label{eq:jacobian-matrix}
\end{equation}
is the discrete, time-indexed realization of the ALE Jacobian $\bm{\mathcal{J}}^{\mathrm{ALE}}$ introduced in \eqref{eq:jacobian-ale}. A standard sufficient condition for local invertibility of $\mathbf{F}_k^n$ is that
\begin{equation}
    \det \bm{\mathcal{J}}_k^n(\boldsymbol{\xi}) > 0 
    \quad \text{for all } \boldsymbol{\xi} \in \mathcal{P}, 
    \; k = 1, \ldots, n_p.
    \label{eq:bijectivity-discrete}
\end{equation}
A vanishing determinant signals a local collapse of the parameterization; a sign change corresponds to a fold-over and hence to loss of injectivity.

In practice, \eqref{eq:bijectivity-discrete} is monitored through the minimum scaled Jacobian
\begin{equation}
    \mathcal{J}_{\min}(t^n) 
    := \min_{1 \leq k \leq n_p} \; 
    \min_{\boldsymbol{\xi} \in \mathcal{S}}
    \frac{\det \bm{\mathcal{J}}_k^n(\boldsymbol{\xi})}
    {\prod_{a=1}^{d} \bigl\| \partial_{\xi_a} \mathbf{F}_k^n(\boldsymbol{\xi}) \bigr\|},
    \label{eq:jmin-monitor}
\end{equation}
where $\mathcal{S} \subset \mathcal{P}$ is the set of Gauss quadrature points used in the element-wise integration. By Hadamard's inequality, $\mathcal{J}_{\min}(t^n) \in [-1, 1]$, with $\mathcal{J}_{\min}(t^n) = 1$ attained if and only if the columns of $\bm{\mathcal{J}}_k^n$ are pairwise orthogonal at every sampling point (i.e., the parameterization is locally orthogonal in the $\xi_a$-directions). A parameterization is considered admissible whenever $\mathcal{J}_{\min}(t^n) > 0$; values close to zero indicate proximity to degeneracy. The indicator $\mathcal{J}_{\min}$ is tracked throughout the numerical experiments of \Cref{sec:numerical-examples}.

\section{Numerical experiments}
\label{sec:numerical-examples}

This section assesses the parameterization-driven ALE framework through six studies spanning kinematic parameterization tests, coupled FSI benchmarks, and three-dimensional large-motion configurations.


The first case (\Cref{sec:mesh-quality}) is purely kinematic: the fluid
parameterization is driven by prescribed rigid-body rotation of an embedded
square, with no flow solver involved. This isolates the parameterization
itself and quantifies the range of cumulative rotation angles for which the
barrier-patch construction remains valid. This test has two purposes. First,
it provides a stress test for the parameterization method without tangential slip. Second, the comparison between the higher-order spline result and the
first-order case $p=1$ shows that the higher degree and inter-element continuity brought by the spline representation help preserve parameterization validity over a larger admissible rotation range
under large accumulated deformation.

The perpendicular-flap benchmark from the preCICE tutorial suite (\Cref{sec:perpendicular-flap}) verifies the response in a moderate-deformation two-way coupled setting, while the Turek--Hron FSI2 benchmark (\Cref{sec:turek-hron-fsi2}) tests sustained large-amplitude periodic deformation. The sphere-in-channel study (\Cref{sec:falling-sphere}) extends the assessment to three-dimensional large translation: its prescribed-motion part compares mesh quality with Di Cristofaro et al.~\citep{DiCristofaro2025ALEFlying}, whereas its free-fall part validates the coupled velocity response against published data \citep{Casquero2015NURBSImmersed,Black2026ImmersedFSI}. The flow past a rotating square (\Cref{sec:flow-past-square}) then returns to the geometry of the first case under continuous large-angle rotation and uses both isogeometric and finite-element flow solvers to demonstrate solver portability. Finally, the three-dimensional rotor tests the volumetric construction under sustained rotation. Thus, the perpendicular-flap, Turek--Hron, and free-fall cases provide dynamical comparisons with external data, while the prescribed sphere, rotating square, and rotor isolate robustness under large translation or rotation.

The fluid is discretized with Taylor--Hood-type isogeometric spaces (\Cref{sec:iga}) on the multi-patch geometries constructed in \Cref{sec:mesh-generation}. All fluid solves are performed in G+Smo \citep{Juettler2014Gismo, Gismo2025Zenodo}; the rotating-square case additionally uses \textsc{deal.II} \citep{dealII95} as an external finite-element solver. The three two-way coupled benchmarks (perpendicular flap, Turek--Hron FSI2, and free fall) are coupled in a partitioned manner through the preCICE library \citep{Chourdakis2021}.

\subsection{Admissible rotation range of the mesh parameterization}
\label{sec:mesh-quality}

Our first example isolates the parameterization itself, independent of any flow or FSI coupling, and quantifies the range of cumulative rotation angles for which it remains valid. We consider a square of side length $L$ centered in a $3L \times 3L$ outer box; the annular region is partitioned into four body-fitted patches, matching the inner block of the fluid domain used in \Cref{sec:flow-past-square}. The square is rotated incrementally about its center, and each mesh-update method is applied step by step until either $\mathcal{J}_{\min}$ drops to zero (parameterization failure) or a target cumulative rotation is reached. The indicator of interest is $\mathcal{J}_{\min}(\theta)$ from \eqref{eq:jmin-monitor}, tracked as a function of the rotation angle $\theta$.

\begin{figure}[H]
    \centering

    \resizebox{\textwidth}{!}{%
    \begin{tikzpicture}[
        image/.style={
            inner sep=0pt,
            outer sep=0pt
        },
        lab/.style={
            font=\footnotesize,
            align=center,
            text width=3.0cm,
            anchor=north,
            fill=black!2,
            draw=black!16,
            rounded corners=2pt,
            inner xsep=3pt,
            inner ysep=3pt
        },
        biglab/.style={
            font=\footnotesize,
            align=center,
            text width=4.2cm,
            anchor=north,
            fill=ForestGreen!5,
            draw=ForestGreen!45!black,
            rounded corners=2pt,
            inner xsep=4pt,
            inner ysep=4pt
        },
        group title/.style={
            font=\small\bfseries,
            align=center,
            text=black!80
        },
        chip/.style={
            font=\scriptsize\bfseries,
            inner xsep=3pt,
            inner ysep=1.5pt,
            rounded corners=1.5pt,
            fill=black!75,
            text=white,
            anchor=north east
        },
        chip green/.style={
            chip,
            fill=ForestGreen!65!black
        },
        chip red/.style={
            chip,
            fill=red!75!black,
            text=white
        }
    ]

    \def\smallW{3.0cm}
    \def\bpW{3.4cm}

    \def\xA{0}
    \def\xB{3.55}
    \def\xC{7.10}
    \def\xSep{9.25}
    \def\xD{12.55}

    \def\yTop{3.25}
    \def\yBot{-1.85}

    \def\yBPa{3.85}
    \def\yBPb{-0.05}
    \def\yBPc{-3.95}

    \node[group title] at (3.55,5.30)
        {Classical mesh-update strategies $p=2$};

    \node[group title] at (\xD,6.45)
        {Parameterization-driven strategy $p=2$};

    \draw[black!18, line width=0.5pt]
        (-1.55,5.05) -- (8.65,5.05);

    \draw[ForestGreen!35!black, line width=0.5pt]
        (10.60,6.20) -- (14.50,6.20);

    \node[image] (HE) at (\xA,\yTop)
        {\includegraphics[width=\smallW]{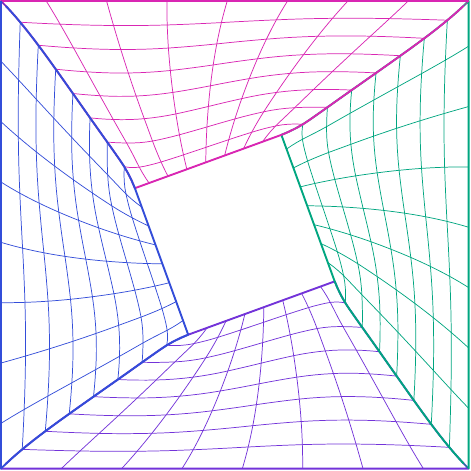}};
    \node[chip] at ([xshift=-2pt,yshift=-2pt]HE.north east)
        {$20.7^\circ$};
    \node[lab, yshift=-1.4mm] at (HE.south)
        {{\bfseries Harmonic Extension}~\citep{Johnson1994HE}\\[-0.5mm]
        {\color{black!75}$\theta_{\max}=20.7^\circ$}};

    \node[image] (IHE) at (\xB,\yTop)
        {\includegraphics[width=\smallW]{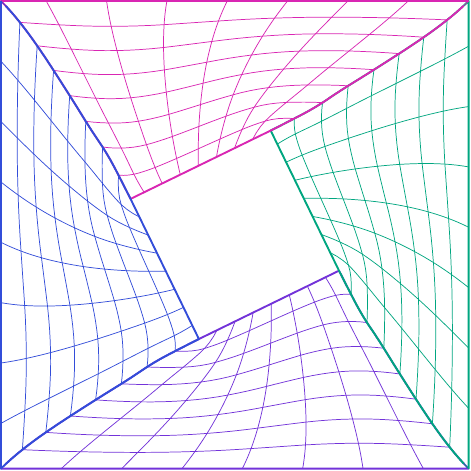}};
    \node[chip] at ([xshift=-2pt,yshift=-2pt]IHE.north east)
        {$26.8^\circ$};
    \node[lab, yshift=-1.4mm] at (IHE.south)
        {{\bfseries Incremental Harmonic}\\
        {\bfseries Extension}~\citep{Wick2011IHE}\\[-0.5mm]
        {\color{black!75}$\theta_{\max}=26.8^\circ$}};

    \node[image] (LE) at (\xC,\yTop)
        {\includegraphics[width=\smallW]{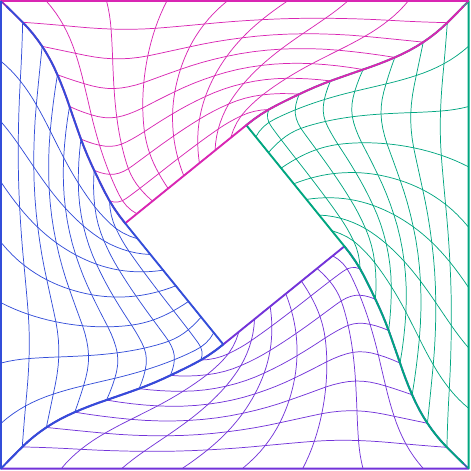}};
    \node[chip] at ([xshift=-2pt,yshift=-2pt]LE.north east)
        {$39.1^\circ$};
    \node[lab, yshift=-1.4mm] at (LE.south)
        {{\bfseries Linear Elasticity}~\citep{Tezduyar1992LE}\\[-0.5mm]
        {\color{black!75}$\theta_{\max}=39.1^\circ$}};

    \node[image] (ILE) at (\xA,\yBot)
        {\includegraphics[width=\smallW]{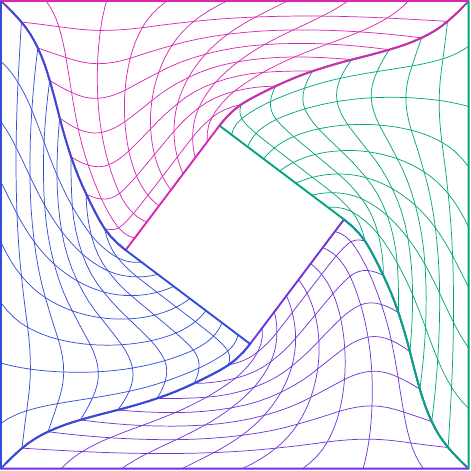}};
    \node[chip] at ([xshift=-2pt,yshift=-2pt]ILE.north east)
        {$53.2^\circ$};
    \node[lab, yshift=-1.4mm] at (ILE.south)
        {{\bfseries Incremental Linear}\\
        {\bfseries Elasticity}~\citep{stein2002advanced}\\[-0.5mm]
        {\color{black!75}$\theta_{\max}=53.2^\circ$}};

    \node[image] (TINE) at (\xB,\yBot)
        {\includegraphics[width=\smallW]{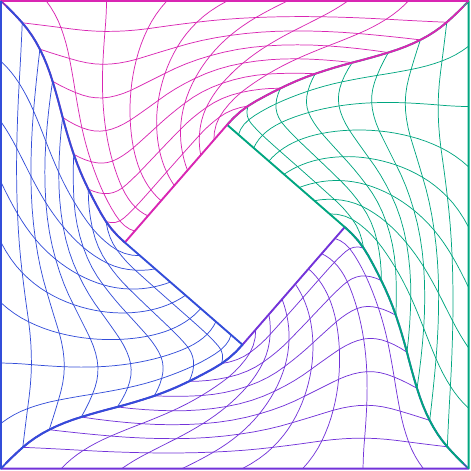}};
    \node[chip] at ([xshift=-2pt,yshift=-2pt]TINE.north east)
        {$49.8^\circ$};
    \node[lab, yshift=-1.4mm] at (TINE.south)
        {{\bfseries Tangential Incremental}\\
        {\bfseries Nonlinear Elasticity}~\citep{Shamanskiy2021MeshMoving}\\[-0.5mm]
        {\color{black!75}$\theta_{\max}=49.8^\circ$}};

    \node[image] (BHE) at (\xC,\yBot)
        {\includegraphics[width=\smallW]{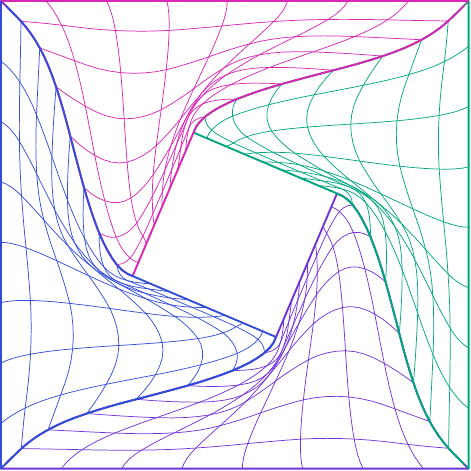}};
    \node[chip] at ([xshift=-2pt,yshift=-2pt]BHE.north east)
        {$67.1^\circ$};
    \node[lab, yshift=-1.4mm] at (BHE.south)
        {{\bfseries Bi-harmonic Extension}~\citep{Helenbrook2003BHE}\\[-0.5mm]
        {\color{black!75}$\theta_{\max}=67.1^\circ$}};

    \draw[gray!50, line width=0.5pt]
        (\xSep,-6.60) -- (\xSep,6.70);

    \node[image] (BPa) at (\xD,\yBPa)
        {\includegraphics[width=\bpW]{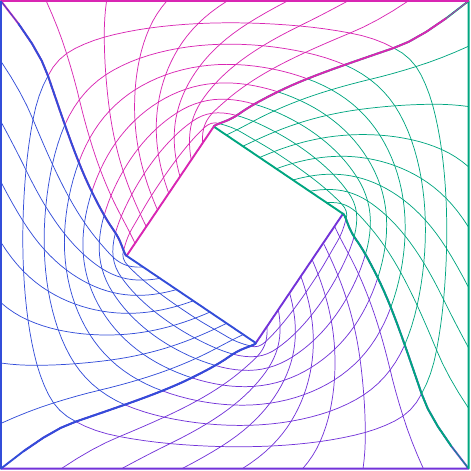}};
    \node[chip green] at ([xshift=-2pt,yshift=-2pt]BPa.north east)
        {$45^\circ$};

    \node[image] (BPb) at (\xD,\yBPb)
        {\includegraphics[width=\bpW]{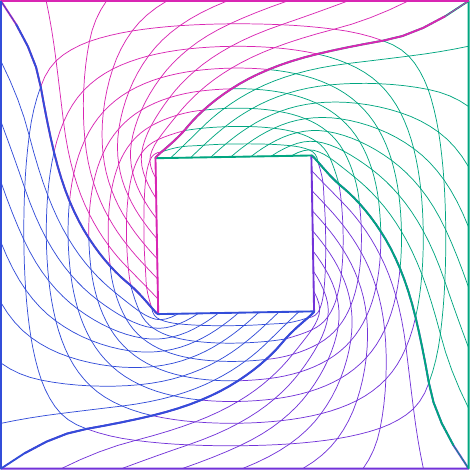}};
    \node[chip green] at ([xshift=-2pt,yshift=-2pt]BPb.north east)
        {$90^\circ$};

    \node[image] (BPc) at (\xD,\yBPc)
        {\includegraphics[width=\bpW]{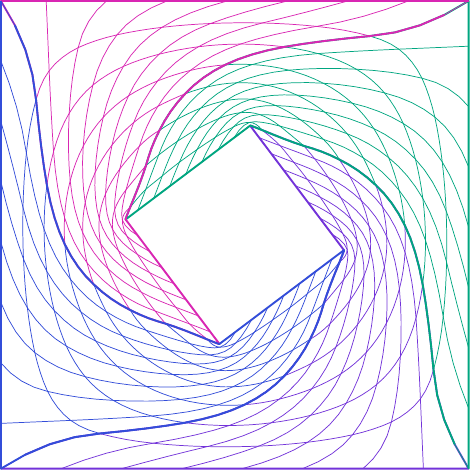}};
    \node[chip red] at ([xshift=-2pt,yshift=-2pt]BPc.north east)
        {{\color{red!8!white}$127.5^\circ$}};

    \draw[ForestGreen!55!black, line width=0.9pt, ->, >=Stealth]
        ($(BPa.south west) + (-0.55, 0.1)$)
        --
        ($(BPc.north west) + (-0.55, -0.1)$);

    \node[font=\scriptsize\itshape, text=ForestGreen!50!black,
          rotate=90, anchor=south]
        at ($(BPb.west) + (-0.65, 0)$)
        {increasing rotation};

    \node[biglab, yshift=-2.2mm] at (BPc.south)
        {{\bfseries Barrier Patch Method}\\[-0.5mm]
        {\color{red!75!black}$\theta_{\max}=127.5^\circ$}};

    \end{tikzpicture}%
    }

    \caption{Maximum admissible rotation angles for different mesh-motion strategies. Each panel in the left block shows the parameterization at the maximum angle the corresponding classical method can handle before losing validity, with the angle indicated in the upper-right chip. The right column shows the parameterization-driven Barrier Patch Method at three rotation angles ($45^\circ$, $90^\circ$, and its admissible limit {\color{red!75!black}$127.5^\circ$}), demonstrating that the parameterization remains valid throughout the rotation and reaches a substantially larger maximum angle than any classical scheme.}
    \label{fig:max-angle-comparison}
\end{figure}

\paragraph{Maximum sustainable rotation angle}
We compare the proposed barrier-patch (BP) parameterization with six classical mesh-update strategies: harmonic extension (HE), incremental harmonic extension (IHE), linear elasticity (LE), incremental linear elasticity (ILE), tangential incremental nonlinear elasticity (TINE), and bi-harmonic extension (BHE); detailed formulations are given in \citep{Shamanskiy2021MeshMoving}. The maximum rotation angle sustained by each method is reported in \Cref{fig:max-angle-comparison}. All six classical methods lose validity well before $90^\circ$; the best performer among them, bi-harmonic extension, fails at $67.1^\circ$. The Barrier Patch Method remains valid up to $127.5^\circ$.

The full benefit of the proposed approach relies on the spline parameterization, and in particular on its degree and inter-element continuity, which together yield smoother Jacobian fields; these are essential when the domain undergoes large accumulated deformation. To quantify this effect, we repeated the stress test with a first-order ($p=1$) parameterization, which lacks the inter-element smoothness the method is designed to exploit, while keeping the regeneration algorithm, the patch layout, and the refinement level unchanged. As shown in \Cref{fig:p1-method-stress-comparison}, the maximum admissible rotation angle drops to $31.35^\circ$, far below the $127.5^\circ$ reached by the higher-order barrier-patch parameterization in \Cref{fig:max-angle-comparison}. Higher-order spline parameterizations thus provide a substantially larger resistance to large mesh deformations; all subsequent examples therefore employ parameterizations of degree $p \geq 2$.

\begin{figure}[H]
    \centering

    \resizebox{\textwidth}{!}{%
    \begin{tikzpicture}[
        image/.style={
            inner sep=0pt,
            outer sep=0pt
        },
        lab/.style={
            font=\footnotesize,
            align=center,
            text width=3.0cm,
            anchor=north,
            fill=black!2,
            draw=black!16,
            rounded corners=2pt,
            inner xsep=3pt,
            inner ysep=3pt
        },
        biglab/.style={
            font=\footnotesize,
            align=center,
            text width=4.2cm,
            anchor=north,
            fill=ForestGreen!5,
            draw=ForestGreen!45!black,
            rounded corners=2pt,
            inner xsep=4pt,
            inner ysep=4pt
        },
        group title/.style={
            font=\small\bfseries,
            align=center,
            text=black!80
        },
        chip/.style={
            font=\scriptsize\bfseries,
            inner xsep=3pt,
            inner ysep=1.5pt,
            rounded corners=1.5pt,
            fill=black!75,
            text=white,
            anchor=north east
        },
        chip green/.style={
            chip,
            fill=ForestGreen!65!black
        }
    ]

    \def\smallW{3.0cm}
    \def\bpW{3.4cm}

    \def\xA{0}
    \def\xB{3.55}
    \def\xC{7.10}
    \def\xSep{9.25}
    \def\xD{12.55}

    \def\yTop{3.25}
    \def\yBot{-1.85}

    \def\yBPa{3.85}
    \def\yBPb{-0.05}
    \def\yBPc{-3.95}

    \node[group title] at (3.55,5.30)
        {Classical mesh-update strategies $p=1$};

    \node[group title] at (\xD,6.45)
        {Parameterization-driven strategy $p=1$};

    \draw[black!18, line width=0.5pt]
        (-1.55,5.05) -- (8.65,5.05);

    \draw[ForestGreen!35!black, line width=0.5pt]
        (10.60,6.20) -- (14.50,6.20);

    \node[image] (HE) at (\xA,\yTop)
        {\includegraphics[width=\smallW,trim=23 23 23 23,clip]{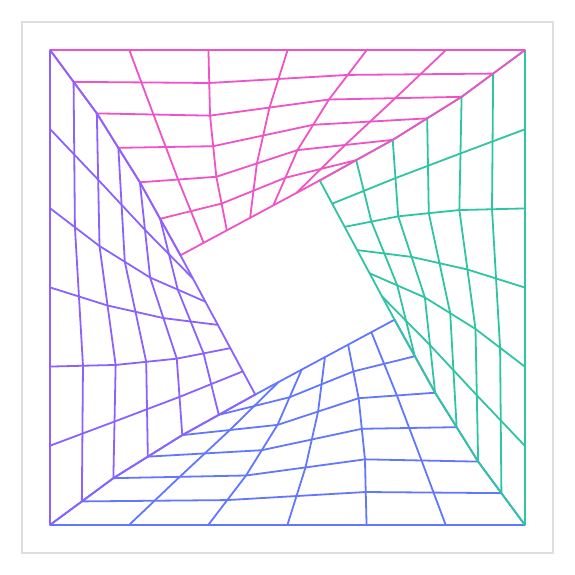}};
    \node[chip] at ([xshift=-2pt,yshift=-2pt]HE.north east)
        {$28.25^\circ$};
    \node[lab, yshift=-1.4mm] at (HE.south)
        {{\bfseries Harmonic Extension}\\[-0.5mm]
        {\color{black!75}$\theta_{\max}=28.25^\circ$}};

    \node[image] (IHE) at (\xB,\yTop)
        {\includegraphics[width=\smallW,trim=23 23 23 23,clip]{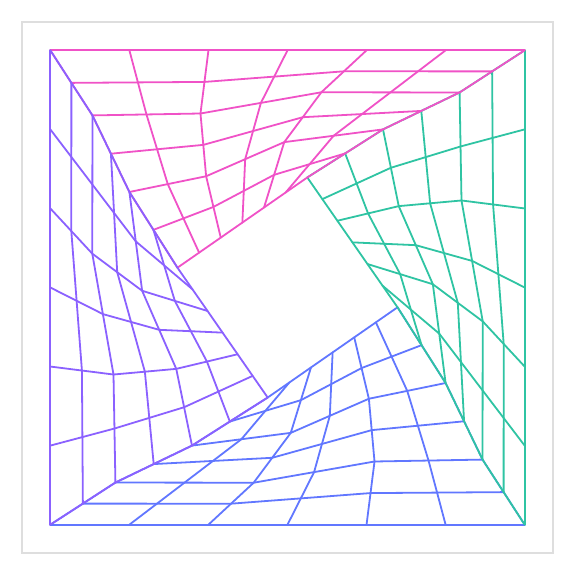}};
    \node[chip] at ([xshift=-2pt,yshift=-2pt]IHE.north east)
        {$34.80^\circ$};
    \node[lab, yshift=-1.4mm] at (IHE.south)
        {{\bfseries Incremental Harmonic}\\
        {\bfseries Extension}\\[-0.5mm]
        {\color{black!75}$\theta_{\max}=34.80^\circ$}};

    \node[image] (LE) at (\xC,\yTop)
        {\includegraphics[width=\smallW,trim=23 23 23 23,clip]{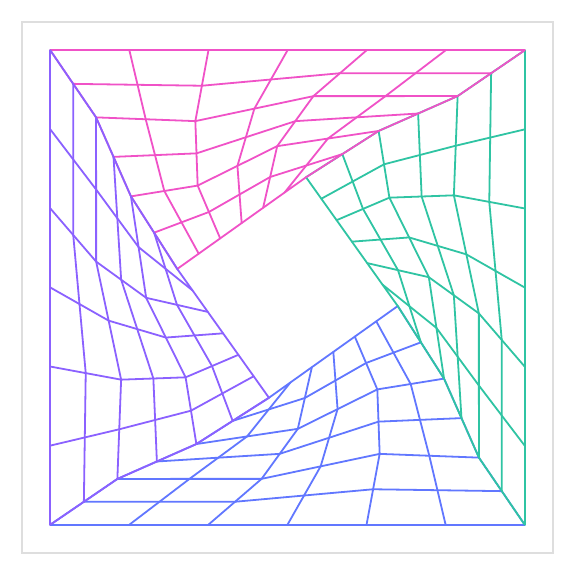}};
    \node[chip] at ([xshift=-2pt,yshift=-2pt]LE.north east)
        {$35.45^\circ$};
    \node[lab, yshift=-1.4mm] at (LE.south)
        {{\bfseries Linear Elasticity}\\[-0.5mm]
        {\color{black!75}$\theta_{\max}=35.45^\circ$}};

    \node[image] (ILE) at (\xA,\yBot)
        {\includegraphics[width=\smallW,trim=23 23 23 23,clip]{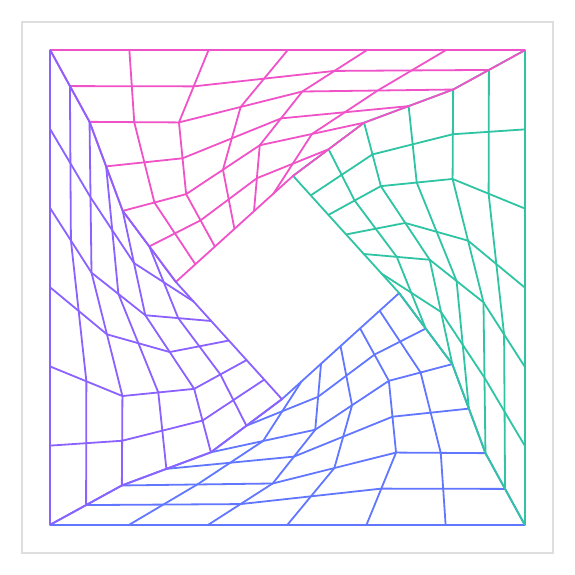}};
    \node[chip] at ([xshift=-2pt,yshift=-2pt]ILE.north east)
        {$42.15^\circ$};
    \node[lab, yshift=-1.4mm] at (ILE.south)
        {{\bfseries Incremental Linear}\\
        {\bfseries Elasticity}\\[-0.5mm]
        {\color{black!75}$\theta_{\max}=42.15^\circ$}};

    \node[image] (TINE) at (\xB,\yBot)
        {\includegraphics[width=\smallW,trim=23 23 23 23,clip]{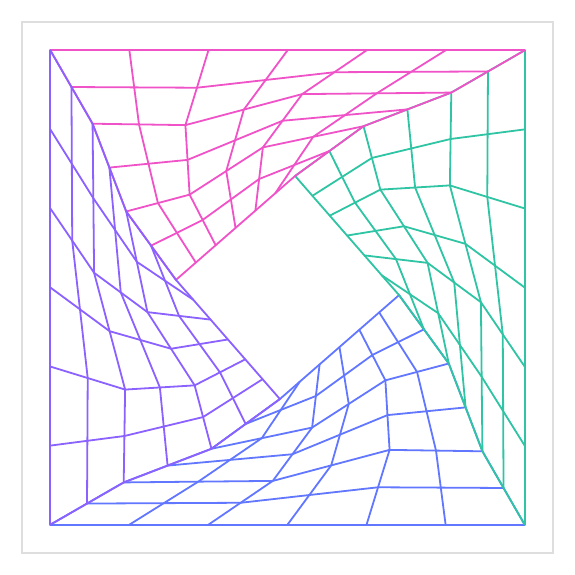}};
    \node[chip] at ([xshift=-2pt,yshift=-2pt]TINE.north east)
        {$41.10^\circ$};
    \node[lab, yshift=-1.4mm] at (TINE.south)
        {{\bfseries Tangential Incremental}\\
        {\bfseries Nonlinear Elasticity}\\[-0.5mm]
        {\color{black!75}$\theta_{\max}=41.10^\circ$}};

    \node[image] (BHE) at (\xC,\yBot)
        {\includegraphics[width=\smallW,trim=23 23 23 23,clip]{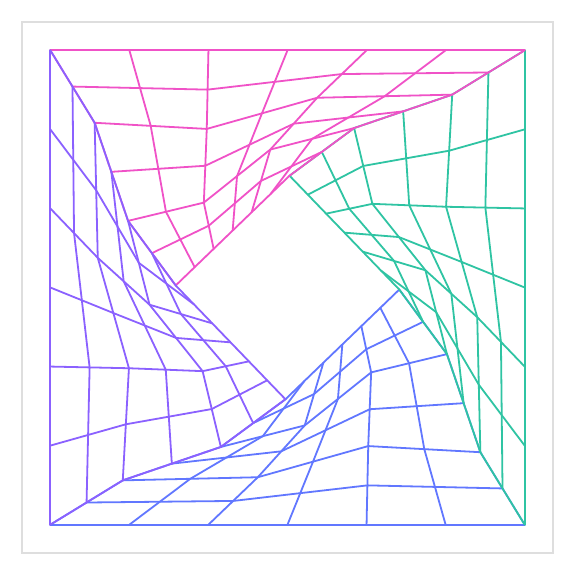}};
    \node[chip] at ([xshift=-2pt,yshift=-2pt]BHE.north east)
        {$43.95^\circ$};
    \node[lab, yshift=-1.4mm] at (BHE.south)
        {{\bfseries Bi-harmonic Extension}\\[-0.5mm]
        {\color{black!75}$\theta_{\max}=43.95^\circ$}};

    \draw[gray!50, line width=0.5pt]
        (\xSep,-6.60) -- (\xSep,6.70);

    \node[image] (BPa) at (\xD,\yBPa)
        {\includegraphics[width=\bpW,trim=23 23 23 23,clip]{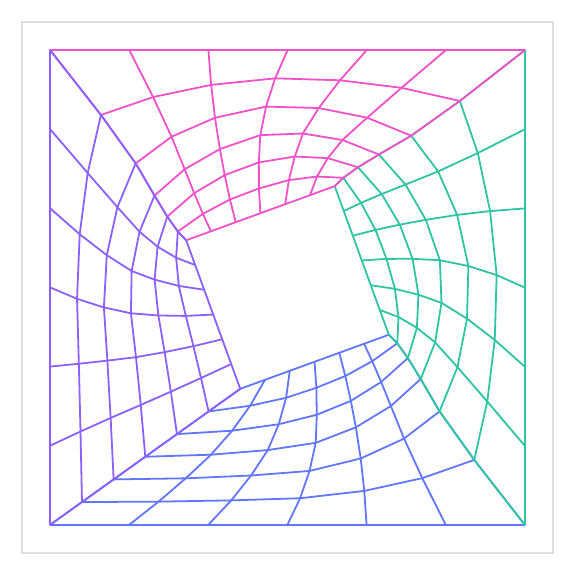}};
    \node[chip green] at ([xshift=-2pt,yshift=-2pt]BPa.north east)
        {$20^\circ$};

    \node[image] (BPb) at (\xD,\yBPb)
        {\includegraphics[width=\bpW,trim=23 23 23 23,clip]{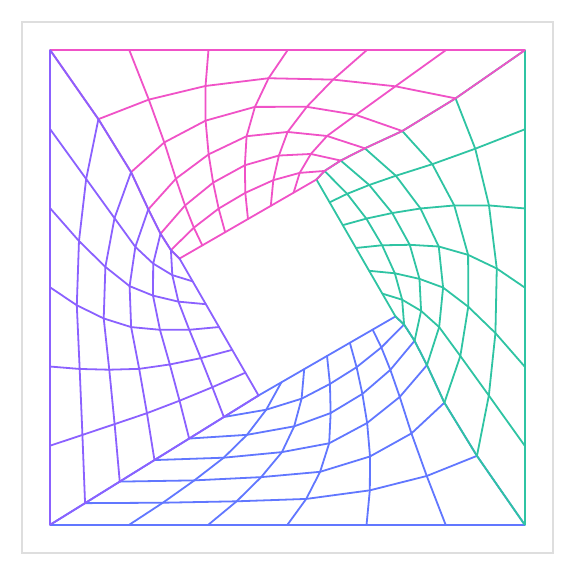}};
    \node[chip green] at ([xshift=-2pt,yshift=-2pt]BPb.north east)
        {$30^\circ$};

    \node[image] (BPc) at (\xD,\yBPc)
        {\includegraphics[width=\bpW,trim=23 23 23 23,clip]{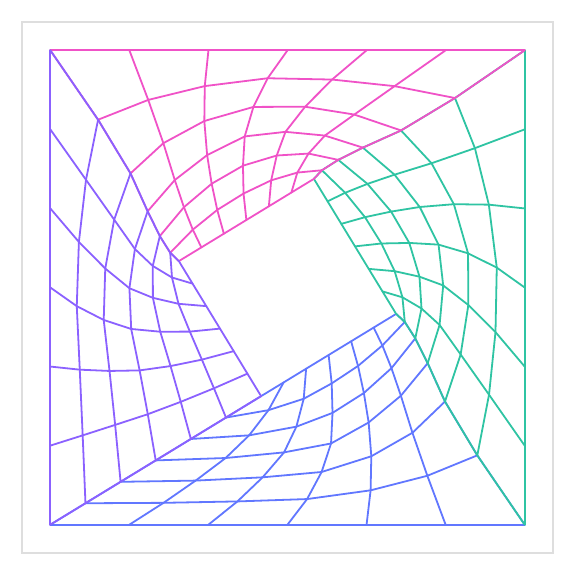}};
    \node[chip green] at ([xshift=-2pt,yshift=-2pt]BPc.north east)
        {$31.35^\circ$};

    \draw[ForestGreen!55!black, line width=0.9pt, ->, >=Stealth]
        ($(BPa.south west) + (-0.55, 0.1)$)
        --
        ($(BPc.north west) + (-0.55, -0.1)$);

    \node[font=\scriptsize\itshape, text=ForestGreen!50!black,
          rotate=90, anchor=south]
        at ($(BPb.west) + (-0.65, 0)$)
        {increasing rotation};

    \node[biglab, yshift=-2.2mm] at (BPc.south)
        {{\bfseries Barrier Patch Method}\\[-0.5mm]
        {\color{ForestGreen!55!black}$\theta_{\max}=31.35^\circ$}};

    \end{tikzpicture}%
    }

    \caption{Maximum admissible rotation angles for first-order ($p=1$) parameterizations with 2 uniform refinements, using a $0.05^\circ$ angle resolution and without inner-boundary tangential slip. Each panel in the left block shows the parameterization at the maximum sampled angle the corresponding classical method can handle before losing validity. The right column shows the Barrier Patch Method at two intermediate angles ($20^\circ$ and $30^\circ$) and its admissible limit 31.35$^\circ$.}
    \label{fig:p1-method-stress-comparison}
\end{figure}

\subsection{Perpendicular flap}
\label{sec:perpendicular-flap}

As the first coupled configuration, we consider the perpendicular-flap benchmark from the preCICE tutorial suite \citep{Chourdakis2021}. The benchmark models two-dimensional fluid--structure interaction between channel flow and an elastic flap mounted perpendicular to the bottom wall; the incoming flow exerts pressure on the flap and induces oscillatory motion. Reference solutions from several solver combinations are available, and the accuracy of our isogeometric solver and the spline-based coupling interface on this benchmark has been established in \citep{li2026isogeometric}. Here we therefore focus on the parameterization component and compare it to classical mesh-update strategies.

\paragraph{Problem setup}
The physical domain, geometric dimensions, and material parameters are summarized in \Cref{fig:perpendicular-flap-setup}. The flow is incompressible; a uniform horizontal velocity is prescribed on $\Gamma_{\mathrm{inflow}}$, a free boundary condition on $\Gamma_{\mathrm{outflow}}$, and no-slip on the top and bottom walls. The flap is clamped on $\Gamma_{\mathrm{fixed}}$ and loaded by the fluid traction on its wetted boundary $\Gamma_{fs}$. The structure is modeled as a St.\ Venant--Kirchhoff solid. 
Time integration is carried out with the backward Euler scheme up to $T = 5\,\mathrm{s}$ using a constant time step $\Delta t = 0.01\,\mathrm{s}$. The two subproblems are coupled in a partitioned manner through preCICE using an implicit coupling scheme with IQN-ILS quasi-Newton acceleration~\citep{Delaisse2023QuasiNewton}.

\begin{figure}[H]
    \centering
    \begin{minipage}[c]{0.60\linewidth}
        \centering
        \includegraphics[width=\linewidth]{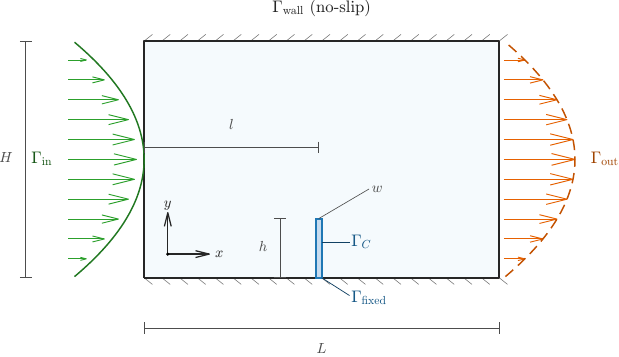}
    \end{minipage}%
    \hfill
    \begin{minipage}[c]{0.36\linewidth}
        \centering
        \footnotesize
        \setlength{\tabcolsep}{6pt}
        \renewcommand{\arraystretch}{1.25}
        \begin{tabular}{@{}l r l@{}}
            \toprule
            \multicolumn{3}{c}{\textsc{Dimensions}} \\
            \midrule
            $L$  & $6.00$  & m \\
            $H$  & $4.00$  & m \\
            $l$  & $2.95$  & m \\
            $h$  & $1.00$  & m \\
            $w$  & $0.10$  & m \\
            \midrule
            \multicolumn{3}{c}{\textsc{Fluid}} \\
            \midrule
            $\rho^f$ & $1.0$          & kg/m$^3$ \\
            $\nu^f$  & $1.0\times10^{-5}$ & m$^2$/s \\
            \midrule
            \multicolumn{3}{c}{\textsc{Solid}} \\
            \midrule
            $E$      & $4\times 10^{6}$   & N/m$^2$ \\
            $\nu^s$  & $0.3$              & --      \\
            $\rho^s$ & $3\times 10^{3}$   & kg/m$^3$ \\
            \bottomrule
        \end{tabular}
    \end{minipage}
    \caption{Setup and parameters of the perpendicular-flap 
    benchmark. The geometry schematic is adapted from the preCICE 
    documentation \citep{Chourdakis2021}.}
    \label{fig:perpendicular-flap-setup}
\end{figure}

\paragraph{Snapshot of the coupled response}
\Cref{fig:flap-snapshots} shows the fluid velocity magnitude together with the regenerated isogeometric parameterization at six instants spanning the simulation interval. As the flap oscillates in response to the channel flow, the surrounding patches deform substantially near the flap tip, yet the barrier-patch parameterization keeps every element well-shaped at every time step.

\begin{figure}[H]
    \centering
    \includegraphics[width=0.5\linewidth]{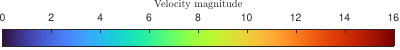} \\
    \subfigure[$t = 0.1\,\mathrm{s}$]{%
        \includegraphics[width=0.30\linewidth]{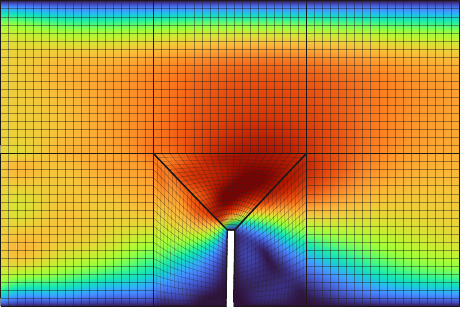}}
    \hfill
    \subfigure[$t = 1.0\,\mathrm{s}$]{%
        \includegraphics[width=0.30\linewidth]{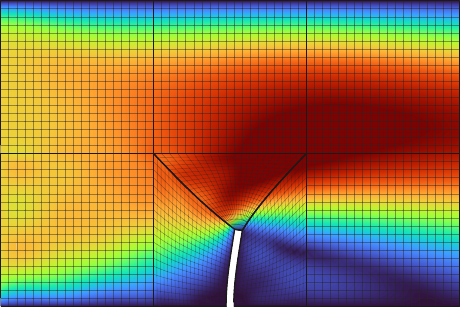}}
    \hfill
    \subfigure[$t = 2.0\,\mathrm{s}$]{%
        \includegraphics[width=0.30\linewidth]{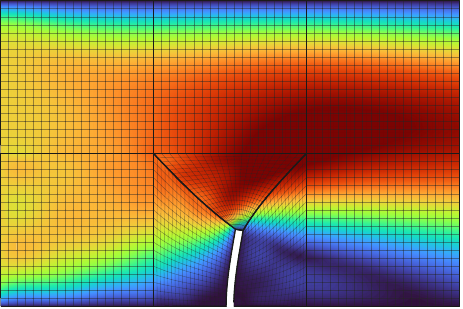}}

    \vspace{0.4em}

    \subfigure[$t = 3.0\,\mathrm{s}$]{%
        \includegraphics[width=0.30\linewidth]{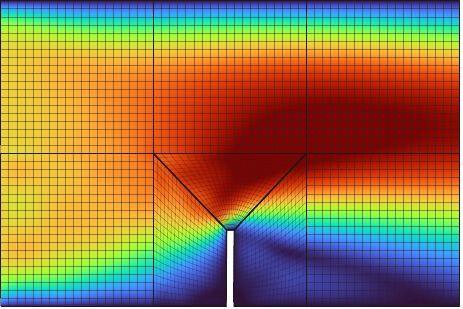}}
    \hfill
    \subfigure[$t = 4.0\,\mathrm{s}$]{%
        \includegraphics[width=0.30\linewidth]{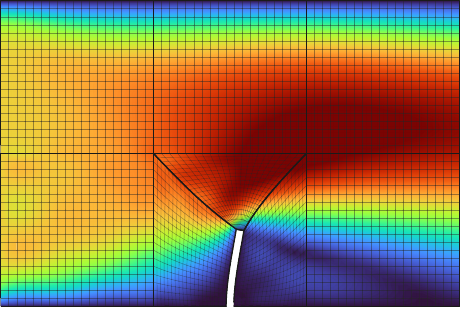}}
    \hfill
    \subfigure[$t = 5.0\,\mathrm{s}$]{%
        \includegraphics[width=0.30\linewidth]{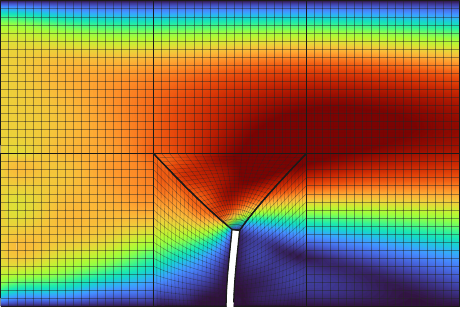}}

    \caption{Time history of the fluid velocity magnitude and the isogeometric parameterization for the perpendicular-flap benchmark, shown at six instants spanning $[0.1, 5.0]\,\mathrm{s}$. The flap oscillates in response to the channel flow, and the surrounding patches are reparameterized at each time step to remain valid throughout the deformation. All panels share the common color bar shown above each frame.}
    \label{fig:flap-snapshots}
\end{figure}

\paragraph{Comparison with classical mesh-update strategies}
We compare the barrier-patch (BP) parameterization with five classical mesh-update strategies, while keeping the fluid and structural discretizations identical: harmonic extension (HE), incremental harmonic extension (IHE), linear elasticity (LE), incremental linear elasticity (ILE), and tangential incremental nonlinear elasticity (TINE).

\begin{figure}[H]
    \centering
    \begin{tikzpicture}
        \definecolor{cBP}{RGB}{31,119,180}
        \definecolor{cHE}{RGB}{230,97,1}
        \definecolor{cIHE}{RGB}{44,160,44}
        \definecolor{cLE}{RGB}{214,39,40}
        \definecolor{cILE}{RGB}{148,103,189}
        \definecolor{cTINE}{RGB}{140,86,75}
        \begin{axis}[
          width=0.82\linewidth,
          height=0.36\textheight,
          xlabel={Time [s]},
          ylabel={Tip displacement $x$ [mm]},
          xlabel style={font=\normalsize, yshift=-3pt},
          ylabel style={font=\normalsize, yshift=6pt},
          tick label style={font=\small},
          xmin=0, xmax=5,
          axis line style={line width=1.2pt, black!85},
          tick style={line width=1.0pt, black!85},
          major tick length=5pt,
          minor tick length=3pt,
          grid=both,
          major grid style={line width=0.6pt, draw=gray!30},
          minor grid style={line width=0.4pt, draw=gray!15, dashed},
          minor tick num=1,
          filter discard warning=false,
          enlarge x limits=false,
          enlarge y limits=0.08,
          legend style={
            at={(0.97,0.03)}, anchor=south east,
            draw=black!50, line width=0.8pt,
            fill=white, fill opacity=0.92, text opacity=1,
            font=\small,
            column sep=8pt, row sep=2pt, inner sep=7pt,
            rounded corners=2pt, legend columns=3
          },
          legend cell align={left},
          legend image post style={line width=1.8pt},
        ]
        \addplot[color=cBP, solid, line width=2.5pt, each nth point=2]
          table [x index=0, y expr=\thisrowno{3}*1000, col sep=space, skip first n=1]
          {BP_solid_watchpoint.dat};
        \addlegendentry{BP}
        \addplot[color=cHE, dash pattern=on 8pt off 4pt, line width=2.3pt, each nth point=2]
          table [x index=0, y expr=\thisrowno{3}*1000, col sep=space, skip first n=1]
          {HE_solid_watchpoint.dat};
        \addlegendentry{HE}
        \addplot[color=cIHE, dash pattern=on 1pt off 3pt, line width=2.4pt, each nth point=2]
          table [x index=0, y expr=\thisrowno{3}*1000, col sep=space, skip first n=1]
          {IHE_solid_watchpoint.dat};
        \addlegendentry{IHE}
        \addplot[color=cLE, dash pattern=on 6pt off 3pt on 1pt off 3pt, line width=2.3pt, each nth point=2]
          table [x index=0, y expr=\thisrowno{3}*1000, col sep=space, skip first n=1]
          {LE_solid_watchpoint.dat};
        \addlegendentry{LE}
        \addplot[color=cILE, dash pattern=on 4pt off 3pt, line width=2.3pt, each nth point=2]
          table [x index=0, y expr=\thisrowno{3}*1000, col sep=space, skip first n=1]
          {ILE_solid_watchpoint.dat};
        \addlegendentry{ILE}
        \addplot[color=cTINE, dash pattern=on 6pt off 2pt on 2pt off 2pt on 1pt off 2pt, line width=2.3pt, each nth point=2]
          table [x index=0, y expr=\thisrowno{3}*1000, col sep=space, skip first n=1]
          {TINE_solid_watchpoint.dat};
        \addlegendentry{TINE}
        \end{axis}
    \end{tikzpicture}
    \caption{Perpendicular-flap benchmark: horizontal tip displacement histories obtained with six mesh-update strategies.}
    \label{fig:flap-tip-displacement}
\end{figure}

\Cref{fig:flap-tip-displacement} shows the tip-displacement histories over $T = 5\,\mathrm{s}$. The six curves are visually indistinguishable on the plotted scale, so the barrier-patch parameterization reproduces the response of the established mesh-update strategies in this moderate-deformation regime.

\begin{figure}[H]
    \centering
    \begin{tikzpicture}
        \definecolor{cBP}{RGB}{31,119,180}
        \definecolor{cHE}{RGB}{230,97,1}
        \definecolor{cIHE}{RGB}{44,160,44}
        \definecolor{cLE}{RGB}{214,39,40}
        \definecolor{cILE}{RGB}{148,103,189}
        \definecolor{cTINE}{RGB}{140,86,75}
        \begin{axis}[
          width=0.82\linewidth,
          height=0.36\textheight,
          xlabel={Time [s]},
          ylabel={$\mathcal{J}_{\min}$},
          xlabel style={font=\normalsize, yshift=-3pt},
          ylabel style={font=\normalsize, yshift=6pt},
          tick label style={font=\small},
          xmin=0, xmax=5,
          axis line style={line width=1.2pt, black!85},
          tick style={line width=1.0pt, black!85},
          major tick length=5pt,
          minor tick length=3pt,
          grid=both,
          major grid style={line width=0.6pt, draw=gray!30},
          minor grid style={line width=0.4pt, draw=gray!15, dashed},
          minor tick num=1,
          enlarge x limits=false,
          enlarge y limits=0.08,
          legend style={
            at={(0.97,0.03)}, anchor=south east,
            draw=black!50, line width=0.8pt,
            fill=white, fill opacity=0.92, text opacity=1,
            font=\small,
            column sep=8pt, row sep=2pt, inner sep=7pt,
            rounded corners=2pt, legend columns=3
          },
          legend cell align={left},
          legend image post style={line width=1.8pt},
        ]
        \addplot[color=cBP, solid, line width=2.5pt, each nth point=2]
          table [x=time, y=minScaledJ, col sep=comma]
          {mesh_quality_jacobian_barrier.csv};
        \addlegendentry{BP}
        \addplot[color=cHE, dash pattern=on 8pt off 4pt, line width=2.3pt, each nth point=2]
          table [x=time, y=minScaledJ, col sep=comma]
          {mesh_quality_jacobian_he.csv};
        \addlegendentry{HE}
        \addplot[color=cIHE, dash pattern=on 1pt off 3pt, line width=2.4pt, each nth point=2]
          table [x=time, y=minScaledJ, col sep=comma]
          {mesh_quality_jacobian_ihe.csv};
        \addlegendentry{IHE}
        \addplot[color=cLE, dash pattern=on 6pt off 3pt on 1pt off 3pt, line width=2.3pt, each nth point=2]
          table [x=time, y=minScaledJ, col sep=comma]
          {mesh_quality_jacobian_le.csv};
        \addlegendentry{LE}
        \addplot[color=cILE, dash pattern=on 4pt off 3pt, line width=2.3pt, each nth point=2]
          table [x=time, y=minScaledJ, col sep=comma]
          {mesh_quality_jacobian_ile.csv};
        \addlegendentry{ILE}
        \addplot[color=cTINE, dash pattern=on 6pt off 2pt on 2pt off 2pt on 1pt off 2pt, line width=2.3pt, each nth point=2]
          table [x=time, y=minScaledJ, col sep=comma]
          {mesh_quality_jacobian_tine.csv};
        \addlegendentry{TINE}
        \end{axis}
    \end{tikzpicture}
    \caption{Perpendicular-flap benchmark: minimum scaled Jacobian $\mathcal{J}_{\min}$ of the fluid ALE mesh over time. The barrier-patch method maintains $\mathcal{J}_{\min}$ in a narrow band around $0.70$.}
    \label{fig:flap-min-jacobian}
\end{figure}

\Cref{fig:flap-min-jacobian} reports the minimum scaled Jacobian of the fluid ALE mesh over $[0, 5]\,\mathrm{s}$. All six methods remain strictly positive throughout, confirming the validity of each parameterization. HE, IHE, LE, ILE, and TINE track the oscillation of the flap: $\mathcal{J}_{\min}$ decreases during phases of large tip displacement and recovers as the flap returns toward its undeformed configuration. The barrier-patch method, by contrast, maintains $\mathcal{J}_{\min}$ above $0.6$ with only minor fluctuations, yielding a mesh quality that is effectively decoupled from the structural motion.

\subsection{Turek-Hron FSI 2}
\label{sec:turek-hron-fsi2}

The second benchmark is the Turek--Hron FSI2 problem~\citep{turek2006proposal}, a standard test case for fluid--structure interaction. It models the two-dimensional flow-induced oscillations of an elastic beam clamped to the downstream side of a rigid cylinder. The parameters of the FSI2 configuration place the system in a regime of large-amplitude self-sustained oscillations, making it a demanding test of both solver accuracy and mesh robustness under sustained deformation.

\paragraph{Problem setup} The physical domain, shown in \Cref{fig:turek-hron-setup}, is a rectangular channel of length $L = 2.5\,\mathrm{m}$ and height $H = 0.41\,\mathrm{m}$. A rigid circular cylinder of diameter $D = 0.1\,\mathrm{m}$ is centered at $(0.2, 0.2)\,\mathrm{m}$, with an elastic beam of length $l = 0.35\,\mathrm{m}$ and thickness $h = 0.02\,\mathrm{m}$ clamped to its downstream side.

\begin{figure}[H]
    \centering
    \includegraphics[width=\linewidth]{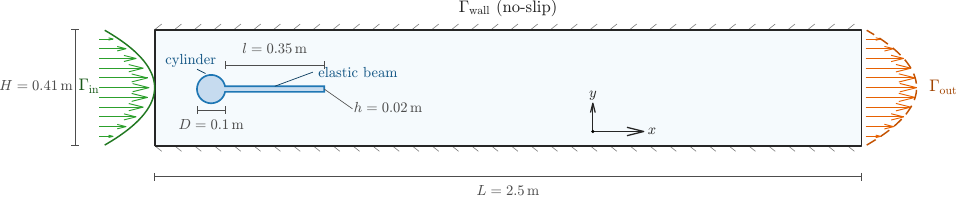}
    \caption{Turek--Hron FSI2 benchmark setup: a rigid cylinder with an elastic beam clamped to its downstream side, immersed in a channel flow. The inflow is a parabolic profile; the outflow boundary is stress-free.}
    \label{fig:turek-hron-setup}
\end{figure}

A parabolic velocity profile is prescribed at the inflow boundary $\Gamma_{\text{in}}$,
\begin{equation}
  v_{\text{in}}(y) = 1.5\,\bar{v}\,\frac{4\,y\,(H - y)}{H^2},
\end{equation}
with mean velocity $\bar{v} = 1\,\mathrm{m/s}$, corresponding to a Reynolds number $\mathrm{Re} = 100$. No-slip conditions are imposed on the upper and lower walls, and a stress-free condition is applied at the outlet $\Gamma_{\text{out}}$.

The fluid has density $\rho^f = 10^3\,\mathrm{kg/m^3}$ and dynamic viscosity $\mu^f = 1.0\,\mathrm{Pa \cdot s}$. The solid is modeled as a St.~Venant--Kirchhoff material with density $\rho^s = 10^4\,\mathrm{kg/m^3}$, shear modulus $\mu^s = 0.5 \times 10^6\,\mathrm{Pa}$, and Poisson's ratio $\nu^s = 0.4$. Time integration uses a constant step $\Delta t = 0.0025\,\mathrm{s}$ over the interval $[0, 18]\,\mathrm{s}$.

Fluid--structure coupling follows the scheme used in the previous benchmark: a partitioned implicit solver with Aitken relaxation~\citep{Kutteler2008Aitken}. The fluid parameterization is handled by the barrier-patch parameterization of \Cref{sec:mesh-generation}; at each coupling iteration, only the patches adjacent to the elastic beam are regenerated, while the surrounding extension patches remain fixed.

\begin{figure}[H]
    \centering
    \includegraphics[width=0.85\linewidth]{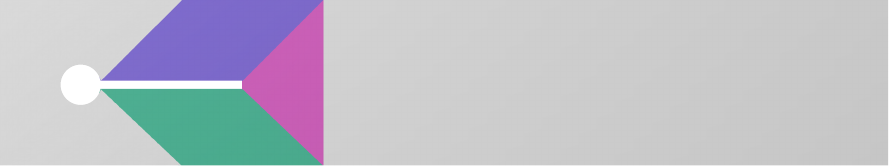}
    \caption{Patch decomposition of the fluid domain for the FSI2 benchmark. The rigid cylinder (white disc) and the elastic beam (white strip) form the inner obstacle. Three coloured patches conforming to the cylinder--beam interface (purple, magenta, and green) constitute the \emph{regenerated} region: at every coupling iteration these patches are reparameterized to follow the deformed beam, using the barrier-patch construction of \Cref{sec:mesh-generation}. The remaining grey region is covered by \emph{frozen} extension patches that span the rest of the channel up to the inflow, outflow, and wall boundaries; these are generated once and remain fixed throughout the simulation. Confining the regeneration to a narrow neighbourhood of the beam keeps the per-iteration meshing cost essentially independent of the channel length.}
    \label{fig:turek-hron-patches}
\end{figure}

\paragraph{Patch layout} The fluid domain is decomposed into a small set of spline patches assembled around the cylinder--beam obstacle, as illustrated in \Cref{fig:turek-hron-patches}. Three patches form a thin layer that conforms to the obstacle: two wedge-shaped patches above and below the beam, and a patch closing off its tip. These three patches are regenerated at every coupling iteration so that their boundaries track the deformed beam exactly. The remainder of the channel (the inflow region upstream of the cylinder, the long downstream wake, and the upper and lower wall strips) is tiled by a fixed set of extension patches that are constructed once and never modified during the simulation. This choice isolates the geometric variability to a narrow neighborhood of the beam, where the deformation is concentrated, and keeps the cost of each remeshing operation independent of the channel length.

\paragraph{Results}
We evaluate the barrier-patch parameterization along three dimensions: the validity and quality of the regenerated mesh under sustained large deformation, the accuracy of the FSI response against published reference values, and the computational overhead introduced by parameterization within the coupling loop.

\Cref{fig:turek-hron-fsi2-snapshots} presents the velocity field and the regenerated parameterization at two instants in the periodic regime, chosen near opposite extrema of the beam oscillation. The barrier-patch parameterization delivers a well-shaped parameterization in both configurations, with no visible element degradation even at maximum tip deflection. Across the full simulation window, the minimum scaled Jacobian remains strictly positive, confirming that parameterization validity is preserved throughout the sustained large-amplitude motion.

\begin{figure}[H]
    \centering
  \includegraphics[width=0.5\linewidth,trim=250 449 250 85,clip]{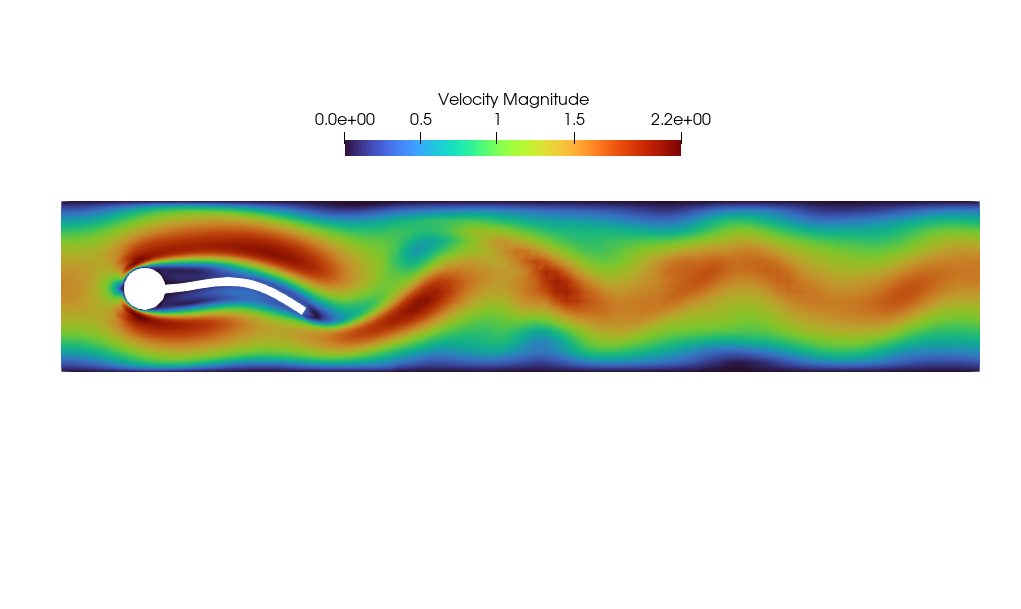}

  \vspace{0.15em}

  \subfigure[Velocity field, $t = 14.93\,\mathrm{s}$.]{
    \includegraphics[width=0.47\linewidth,trim=55 230 35 190,clip]{Turek-Hron-fsi2-velocity14.93.png}
  }
  \hfill
  \subfigure[Velocity field, $t = 17.28\,\mathrm{s}$.]{
    \includegraphics[width=0.47\linewidth,trim=55 230 35 190,clip]{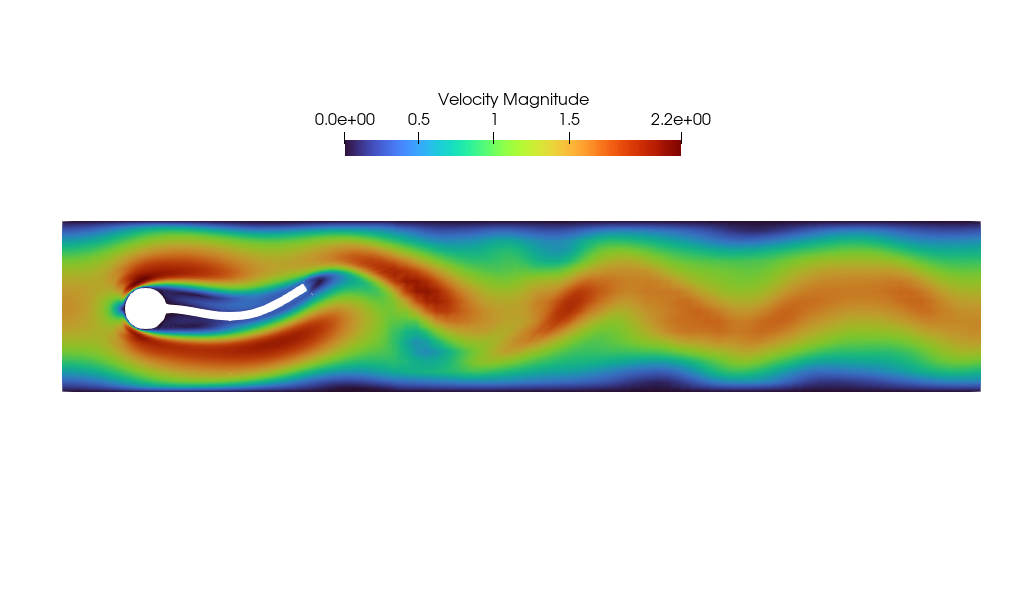}
  }

  \vspace{0.3em}

  \subfigure[Barrier-patch parameterization around the beam, $t = 14.93\,\mathrm{s}$.]{
    \includegraphics[width=0.47\linewidth]{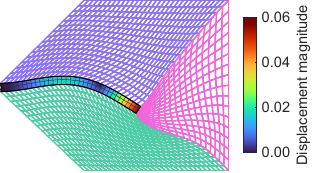}
  }
  \hfill
  \subfigure[Barrier-patch parameterization around the beam, $t = 17.28\,\mathrm{s}$.]{
    \includegraphics[width=0.47\linewidth]{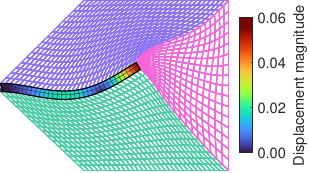}
  }
  \caption{Turek--Hron FSI2 benchmark: velocity field and regenerated mesh at two instants in the periodic regime, near opposite extrema of the beam oscillation. The velocity panels share the common color bar shown at the top.}
    \label{fig:turek-hron-fsi2-snapshots}
\end{figure}

To assess solution accuracy, \Cref{fig:fsi2-tip-displacement} reports the time history of the vertical tip displacement of the elastic beam. After an initial transient of roughly $5\,\mathrm{s}$, the response settles into a periodic regime whose amplitude agrees closely with the reference value $\pm 80.6\,\mathrm{mm}$ reported by Turek and Hron~\citep{turek2006proposal}, indicating that the barrier-patch parameterization does not compromise the fidelity of the coupled FSI response.

\begin{figure}[H]
    \centering
    \begin{tikzpicture}
        \definecolor{nbBlue}{RGB}{31,119,180}   
        \definecolor{nbOrange}{RGB}{230,97,1}   
        \begin{axis}[
          width=0.82\linewidth,
          height=0.36\textheight,
          xlabel={Time [s]},
          ylabel={Tip displacement $y$ [mm]},
          xlabel style={font=\normalsize, yshift=-3pt},
          ylabel style={font=\normalsize, yshift=6pt},
          tick label style={font=\small},
          xmin=0, xmax=19,
          axis line style={line width=1.2pt, black!85},
          tick style={line width=1.0pt, black!85},
          major tick length=5pt,
          minor tick length=3pt,
          grid=both,
          major grid style={line width=0.6pt, draw=gray!30},
          minor grid style={line width=0.4pt, draw=gray!15, dashed},
          minor tick num=1,
          enlarge x limits=false,
          enlarge y limits=0.08,
          legend style={
            at={(0.03,0.97)},
            anchor=north west,
            draw=black!50,
            line width=0.8pt,
            fill=white,
            fill opacity=0.92,
            text opacity=1,
            font=\small,
            row sep=3pt,
            inner sep=7pt,
            rounded corners=2pt
          },
          legend cell align={left},
          legend image post style={line width=2.0pt}
        ]
        \addplot[
          color=nbBlue,
          line width=1.6pt,
          each nth point=2,
          filter discard warning=false,
          unbounded coords=discard
        ] table [x index=0, y expr=\thisrowno{5}*1000, col sep=space, skip first n=1]
          {Turek_FSI2.txt};
        \addlegendentry{Present}

        \addplot[
          color=nbOrange,
          dash pattern=on 10pt off 5pt,
          line width=1.4pt,
          forget plot
        ] coordinates {(0,80.6) (19,80.6)};
        \addplot[
          color=nbOrange,
          dash pattern=on 10pt off 5pt,
          line width=1.4pt
        ] coordinates {(0,-80.6) (19,-80.6)};
        \addlegendentry{Ref.\ $\pm 80.6$\,mm}
        \end{axis}
    \end{tikzpicture}
    \caption{FSI2 benchmark: time history of the vertical tip displacement of the elastic beam. The dashed orange lines mark the reference amplitude $\pm 80.6\,\mathrm{mm}$ reported by Turek and Hron~\citep{turek2006proposal}.}
    \label{fig:fsi2-tip-displacement}
\end{figure}

\Cref{fig:fsi2-timing} reports the cumulative wall-clock time, decomposed into contributions from the beam solver, the flow solver, and the barrier-patch parameterization. The flow solve dominates the total cost, while parameterization contributes a modest share that grows approximately linearly with simulation time. The barrier-patch procedure therefore imposes an overhead comparable to that of conventional mesh-motion schemes, despite offering the stronger guarantee of strictly positive Jacobians under sustained large deformation.

\begin{figure}[H]
    \centering
    \begin{tikzpicture}
        \definecolor{cBeam}{RGB}{31,119,180}    
        \definecolor{cFlow}{RGB}{230,97,1}      
        \definecolor{cALE}{RGB}{44,160,44}      
        \begin{axis}[
          width=0.82\linewidth,
          height=0.36\textheight,
          xlabel={Time [s]},
          ylabel={Cumulative wall-clock time [s]},
          xlabel style={font=\normalsize, yshift=-3pt},
          ylabel style={font=\normalsize, yshift=6pt},
          tick label style={font=\small},
          xmin=0, xmax=19,
          ymin=0,
          axis line style={line width=1.2pt, black!85},
          tick style={line width=1.0pt, black!85},
          major tick length=5pt,
          minor tick length=3pt,
          axis on top=false,
          grid=both,
          major grid style={line width=0.6pt, draw=gray!30, opacity=0.4},
          minor grid style={line width=0.6pt, draw=gray!15, opacity=0.4, dashed},
          minor tick num=1,
          enlarge x limits=false,
          enlarge y limits=false,
          legend style={
            at={(0.03,0.97)},
            anchor=north west,
            draw=black!50,
            line width=0.8pt,
            fill=white,
            fill opacity=0.95,
            text opacity=1,
            font=\small,
            row sep=3pt,
            inner sep=7pt,
            rounded corners=2pt
          },
          legend cell align={left},
          legend image post style={line width=0pt},
          stack plots=y,
          area style,
        ]
        \addplot[
          fill=cBeam, fill opacity=1.0,
          draw=cBeam!70!black, line width=1.6pt,
          each nth point=5,
          filter discard warning=false,
          unbounded coords=discard
        ] table [x index=0, y index=10, col sep=space, skip first n=1]
          {Turek_FSI2.txt} \closedcycle;
        \addlegendentry{Beam solver}

        \addplot[
          fill=cFlow, fill opacity=1.0,
          draw=cFlow!70!black, line width=1.6pt,
          each nth point=5,
          filter discard warning=false,
          unbounded coords=discard
        ] table [x index=0, y index=9, col sep=space, skip first n=1]
          {Turek_FSI2.txt} \closedcycle;
        \addlegendentry{Flow solver}

        \addplot[
          fill=cALE, fill opacity=1.0,
          draw=cALE!70!black, line width=1.6pt,
          each nth point=5,
          filter discard warning=false,
          unbounded coords=discard
        ] table [x index=0, y index=8, col sep=space, skip first n=1]
          {Turek_FSI2.txt} \closedcycle;
        \addlegendentry{Barrier-patch mesh}
        \end{axis}
    \end{tikzpicture}
    \caption{FSI2 benchmark: cumulative wall-clock time decomposed into the beam solver, the flow solver, and the barrier-patch mesh-regeneration step.}
    \label{fig:fsi2-timing}
\end{figure}

\subsection{Sphere-in-channel benchmarks in three dimensions}
\label{sec:sphere-falling-channel}
\label{sec:falling-sphere}

To assess the method under large translational motion, we consider two three-dimensional sphere-in-channel configurations. The prescribed-translation benchmark of Di Cristofaro et al.~\citep{DiCristofaro2025ALEFlying} tests whether a fixed spline space and patch topology remain valid over a $40\,\mathrm{m}$ displacement and permits direct comparison of failure time, mesh reconnection, and mesh-size growth. The free-fall benchmark then tests the two-way coupled response against published terminal-velocity data.

\subsubsection{Prescribed large translation}
\label{sec:falling-sphere-prescribed}

\paragraph{Problem setup}
We first consider the prescribed-translation case of Di Cristofaro et al.~\citep{DiCristofaro2025ALEFlying}, illustrated in \Cref{fig:prescribed-falling-sphere}. The fluid occupies a cylinder of length $L=60\,\mathrm{m}$ and diameter $D_2=20\,\mathrm{m}$, and contains a concentric rigid sphere of diameter $D_1=8\,\mathrm{m}$. Thus the channel and sphere radii are $10\,\mathrm{m}$ and $4\,\mathrm{m}$, respectively. With the axial coordinate spanning $0\leq z\leq 60\,\mathrm{m}$, the sphere center starts at $z_c(0)=50\,\mathrm{m}$ above the lower base and is assigned the velocity
\begin{equation}
    \mathbf{u}^f\big|_{\Gamma_s}=-5\,\mathbf{e}_z\ \mathrm{m/s}.
    \label{eq:prescribed-sphere-velocity}
\end{equation}
Homogeneous Dirichlet conditions are imposed on the remaining fixed boundary. As in the reference study, $\rho^f=1\,\mathrm{kg/m^3}$ and $\mu^f=1\,\mathrm{Pa\,s}$. The interval $[0,8]\,\mathrm{s}$ is discretized with $\Delta t=0.02\,\mathrm{s}$, during which the sphere translates by $40\,\mathrm{m}$. The fluid domain is represented by six volumetric spline patches, in the same O-grid topology used for the free-fall case of \Cref{sec:falling-sphere-free}, and is reconstructed independently from the current sphere position at each step; no connectivity change or refinement is performed, so the element and control-point counts are identical at every one of the 400 steps. The geometry and pressure space use degree $(2,2,2)$, while the velocity space is obtained by one degree elevation and uses degree $(3,3,3)$, giving a Taylor--Hood $Q_3/Q_2$ pair. The discretization contains $3\times3\times2=18$ volume elements per patch and 108 volume elements in total. After identification of conforming patch-interface degrees of freedom and elimination of Dirichlet unknowns, the coupled system contains 4052 global free degrees of freedom, comprising 3660 velocity degrees of freedom (1220 per component) and 392 pressure degrees of freedom.

\begin{figure}[H]
    \centering
    \begin{minipage}[b]{0.50\linewidth}
        \centering
        \includegraphics[height=8.6cm]{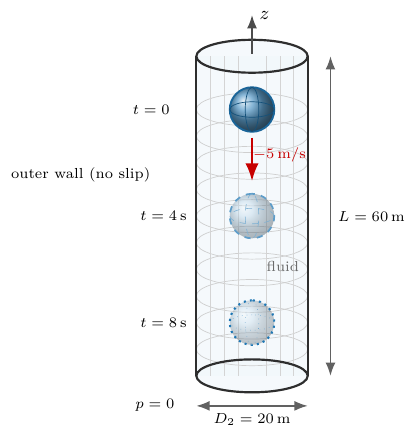}

        \smallskip
        {{\small (a) Longitudinal trajectory}}
    \end{minipage}
    \hfill
    \begin{minipage}[b]{0.47\linewidth}
        \centering
        \includegraphics[height=8.6cm]{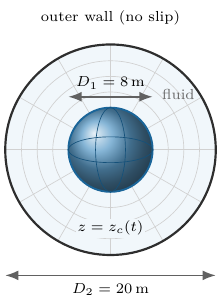}

        \smallskip
        {{\small (b) Transverse cross-section}}
    \end{minipage}
    \caption{Prescribed-translation setup based on Di Cristofaro et al.~\citep{DiCristofaro2025ALEFlying}. Panel~(a) shows the sphere moving vertically from $z_c=50\,\mathrm{m}$ to $z_c=10\,\mathrm{m}$ with velocity $-5\,\mathbf{e}_z\,\mathrm{m/s}$. Panel~(b) shows the transverse section through the sphere centre, with channel diameter $D_2=20\,\mathrm{m}$ and sphere diameter $D_1=8\,\mathrm{m}$.}
    \label{fig:prescribed-falling-sphere}
\end{figure}

\paragraph{Results}
\Cref{fig:prescribed-falling-sphere-fields} shows the fluid response alongside the regenerated volumetric parameterization at the two latest comparison times.

\begin{figure}[H]
    \centering
    \begin{minipage}[t]{0.49\linewidth}
        \centering
        \includegraphics[width=\linewidth]{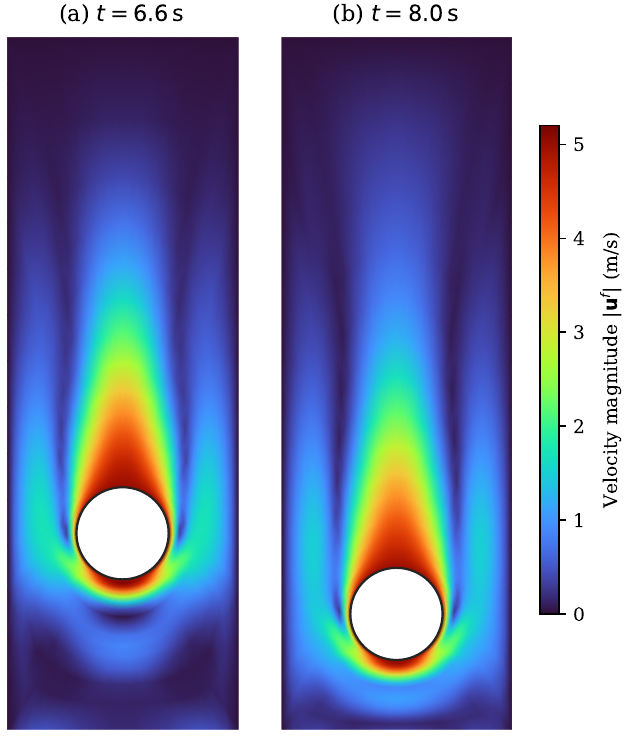}
    \end{minipage}
    \hfill
    \begin{minipage}[t]{0.49\linewidth}
        \centering
        \includegraphics[width=\linewidth]{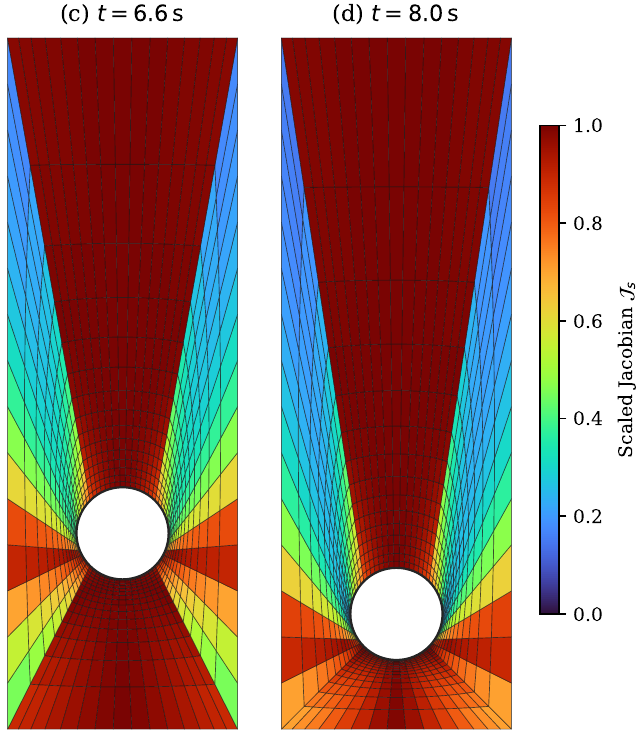}
    \end{minipage}
    \caption{Centre-plane views of the prescribed-translation test at $t=6.6\,\mathrm{s}$ and $8.0\,\mathrm{s}$, selected to match the two snapshot times reported by Di Cristofaro et al.~\citep{DiCristofaro2025ALEFlying}. Panels~(a,b) show the fluid-velocity magnitude with the common range $0\leq\lVert\mathbf{u}^f\rVert\leq5.2\,\mathrm{m/s}$. Panels~(c,d) show the regenerated volumetric parameterization, coloured by the cellwise scaled Jacobian $\mathcal{J}_s$ on the common rainbow scale $[0,1]$; the overlaid lines show the parameterization grid. The white disk denotes the rigid sphere.}
    \label{fig:prescribed-falling-sphere-fields}
\end{figure}

Panels~(c,d) are colored by the scaled Jacobian
$\mathcal{J}_s$, i.e., the quantity minimized in
\eqref{eq:jmin-monitor}. Since $\mathcal{J}_s$ is normalized by the lengths of
the parametric tangent vectors, it reflects their relative orientation rather
than the cell aspect ratio. The long cells in the upper central patches are
nearly orthogonal and therefore have values close to one; the more skewed
cells in the neighboring patches have lower values. Jumps at patch
interfaces are expected because the geometry is $C^0$-continuous there,
whereas its patch-wise derivatives need not be.

\begin{figure}[H]
    \centering
    \begin{tikzpicture}
    \definecolor{sphereBlue}{RGB}{31,119,180}
    \definecolor{cMin}{RGB}{192,57,43}
    \definecolor{cMax}{RGB}{39,174,96}
    \pgfmathsetmacro{\detjMinTimeA}{0}
    \pgfmathsetmacro{\detjMinTimeB}{8}
    \pgfmathsetmacro{\detjMinVal}{0.072177}
    \pgfmathsetmacro{\detjMaxTime}{4}
    \pgfmathsetmacro{\detjMaxVal}{0.093219}
    \begin{axis}[
        width=0.76\linewidth,
        height=6.4cm,
        xlabel={Time \(t\) (s)},
        ylabel={\(\mathcal{J}_{\min}\)},
        xmin=0, xmax=8,
        ymin=0.065, ymax=0.100,
        xtick={0,2,4,6,8},
        ytick={0.07,0.08,0.09,0.10},
        grid=both,
        major grid style={line width=0.5pt,draw=gray!30},
        minor grid style={line width=0.3pt,draw=gray!15,dashed},
        tick label style={font=\small},
        label style={font=\small},
        clip=false,
    ]
        \addplot[sphereBlue,line width=2.2pt]
            table[
                x=time,
                y=min_scaled_jacobian,
                col sep=comma
            ]{data/falling_sphere_min_scaled_jacobian.csv};

        \node[circle, draw=cMin, line width=1.6pt, minimum size=10pt,
              inner sep=0pt, fill=white, fill opacity=0.3]
              at (axis cs:\detjMinTimeA,\detjMinVal) {};
        \node[circle, fill=cMin, minimum size=3.5pt, inner sep=0pt]
              at (axis cs:\detjMinTimeA,\detjMinVal) {};
        \node[circle, draw=cMin, line width=1.6pt, minimum size=10pt,
              inner sep=0pt, fill=white, fill opacity=0.3]
              at (axis cs:\detjMinTimeB,\detjMinVal) {};
        \node[circle, fill=cMin, minimum size=3.5pt, inner sep=0pt]
              at (axis cs:\detjMinTimeB,\detjMinVal) {};
        \node[font=\scriptsize, anchor=south, align=center,
              fill=white, fill opacity=0.95, text opacity=1,
              draw=cMin, line width=0.9pt,
              inner sep=4pt, rounded corners=2pt]
              at ($(axis cs:4,\detjMinVal)+(0pt,8pt)$)
              {\textbf{Min} \(\pgfmathprintnumber[fixed,precision=4]{\detjMinVal}\) at \(t=0\) and \(8\,\mathrm{s}\)};

        \node[circle, draw=cMax, line width=1.6pt, minimum size=10pt,
              inner sep=0pt, fill=white, fill opacity=0.3]
              at (axis cs:\detjMaxTime,\detjMaxVal) {};
        \node[circle, fill=cMax, minimum size=3.5pt, inner sep=0pt]
              at (axis cs:\detjMaxTime,\detjMaxVal) {};
        \node[font=\scriptsize, anchor=south, align=center,
              fill=white, fill opacity=0.95, text opacity=1,
              draw=cMax, line width=0.9pt,
              inner sep=4pt, rounded corners=2pt]
              at ($(axis cs:\detjMaxTime,\detjMaxVal)+(0pt,8pt)$)
              {\textbf{Max} \(\pgfmathprintnumber[fixed,precision=4]{\detjMaxVal}\) at \(t=\pgfmathprintnumber[fixed,precision=0]{\detjMaxTime}\,\mathrm{s}\)};
    \end{axis}
\end{tikzpicture}
    \caption{Minimum scaled Jacobian $\mathcal{J}_{\min}$ over all spline patches during the prescribed translation, evaluated on 85 sampled configurations of the exported volumetric mesh. The global minimum and maximum are marked in red and green, respectively.}
    \label{fig:prescribed-falling-sphere-detj}
\end{figure}

$\mathcal{J}_{\min}$ remains strictly positive at every sampled configuration, confirming that the discrete bijectivity condition \eqref{eq:bijectivity-discrete} holds throughout the translation. Di Cristofaro et al.~\citep{DiCristofaro2025ALEFlying}, by contrast, report that their tetrahedral mesh fails at $t=2.0\,\mathrm{s}$ without reconnection, and that continuing the run requires reconnection together with an approximately $10\%$ increase in nodes and elements; the present six-patch parameterization completes the full $8.0\,\mathrm{s}$ run on the same fixed connectivity, with no growth in the number of control points or elements. As \Cref{fig:prescribed-falling-sphere-detj} shows, $\mathcal{J}_{\min}$ stays within the narrow band $[0.0722,0.0932]$ throughout, reaching its minimum at the two symmetric endpoints $t=0$ and $t=8\,\mathrm{s}$---where the sphere sits only $10\,\mathrm{m}$ from the nearest end wall---and its maximum at $t=4\,\mathrm{s}$, when the sphere is equidistant from both ends.

\subsubsection{Two-way coupled free fall}
\label{sec:falling-sphere-free}

\paragraph{Problem setup}
We next consider the freely falling sphere introduced by Casquero et al \citep{Casquero2015NURBSImmersed} and revisited by Black et al.\citep{Black2026ImmersedFSI}, whose finest-mesh solution we use as the reference in the comparison below. As in \Cref{sec:falling-sphere-prescribed}, the fluid domain is discretized into six volumetric spline patches conforming to a concentric sphere-in-cylinder geometry, here at a much smaller scale and with the sphere free to translate under the coupled fluid and gravitational forces rather than prescribed. A cylinder of radius $A=2\,\mathrm{cm}$ and height $4\,\mathrm{cm}$ contains a sphere of radius $a=0.25\,\mathrm{cm}$, initially centered at $z_c=2.5\,\mathrm{cm}$. The fluid parameters are $\rho^f=1\,\mathrm{g/cm^3}$ and $\mu^f=10\,\mathrm{dyn\,s/cm^2}$, the sphere density is $\rho^s=1.5\,\mathrm{g/cm^3}$, and gravity is $\mathbf{g}=-981\,\mathbf{e}_z\,\mathrm{cm/s^2}$. A traction-free condition is imposed at the top of the cylinder and no slip on the remaining fixed walls. The sphere is treated as rigid and constrained to axial translation, consistently with the analytical terminal-velocity model. Fluid traction and sphere displacement are exchanged by the implicit partitioned coupling scheme through preCICE. We use $\Delta t=10^{-4}\,\mathrm{s}$ and simulate until $T=0.1\,\mathrm{s}$.

Assuming Stokes flow, the wall-corrected terminal speed is
\begin{equation}
    v_{\infty}
    = \frac{1}{K}\frac{2g a^2(\rho^s-\rho^f)}{9\mu^f},
    \qquad
    K =
    \left[
        1-2.10443\frac{a}{A}
        +2.08877\left(\frac{a}{A}\right)^3
    \right]^{-1},
    \label{eq:falling-sphere-terminal-velocity}
\end{equation}
which gives $v_{\infty}=0.504824\,\mathrm{cm/s}$. The six-patch fluid parameterization contains $3\times3\times2=18$ volume elements per patch, hence 108 volume elements in total. The fluid is discretized in the same manner as in the prescribed-translation
benchmark, using the six-patch Taylor--Hood $Q_3/Q_2$ spline discretization described in \Cref{sec:falling-sphere-prescribed}.

\begin{figure}[H]
    \centering
    \includegraphics[width=0.85\linewidth]{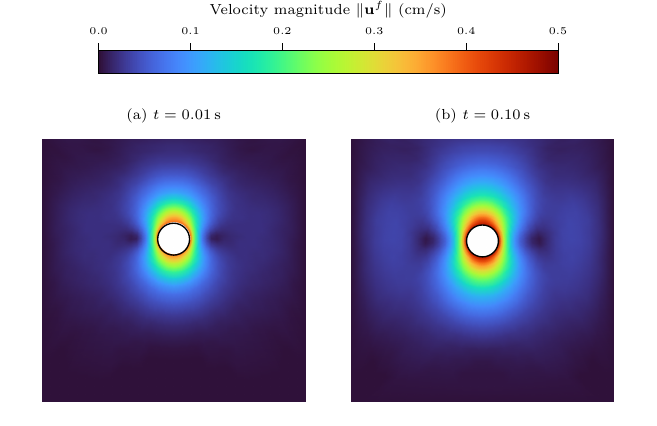}
    \caption{Two-way coupled free-fall test: fluid-velocity magnitude in the sphere's centre plane at $t=0.01\,\mathrm{s}$, shortly after release, and at $t=0.10\,\mathrm{s}$, close to the terminal speed. The white disk denotes the rigid sphere.}
    \label{fig:free-fall-velocity-snapshots}
\end{figure}

\paragraph{Results}
\Cref{fig:free-fall-velocity-snapshots} shows the fluid velocity magnitude in the sphere's centre plane shortly after release and near the end of the simulated interval, namely at $t=0.01\,\mathrm{s}$ and $t=0.10\,\mathrm{s}$.

\Cref{fig:free-falling-sphere-velocity} compares the sphere velocity obtained with the present formulation against the reference data.

\begin{figure}[H]
    \centering
    \begin{tikzpicture}
    \definecolor{nbBlue}{RGB}{31,119,180}
    \definecolor{refM3}{RGB}{204,76,2}
    \begin{axis}[
      width=0.82\linewidth,
      height=0.36\textheight,
      xlabel={Time \(t\) (s)},
      ylabel={Downward sphere speed (cm/s)},
      xlabel style={font=\normalsize, yshift=-3pt},
      ylabel style={font=\normalsize, yshift=6pt},
      tick label style={font=\small},
      xmin=0, xmax=0.1,
      ymin=0, ymax=0.56,
      xtick={0,0.02,0.04,0.06,0.08,0.10},
      axis line style={line width=1.2pt, black!85},
      tick style={line width=1.0pt, black!85},
      major tick length=5pt,
      minor tick length=3pt,
      grid=both,
      major grid style={line width=0.6pt, draw=gray!30},
      minor grid style={line width=0.4pt, draw=gray!15, dashed},
      minor tick num=1,
      enlarge x limits=false,
      enlarge y limits=false,
      legend style={
        at={(0.97,0.05)},
        anchor=south east,
        draw=black!50,
        line width=0.8pt,
        fill=white,
        fill opacity=0.92,
        text opacity=1,
        font=\small,
        row sep=3pt,
        inner sep=7pt,
        rounded corners=2pt
      },
      legend cell align={left},
      legend image post style={line width=2.0pt}
    ]
    \addplot[
      color=refM3, line width=1.6pt,
      each nth point=8, filter discard warning=false,
    ] table [x=time_s, y=paper_M3_3494400_elements_downward_speed_cm_s, col sep=comma]
      {data/falling_sphere_rigid_paper_comparison_data.csv};
    \addlegendentry{Black et al.~\citep{Black2026ImmersedFSI}: 3{,}494{,}400 elements}

    \addplot[
      color=nbBlue, line width=2.2pt,
      each nth point=8, filter discard warning=false,
    ] table [x=time_s, y=rigid_iga_108_elements_downward_speed_cm_s, col sep=comma]
      {data/falling_sphere_rigid_paper_comparison_data.csv};
    \addlegendentry{Present IGA-ALE/preCICE: 108 elements}

    \addplot[
      color=black!65, dash pattern=on 8pt off 4pt, line width=1.3pt,
      domain=0:0.1, samples=2,
    ] {0.5048};
    \addlegendentry{Wall-corrected analytical (0.5048\,cm/s)}
    \end{axis}
\end{tikzpicture}
    \caption{Two-way coupled free-fall test: downward sphere speed obtained with the present IGA-ALE/preCICE formulation, the finest mesh level reported by Black et al.~\citep{Black2026ImmersedFSI}, and the wall-corrected analytical terminal speed from \eqref{eq:falling-sphere-terminal-velocity}. The present discretization uses 108 volume elements with a Taylor--Hood $Q_3/Q_2$ pair and 4052 global free degrees of freedom (3660 velocity and 392 pressure).}
    \label{fig:free-falling-sphere-velocity}
\end{figure}

As shown in \Cref{fig:free-falling-sphere-velocity}, the sphere approaches a terminal speed of $0.5063\,\mathrm{cm/s}$. The relative difference from the wall-corrected value is $0.29\%$, while the difference from the final value of the finest reference curve is approximately $3.4\%$. For comparison, Black et al.~report a $3.02\%$ difference between their finest-grid result and the analytical value. Throughout the 1000 coupled time steps, the minimum scaled Jacobian $\mathcal{J}_{\min}$ of the volumetric parameterization remains in the narrow interval $[0.138429,0.138617]$. This second test therefore demonstrates that the present formulation yields accurate results under two-way coupling, not only mesh validity: the prescribed test isolates parameterization robustness over a long trajectory, while the free-fall test confirms that the same per-step parameterization remains both stable and accurate when the trajectory is determined by the coupled fluid force.

\subsection{Flow past a rotating square}
\label{sec:flow-past-square}

In the kinematic study of \Cref{sec:mesh-quality} we mapped out the admissible rotation range of the mesh-update strategies. We now return to the same geometry, but couple it to the incompressible Navier--Stokes equations. To validate the present method against an independent reference, we reproduce the rotating-square benchmark of Schalk et al.~\citep{mesh_regeneration_Graz}, whose block-structured meshing scheme handles the same connectivity-changing rotation problem in a finite element setting. The benchmark parameters and reference kinetic-energy and momentum histories were obtained directly from the authors and are used here for comparison. The simulation covers one full revolution of the inner square, during which the tangential-slip reparameterization of \Cref{sec:closed-domain} remains active throughout.

\begin{figure}[H]
    \centering
  \begin{tikzpicture}[x=0.5cm,y=0.5cm]
    \definecolor{nbBlue}{RGB}{31,119,180}
    \definecolor{patchOne}{RGB}{198,222,240}
    \definecolor{patchTwo}{RGB}{171,204,232}
    \definecolor{patchThree}{RGB}{140,183,222}
    \definecolor{patchFour}{RGB}{99,153,201}
    \definecolor{patchFive}{RGB}{224,238,248}

    \def\Lch{25}   
    \def\Hch{10}   
    \def\xc{8.3333} 
    \def\yc{5}       
    \def\xb{16.6667} 

    \coordinate (BL) at (0,0);
    \coordinate (BR) at (\xb,0);
    \coordinate (TR) at (\xb,\Hch);
    \coordinate (TL) at (0,\Hch);

    \coordinate (SBL) at (7.3333,4);
    \coordinate (SBR) at (9.3333,4);
    \coordinate (STR) at (9.3333,6);
    \coordinate (STL) at (7.3333,6);

    \fill[patchOne]   (SBL) -- (SBR) -- (BR) -- (BL) -- cycle;
    \fill[patchTwo]   (SBR) -- (STR) -- (TR) -- (BR) -- cycle;
    \fill[patchThree] (STL) -- (STR) -- (TR) -- (TL) -- cycle;
    \fill[patchFour]  (SBL) -- (STL) -- (TL) -- (BL) -- cycle;

    \fill[patchFive] (\xb,0) rectangle (\Lch,\Hch);

    \draw[gray!70, thin] (SBL) -- (BL);
    \draw[gray!70, thin] (SBR) -- (BR);
    \draw[gray!70, thin] (STR) -- (TR);
    \draw[gray!70, thin] (STL) -- (TL);
    \draw[black!85, dashed, line width=0.8pt] (\xb,0) -- (\xb,\Hch);

    \draw[black!85, line width=1.2pt] (0,0) rectangle (\Lch,\Hch);

    \draw[nbBlue!80!black, line width=1.4pt, pattern=north east lines, pattern color=nbBlue!60]
        (SBL) rectangle (STR);

    \draw[red!75!black, line width=1.1pt, -{Latex[length=2mm]}]
        (\xc,\yc) ++(0.62,0) arc[start angle=0, end angle=300, radius=0.62];
    \node[red!75!black, font=\small] at (\xc,\yc+1.55) {$\dot\theta$};

    \node[font=\small] at (4.2,1.4)   {$\Omega_1$};
    \node[font=\small] at (13.0,3.0)  {$\Omega_2$};
    \node[font=\small] at (4.2,8.6)   {$\Omega_3$};
    \node[font=\small] at (2.2,5.0)   {$\Omega_4$};
    \node[font=\small] at (21.0,5.0)  {$\Omega_5$};

    \foreach \y in {1.5,3.5,5,6.5,8.5}{
        \draw[green!45!black, -{Latex[length=2mm]}, line width=0.9pt] (-2.4,\y) -- (-0.3,\y);
    }
    \node[green!45!black, font=\small] at (-3.4,5) {$u_{\mathrm{in}}$};
    \node[green!45!black, font=\small] at (-1.4,-1.1) {inflow};

    \foreach \y in {1.5,3.5,5,6.5,8.5}{
        \draw[orange!85!black, -{Latex[length=2mm]}, line width=0.9pt] (\Lch+0.3,\y) -- (\Lch+2.4,\y);
    }
    \node[orange!85!black, font=\small] at (\Lch+3.6,5) {outflow};

    \node[font=\small] at (\Lch/2,-1.1) {no-slip wall: $\mathbf{u}=\mathbf{0}$};
    \node[font=\small] at (\Lch/2,\Hch+1.1) {no-slip wall: $\mathbf{u}=\mathbf{0}$};

    \draw[{Latex[length=1.8mm]}-{Latex[length=1.8mm]}] (0,\Hch+2.3) -- (\xb,\Hch+2.3) node[midway, above, font=\small] {$L/3$ from inlet};
    \draw[{Latex[length=1.8mm]}-{Latex[length=1.8mm]}] (\xb,\Hch+2.3) -- (\Lch,\Hch+2.3) node[midway, above, font=\small] {remainder};

    \draw[{Latex[length=1.8mm]}-{Latex[length=1.8mm]}] (-4.4,0) -- (-4.4,\Hch) node[midway, left, font=\small, rotate=90, anchor=south] {$H=10\,\mathrm{m}$};

    \draw[{Latex[length=1.6mm]}-{Latex[length=1.6mm]}] (7.3333,3.4) -- (9.3333,3.4) node[midway, below, font=\small] {$2\,\mathrm{m}$};

    \draw[gray!60, thin] (SBR) -- ++(2.2,-1.8);
    \node[font=\small, anchor=north west] at (9.6,2.4) {rotating body: $\mathbf{u}=\dot\theta\,\mathbf{e}_z\times(\mathbf{x}-\mathbf{x}_c)$};

    \node[draw, black!70, fill=white, align=left, font=\small, inner sep=7pt,
          rounded corners=1pt, text width=11.2cm, anchor=north]
        at (\Lch/2,-2.3) {%
            $L=25\,\mathrm{m}$, $H=10\,\mathrm{m}$, square side $=2\,\mathrm{m}$, $\mathbf{x}_c=(L/3,\,H/2)$ \\
            $u_{\mathrm{in}}=2\,\mathrm{m/s}$ (centerline), $\rho^f=1.0\,\mathrm{kg/m^3}$, $\nu^f=0.05\,\mathrm{m^2/s}$, $\dot\theta=2^\circ/\mathrm{s}$
        };
\end{tikzpicture}
    \caption{Rotating-square benchmark setup, following Schalk et al.~\citep{mesh_regeneration_Graz}: a channel of length $L=25\,\mathrm{m}$ and height $H=10\,\mathrm{m}$ contains a rigid square of side $2\,\mathrm{m}$ centered at $\mathbf{x}_c=(L/3,H/2)$. The region containing the square is decomposed into four body-fitted patches $\Omega_1,\dots,\Omega_4$, flush with the inlet and centered on the square, and joins a fixed downstream outflow extension $\Omega_5$.}
    \label{fig:rotating-setup}
\end{figure}

\paragraph{Problem setup}
We consider two-dimensional incompressible flow past a rigid square obstacle of side length $2\,\mathrm{m}$, rotating about its center $\mathbf{x}_c=(L/3,H/2)$ with prescribed angular velocity $\dot{\theta} = 2^\circ/\mathrm{s}$. The physical domain, shown in \Cref{fig:rotating-setup}, is a channel of length $L=25\,\mathrm{m}$ and height $H=10\,\mathrm{m}$; the region containing the square is decomposed into four body-fitted patches flush with the inlet and centered on the square, joined to a fixed downstream outflow extension.

A parabolic inflow profile with centerline velocity $u_{\mathrm{in}} = 2\,\mathrm{m/s}$ is prescribed at the inlet. No-slip conditions are imposed on the upper and lower walls, and a stress-free condition is applied at the outlet. On the square boundary, the velocity follows the rigid-body rotation
\begin{equation}
  \mathbf{v} = \dot{\theta}\,\mathbf{e}_z \times (\mathbf{x} - \mathbf{x}_c)
             = \bigl(-\dot{\theta}(y - y_c),\; \dot{\theta}(x - x_c)\bigr)^\top.
\end{equation}
The fluid density is $\rho^f=1.0\,\mathrm{kg/m^3}$ and the kinematic viscosity is $\nu^f = 0.05\,\mathrm{m^2/s}$.

The fluid domain is decomposed into five patches: four body-fitted patches surrounding the square and one fixed outflow patch. Each patch is parameterized by the barrier-patch method; the spline space is degree-elevated once and uniformly refined four times, with one additional refinement in the obstacle-normal direction on the four near-body patches, as detailed below. This yields 4864 elements and 49,543 mixed velocity--pressure degrees of freedom for a Taylor--Hood $Q_4/Q_3$ spline pair in the present computations, i.e.\ tensor-product degree $(4,4)$ for velocity and $(3,3)$ for pressure. For reference, Schalk et al.~\citep{mesh_regeneration_Graz} report 11,136 quadrilateral elements with a Taylor--Hood $Q_2/Q_1$ discretization. Since their global DoF count is not explicitly listed, we report a rough estimate based on standard structured-mesh scaling for $Q_2/Q_1$ in 2D, \(\text{DoFs}\approx 9N_e\), which gives about \(9\times 11{,}136 \approx 1.00\times10^5\) mixed velocity--pressure DoFs. We solve in G+Smo \citep{Juettler2014Gismo, Gismo2025Zenodo, Hana_paper_gsIncompressibleFlow} with $\Delta t=0.005\,\mathrm{s}$ over one revolution ($\theta\in[0^\circ,360^\circ]$, $T=180\,\mathrm{s}$).

\begin{figure}[H]
    \centering
    \includegraphics[width=0.50\linewidth]{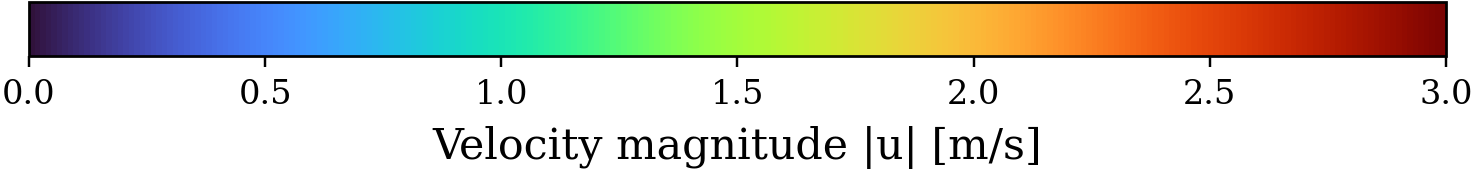}\\
    \subfigure[Velocity, $\theta = 60^\circ$.]{
      \includegraphics[width=0.45\linewidth]{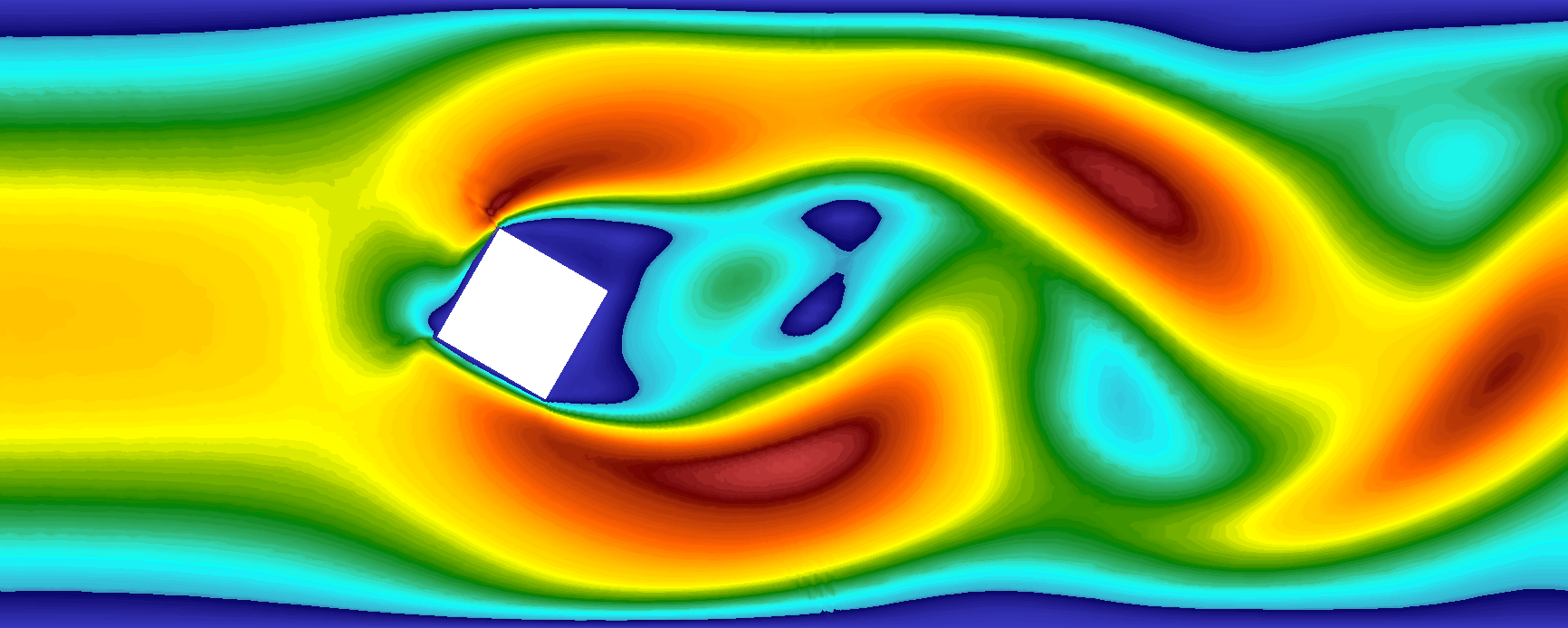}
    }
    \quad
    \subfigure[Mesh, $\theta = 60^\circ$.]{
      \includegraphics[width=0.45\linewidth]{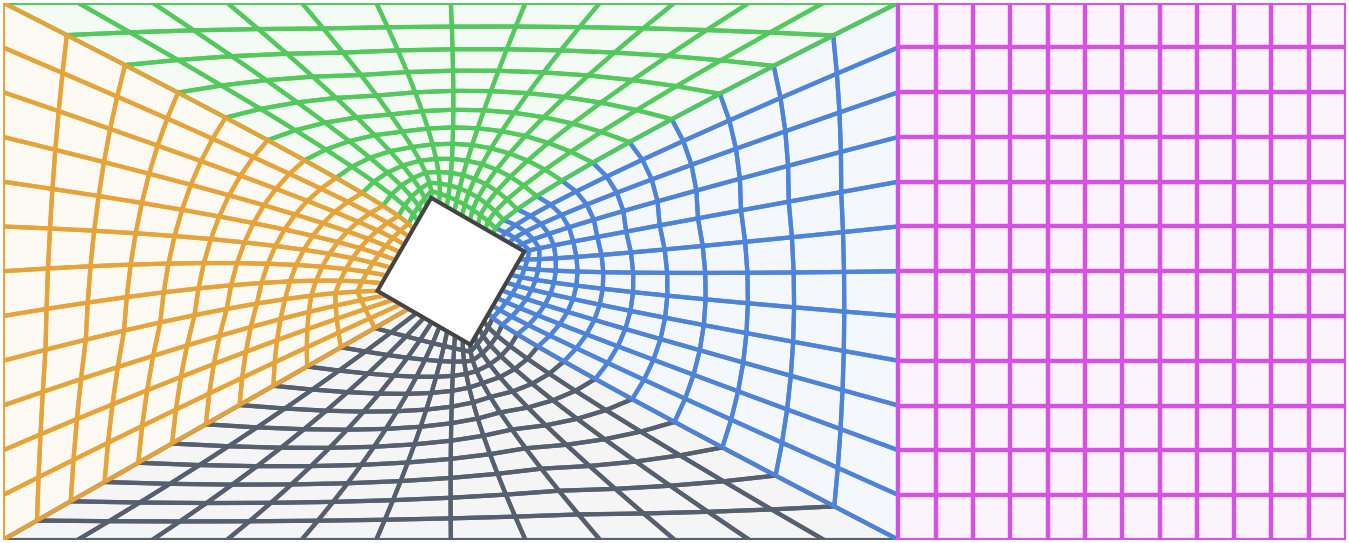}
    }\\
    \subfigure[Velocity, $\theta = 155^\circ$.]{
      \includegraphics[width=0.45\linewidth]{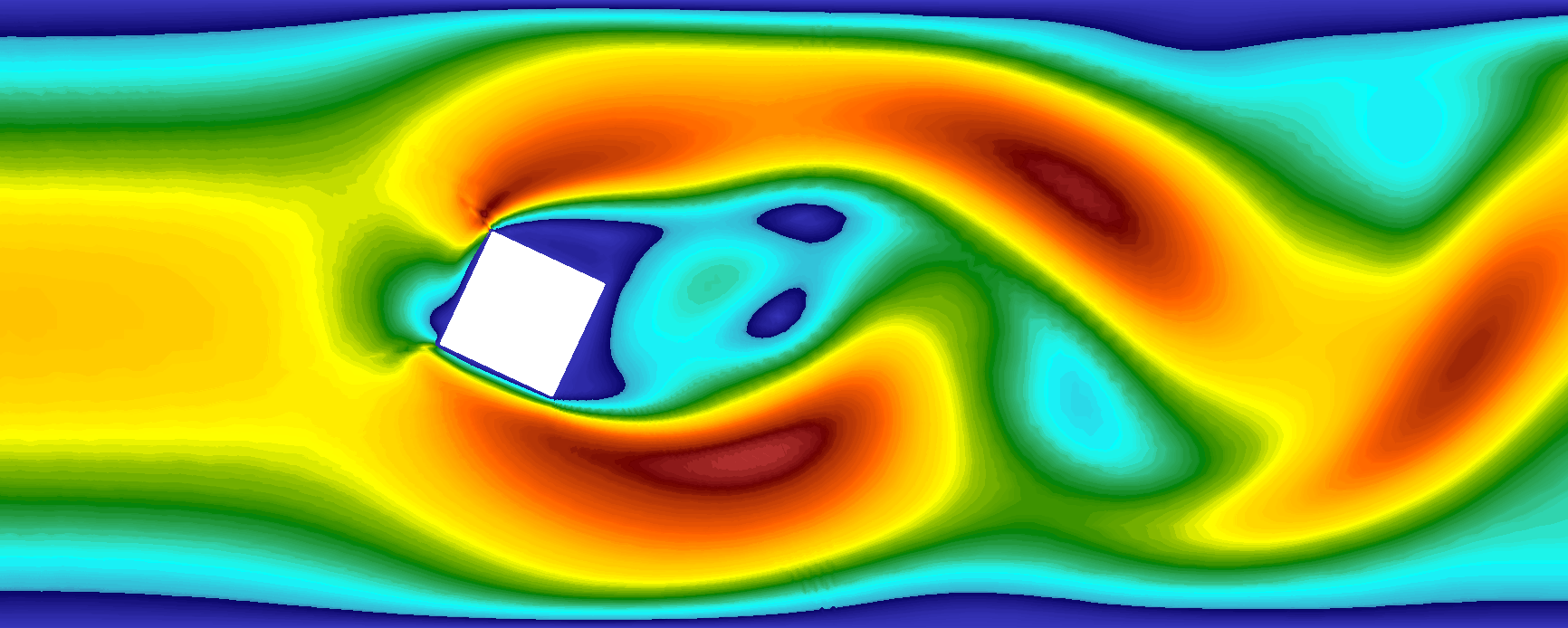}
    }
    \subfigure[Mesh, $\theta = 155^\circ$.]{
      \includegraphics[width=0.45\linewidth]{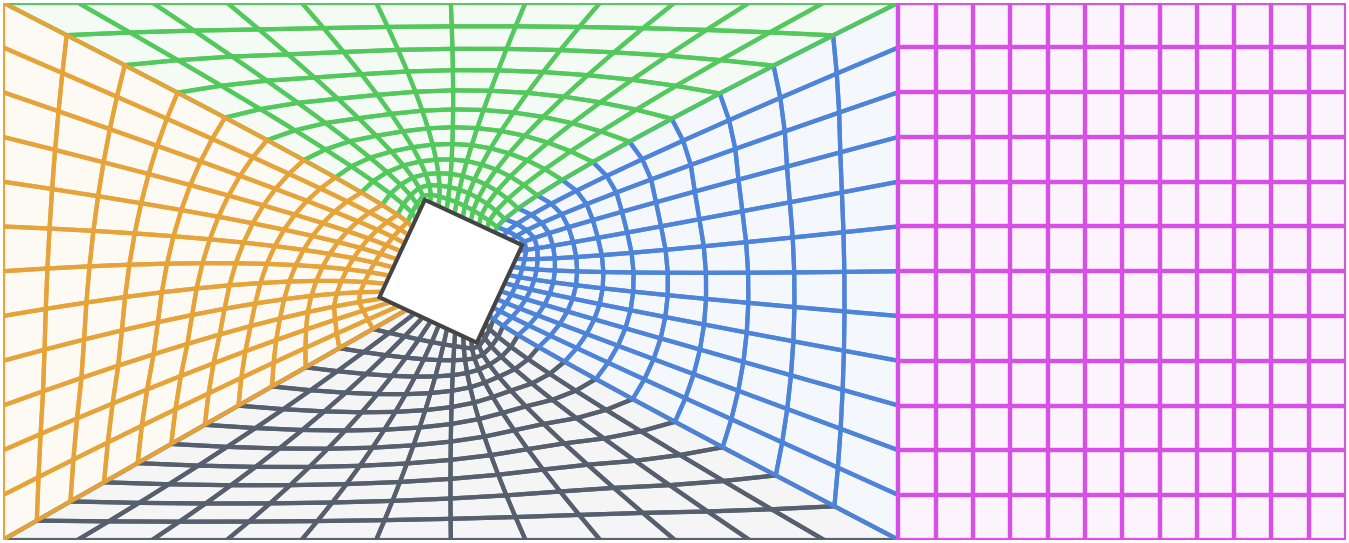}
    }\\
    \subfigure[Velocity, $\theta = 295^\circ$.]{
      \includegraphics[width=0.45\linewidth]{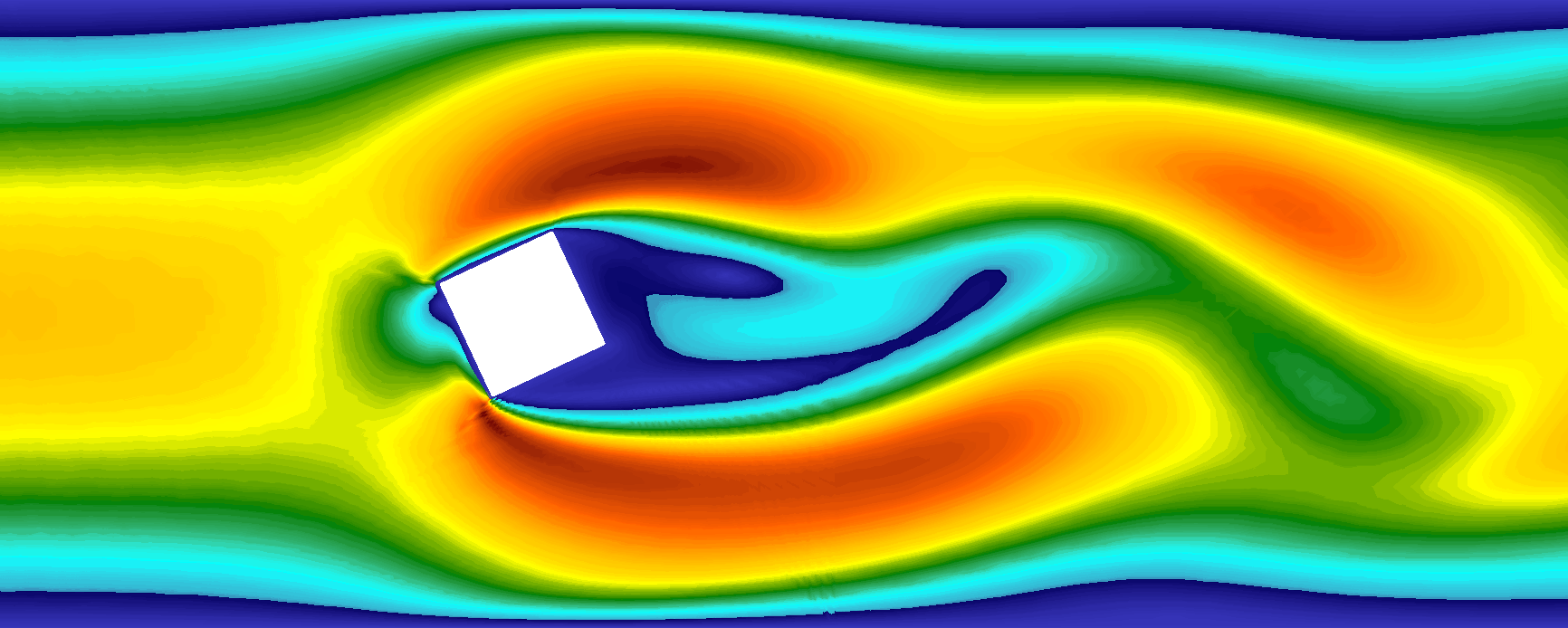}
    }
    \subfigure[Mesh, $\theta = 295^\circ$.]{
      \includegraphics[width=0.45\linewidth]{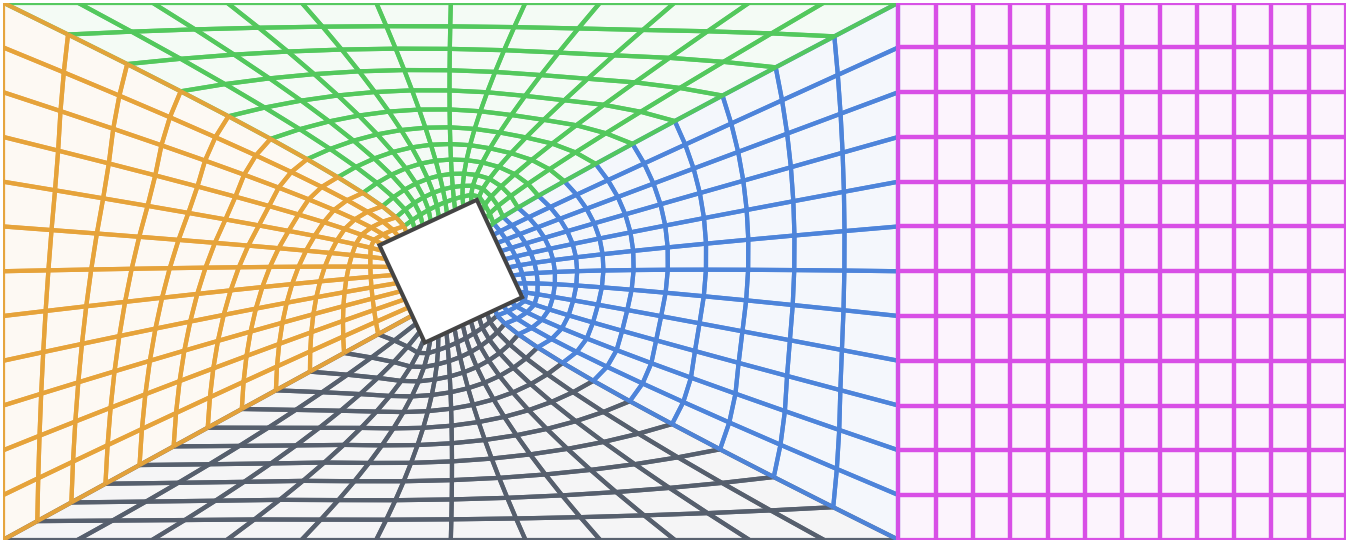}
    }\\
    \subfigure[Velocity, $\theta = 325^\circ$.]{
      \includegraphics[width=0.45\linewidth]{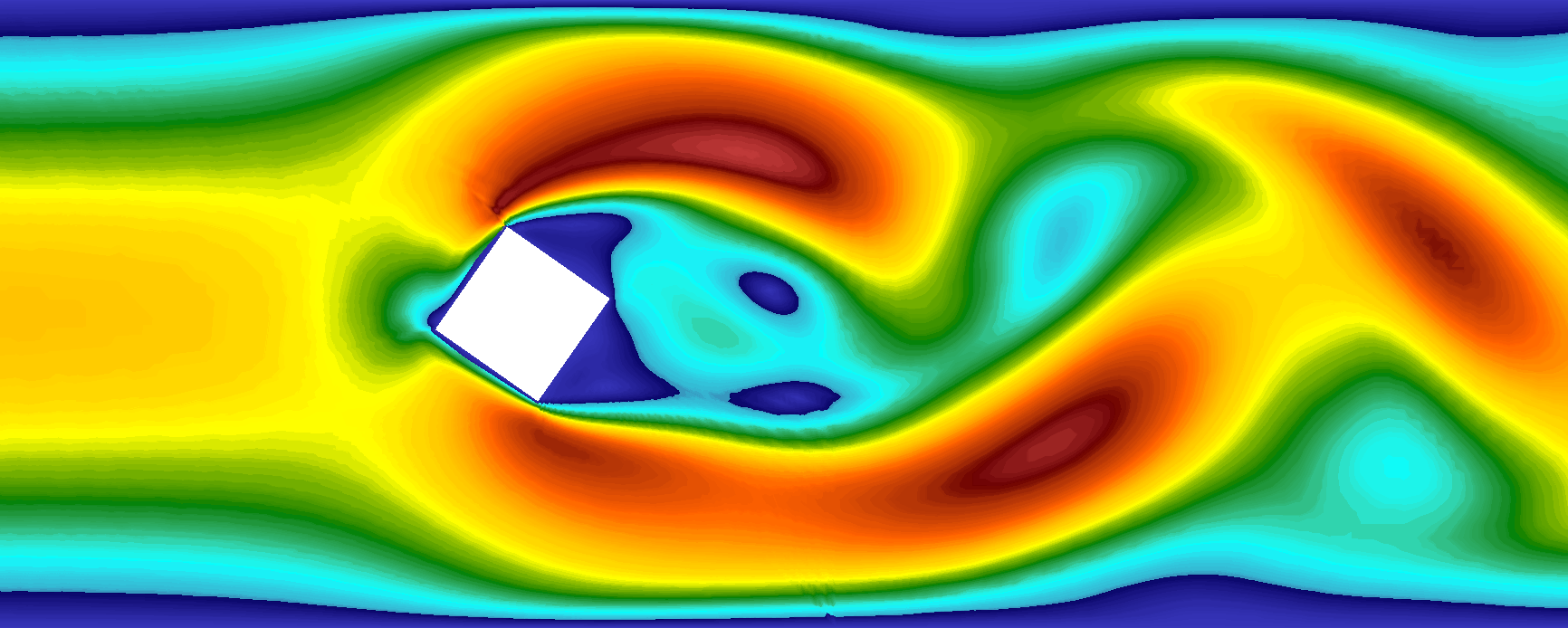}
    }
    \subfigure[Mesh, $\theta = 325^\circ$.]{
      \includegraphics[width=0.45\linewidth]{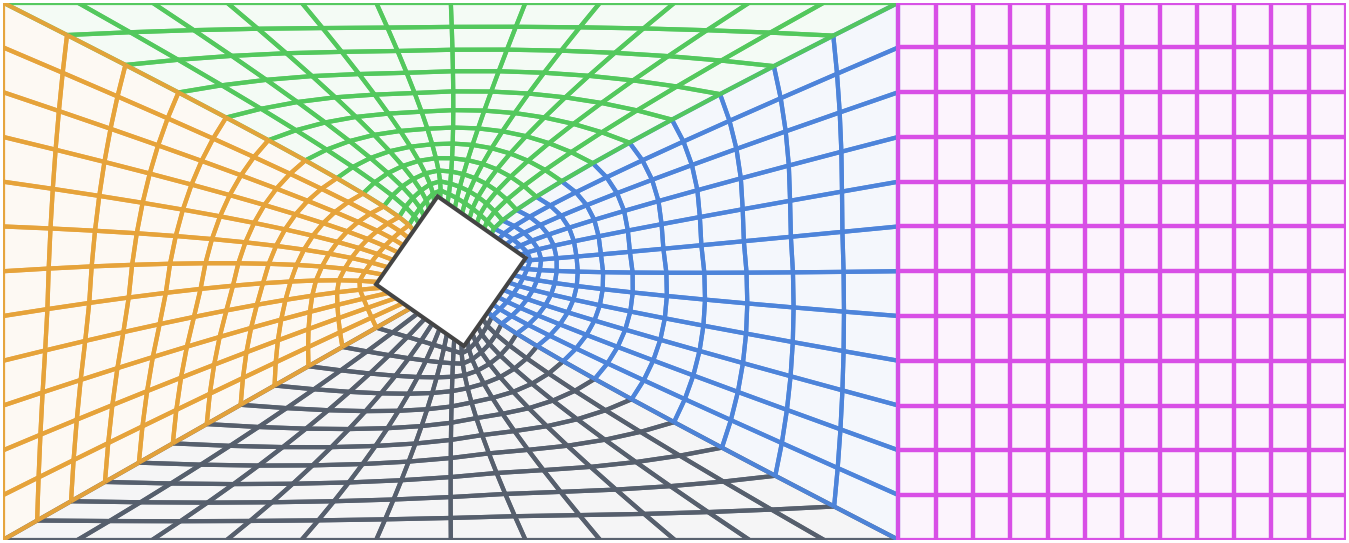}
    }
    \caption{Velocity field (left) and regenerated ALE mesh (right) at selected rotation angles. The two velocity columns share the common color bar shown at the top. The tangential-slip strategy sustains valid and well-shaped patches far beyond the rotation limit of the plain barrier-patch method.}
    \label{fig:rotating-square-snapshots}
\end{figure}

\paragraph{Validation}
\Cref{fig:graz-energy-momentum} compares the present solution against the reference kinetic-energy and momentum histories of Schalk et al.~\citep{mesh_regeneration_Graz}. \Cref{fig:graz-ek} shows the kinetic energy $E_k$ of the flow over a full revolution of the square; the two curves overlap closely throughout, including the transient during the first quarter-revolution and the subsequent periodic regime. \Cref{fig:graz-py} zooms into the vertical linear momentum $P_y$ over $\theta \in [45^\circ,160^\circ]$. Both curves show the same beat pattern, in which the oscillation amplitude grows and decays over successive cycles, while remaining in phase. \Cref{fig:graz-freq-table} reports the frequency $f=\dot{\theta}/(2\overline{\Delta\theta}_{\mathrm{sw}})$, where $\overline{\Delta\theta}_{\mathrm{sw}}$ is the mean angular interval between consecutive sign changes of $P_y$ on $\theta\in[45^\circ,180^\circ]$, the start-up transient of the first quarter-revolution being excluded from the average, and the peak magnitude $A=\max_{\theta\in[0^\circ,360^\circ]}|P_y(\theta)|$. The two frequencies agree to within $0.43\%$ and the amplitudes to within $0.94\%$; the frequency estimate varies by less than $0.01$ percentage points when the start of the averaging window is moved between $40^\circ$ and $60^\circ$. Quantitatively, the kinetic-energy histories of \Cref{fig:graz-ek} agree to $0.18\%$ in the mean and $0.25\%$ in relative $L^2$ over the full revolution, with a maximum pointwise deviation of $0.96\%$. The residual differences are consistent with the small phase and magnitude offsets visible in \Cref{fig:graz-py}.

\begin{figure}[H]
    \centering
  \noindent\makebox[\linewidth][c]{\pgfplotslegendfromname{grazsharedlegend}}
    \vspace{-2pt}

    \begin{minipage}[t]{0.58\linewidth}
        \centering
        \subfigure[Kinetic energy $E_k$ over one full revolution ($\theta \in {[0^\circ,360^\circ]}$).\label{fig:graz-ek}]{%
            \begin{minipage}[t]{\linewidth}
                \centering
                \begin{tikzpicture}
                \begin{axis}[
                  width=\linewidth,
                  height=0.34\textheight,
                  xlabel={Rotation angle $\theta$ [deg]},
                  ylabel={Kinetic energy $E_k$},
                  xlabel style={font=\normalsize, yshift=-3pt},
                  ylabel style={font=\normalsize, yshift=6pt},
                  tick label style={font=\small},
                  xmin=0, xmax=360,
                  xtick distance=90,
                  axis line style={line width=1.2pt, black!85},
                  tick style={line width=1.0pt, black!85},
                  major tick length=5pt,
                  minor tick length=3pt,
                  grid=both,
                  major grid style={line width=0.6pt, draw=gray!30},
                  minor grid style={line width=0.4pt, draw=gray!15, dashed},
                  minor tick num=1,
                  enlarge x limits=false,
                  enlarge y limits=0.05,
                  legend to name=grazsharedlegend,
                  legend style={
                    draw=black!50, line width=0.8pt,
                    fill=white, fill opacity=0.92, text opacity=1,
                    font=\scriptsize, row sep=2pt, inner sep=5pt, rounded corners=2pt
                  },
                  legend columns=1,
                  legend cell align={left},
                  legend image post style={line width=2.4pt}
                ]
                \addplot[color=blue!70!black, line width=2.4pt, each nth point=8]
                  table [header=false, col sep=tab, x expr=0.01*\thisrowno{0}, y index=1]
                  {TableEnergyMomentum.txt};
                \addlegendentry{Schalk et al. (Taylor--Hood \(Q_2/Q_1\), 11,136 quads, \(\sim\!1.00\times10^5\) est. mixed DoFs)}
                \addplot[color=orange!85!black, dash pattern=on 8pt off 4pt, line width=2.2pt, each nth point=8]
                  table [x=angle_deg, y=kinetic_energy, col sep=comma]
                  {original_gismo_energy_momentum.csv};
                \addlegendentry{Present (Taylor--Hood \(Q_4/Q_3\), 4864 elements, \(4.9543\times10^4\) mixed DoFs)}
                \end{axis}
                \end{tikzpicture}
            \end{minipage}%
        }
    \end{minipage}
    \hfill
    \begin{minipage}[t]{0.39\linewidth}
        \centering
        \subfigure[Zoom in $P_y$ for $\theta \in {[45^\circ,160^\circ]}$.\label{fig:graz-py}]{%
            \begin{minipage}[t]{\linewidth}
                \centering
                \begin{tikzpicture}
                \begin{axis}[
                  width=\linewidth,
                  height=0.18\textheight,
                  xlabel={Rotation angle $\theta$ [deg]},
                  ylabel={$P_y$},
                  xlabel style={font=\small, yshift=-2pt},
                  ylabel style={font=\small, yshift=4pt},
                  tick label style={font=\scriptsize},
                  xmin=45, xmax=160,
                  ymin=-10, ymax=10,
                  xtick distance=45,
                  axis line style={line width=1.0pt, black!85},
                  tick style={line width=0.8pt, black!85},
                  grid=both,
                  major grid style={line width=0.5pt, draw=gray!30},
                  minor grid style={line width=0.3pt, draw=gray!15, dashed},
                  enlarge x limits=false,
                ]
                \addplot[color=blue!70!black, line width=1.4pt, each nth point=4]
                  table [x=angle_deg, y=graz_py, col sep=comma]
                  {live_graz_comparison.csv};
                \addplot[color=orange!85!black, dash pattern=on 6pt off 3pt, line width=1.3pt, each nth point=4]
                  table [x=angle_deg, y=momentum_y, col sep=comma]
                  {original_gismo_energy_momentum.csv};
                \end{axis}
                \end{tikzpicture}
            \end{minipage}%
        }

        \vspace{8pt}
  {\small (c) Frequency and amplitude of $P_y$.\label{fig:graz-freq-table}}

  \vspace{2pt}
        \small
        
        \begin{tabular}{lSS}
        \toprule
        {} & {Frequency [Hz]} & {Peak $|P_y|$} \\
        \midrule
        Schalk et al.~\citep{mesh_regeneration_Graz} & 0.1468 & 9.815 \\
        Present & 0.1462 & 9.907 \\
        \midrule
        Relative diff. [\%] & 0.43 & 0.94 \\
        \bottomrule
        \end{tabular}
    \end{minipage}

      \caption{Validation against the reference solution of Schalk et al.~\citep{mesh_regeneration_Graz} for the rotating-square benchmark. (a) Kinetic energy $E_k$ of the flow over one full revolution. (b) Zoom in of vertical momentum $P_y$ over $\theta \in [45^\circ,160^\circ]$. (c) Corresponding frequency and peak magnitude of $P_y$.}
    \label{fig:graz-energy-momentum}
\end{figure}

\paragraph{Scaled-Jacobian history under sustained rotation}
We now return to the diagnostic deferred from \Cref{sec:mesh-quality} and evaluate $\mathcal{J}_{\min}$ of~\eqref{eq:jmin-monitor} on the fluid parameterization throughout a full revolution of the inner square. \Cref{fig:rotating-square-jmin-long} traces the resulting evolution against the rotation angle~$\theta$.

Two observations stand out. First, $\mathcal{J}_{\min}$ stays strictly bounded away from zero throughout the entire rotation, so no mesh degeneration is encountered. The minimum value $\mathcal{J}_{\min} = 0.6085$ is attained at $\theta = 94.5^\circ$, and the maximum $\mathcal{J}_{\min} = 0.7048$ at $\theta = 42.75^\circ$; both lie well above the invalidity threshold $\mathcal{J}_{\min} = 0$, confirming that the regeneration algorithm satisfies the bijectivity condition of~\eqref{eq:jmin-monitor}. Second, the trajectory exhibits a clear periodic pattern synchronized with the rotation, reflecting the $90^\circ$ rotational symmetry of the four-patch inner decomposition. The four quarter-revolution cycles are in fact indistinguishable to four significant digits, each attaining the same extrema $\mathcal{J}_{\min}=0.6085$ and $\mathcal{J}_{\min}=0.7048$, so mesh quality does not drift as the rotation accumulates. This is expected rather than fortuitous: each parameterization is rebuilt from the current square orientation alone, so configurations that differ by a multiple of $90^\circ$ produce the same parameterization up to the symmetry of the patch layout, irrespective of how many revolutions preceded them.

\pgfplotstableread[col sep=comma]{mesh_quality_0_360.csv}\msjTbl
\pgfplotstablecreatecol[create col/expr={
    \thisrow{angle_deg} < 60 ? 99 : \thisrow{minScaledJacDet}
}]{msj_ex1}\msjTbl
\pgfplotstablesort[sort key={msj_ex1}, sort cmp={float <}]
    \msjTblAsc{\msjTbl}
\pgfplotstablegetelem{0}{angle_deg}\of\msjTblAsc            \let\minAngle\pgfplotsretval
\pgfplotstablegetelem{0}{minScaledJacDet}\of\msjTblAsc  \let\minJ\pgfplotsretval
\pgfplotstablesort[sort key={minScaledJacDet}, sort cmp={float >}]
    \msjTblDesc{\msjTbl}
\pgfplotstablegetelem{0}{angle_deg}\of\msjTblDesc           \let\maxAngle\pgfplotsretval
\pgfplotstablegetelem{0}{minScaledJacDet}\of\msjTblDesc \let\maxJ\pgfplotsretval

\begin{figure}[H]
    \centering
    \begin{tikzpicture}
        \definecolor{cP2}{RGB}{0,90,156}
        \definecolor{cP2light}{RGB}{173,206,232}
        \definecolor{cMin}{RGB}{192,57,43}
        \definecolor{cMax}{RGB}{39,174,96}
        \begin{axis}[
          width=0.85\linewidth,
          height=0.38\textheight,
          xlabel={Rotation angle $\theta$ [deg]},
          ylabel={$\mathcal{J}_{\min}$},
          xlabel style={font=\normalsize\bfseries, yshift=-4pt},
          ylabel style={font=\normalsize\bfseries, yshift=8pt},
          tick label style={font=\small},
          xmin=0, xmax=360,
          ymin=0.58, ymax=0.72,
          xtick distance=90,
          ytick distance=0.02,
          axis line style={line width=1.6pt, black},
          tick style={line width=1.3pt, black},
          major tick length=6pt,
          minor tick length=3pt,
          grid=both,
          major grid style={line width=0.7pt, draw=gray!35, opacity=0.6},
          minor grid style={line width=0.4pt, draw=gray!20, opacity=0.5, dashed},
          minor tick num=1,
          enlarge x limits=false,
          enlarge y limits=false,
          clip=false,
        ]
        \addplot[
          draw=none, name path=curve,
        ] table [x=angle_deg, y=minScaledJacDet, col sep=comma]
          {mesh_quality_0_360.csv};

        \path[name path=top] (axis cs:0,0.72) -- (axis cs:360,0.72);

        \addplot[cP2light, opacity=0.3] fill between[of=curve and top];

        \addplot[
          cP2, line width=2.2pt,
          mark=*, mark size=1.8pt,
          mark options={fill=cP2, draw=white, line width=0.6pt},
          mark repeat=2,
        ] table [x=angle_deg, y=minScaledJacDet, col sep=comma]
          {mesh_quality_0_360.csv};
        
        \node[circle, draw=cMin, line width=1.8pt, minimum size=12pt,
              inner sep=0pt, fill=white, fill opacity=0.3] 
              at (axis cs:\minAngle,\minJ) {};
        \node[circle, fill=cMin, minimum size=4pt, inner sep=0pt] 
              at (axis cs:\minAngle,\minJ) {};
        \node[font=\footnotesize, anchor=west, align=left,
              fill=white, fill opacity=0.95, text opacity=1,
              draw=cMin, line width=1.0pt, 
              inner sep=5pt, rounded corners=2.5pt,
              drop shadow={shadow xshift=1pt, shadow yshift=-1pt, opacity=0.2}]
              at ($(axis cs:\minAngle,\minJ)+(12pt,18pt)$)
              {\textbf{Global Min} \\[1pt]
               $\mathcal{J}_{\min} = \pgfmathprintnumber[fixed,precision=4]{\minJ}$ \\
               $\theta = \pgfmathprintnumber[fixed,precision=1]{\minAngle}^\circ$};
        
        \node[circle, draw=cMax, line width=1.8pt, minimum size=12pt,
              inner sep=0pt, fill=white, fill opacity=0.3] 
              at (axis cs:\maxAngle,\maxJ) {};
        \node[circle, fill=cMax, minimum size=4pt, inner sep=0pt] 
              at (axis cs:\maxAngle,\maxJ) {};
        \node[font=\footnotesize, anchor=west, align=left,
              fill=white, fill opacity=0.95, text opacity=1,
              draw=cMax, line width=1.0pt,
              inner sep=5pt, rounded corners=2.5pt,
              drop shadow={shadow xshift=1pt, shadow yshift=-1pt, opacity=0.2}]
              at ($(axis cs:\maxAngle,\maxJ)+(12pt,12pt)$)
              {\textbf{Global Max} \\[1pt]
               $\mathcal{J}_{\min} = \pgfmathprintnumber[fixed,precision=4]{\maxJ}$ \\
               $\theta = \pgfmathprintnumber[fixed,precision=1]{\maxAngle}^\circ$};
        
        \addplot[
          dashed, line width=0.9pt, gray!70,
          domain=0:360, samples=2,
        ] {0.6};
        
        \end{axis}
    \end{tikzpicture}
    \caption{Minimum scaled Jacobian $\mathcal{J}_{\min}$ of the ALE fluid parameterization over one full revolution of the inner square. The red and green circles mark the global minimum and maximum, respectively.}
    \label{fig:rotating-square-jmin-long}
\end{figure}

\paragraph{Cost of tangential slip}
A practical question is whether rebuilding the spline parameterization at every time step introduces a significant computational overhead. To quantify this cost, we measure the cumulative mesh-only regeneration time over 200 evenly spaced configurations across eight successive uniform refinement levels ($r = 0, 1, \ldots, 7$) and for polynomial degrees $p = 2$ and $p = 3$. For each refinement level, the $p=2$ and $p=3$ cases use the same multi-patch layout and the same element topology. Only the spline degree of the parameterization is changed. The reported time includes tangential slip, Barrier Patch, degree elevation, uniform refinement, and topology construction. A higher spline degree increases both the number of quadrature points and the number of active basis functions involved in each local evaluation, even when the number of elements is unchanged.

\Cref{fig:remesh-time} summarizes the results. For coarse to moderately refined meshes ($r \leq 4$), the cumulative mesh-only cost remains between $1.92\,\mathrm{s}$ and $2.41\,\mathrm{s}$ for $p=2$, and between $1.96\,\mathrm{s}$ and $2.43\,\mathrm{s}$ for $p=3$. Beyond this range, the cost grows gradually: at $r=6$, it reaches $5.38\,\mathrm{s}$ for $p=2$ and $6.27\,\mathrm{s}$ for $p=3$, and at the finest level $r=7$, it rises to $14.95\,\mathrm{s}$ and $21.20\,\mathrm{s}$, respectively.

The levels $r$ in \Cref{fig:remesh-time} are obtained by uniformly refining all patches of the layout and therefore form a scaling study rather than a sequence of production runs. The comparison reported above in \Cref{fig:graz-energy-momentum} instead uses a directionally refined configuration. After four uniform refinements, the four near-body patches receive one additional refinement in the obstacle-normal direction, while no additional tangential refinement is applied. The base discretization comprises one element per near-body patch and eleven elements in the outflow patch, fifteen in total, which is the $r=0$ entry of \Cref{fig:remesh-time}. Consequently, each near-body patch contains $2^4\times2^5$ elements and the outflow patch contains $11\times2^4\times2^4$ elements. The total number of elements is therefore
\[N=4\left(2^4\times2^5\right)+11\left(2^4\times2^4\right)=2048+2816=4864.\]
This configuration lies between the uniformly refined levels $r=4$ and $r=5$ of the sweep. The two finest levels are included primarily as stress tests of the regeneration algorithm: in the IGA setting, the high-order and high-continuity B-spline/NURBS basis provides a higher approximation capability per element than low-order $C^0$ finite elements, so resolutions of this order are well beyond what the present benchmark requires. Across all levels considered, the mesh-only regeneration cost per configuration stays at or below approximately $0.11\,\mathrm{s}$, so the parameterization procedure remains computationally affordable relative to the full flow and structural solves.

\begin{figure}[H]
    \centering
    \begin{tikzpicture}
    \begin{axis}[
                width=0.95\linewidth, height=7.5cm,
                ybar=1pt,
        bar width=16pt,
                enlarge x limits=0.06,
                ymin=0, ymax=24,
        xtick=data,
                xticklabels={%
            {$0$ (15)},{$1$ (60)},{$2$ (240)},{$3$ (960)},%
            {$4$ (3840)},{$5$ (15360)},{$6$ (61440)},{$7$ (245760)}},
                xticklabel style={font=\small},
                xlabel={Refinement level $r$ (number of elements)},
                ylabel={Cumulative mesh-only regeneration time [s]},
                legend style={at={(0.02,0.98)}, anchor=north west,
                                            draw=black, fill=white},
        xticklabel style={font=\small},
        ymajorgrids,
        grid style={dashed, gray!40, line width=0.4pt},
        axis y line*=left, axis x line*=bottom,
        nodes near coords,
                nodes near coords style={font=\tiny,
            /pgf/number format/fixed,
            /pgf/number format/precision=2},
    ]
    \addplot+[draw=black, line width=0.4pt,
              fill={rgb,255:red,72;green,120;blue,207}]
                coordinates {
                    (0,2.41281) (1,1.91696) (2,2.34654) (3,2.17957)
                    (4,2.15064) (5,2.70668) (6,5.38437) (7,14.95370)};
        \addlegendentry{$p = 2$}

        \addplot+[draw=black, line width=0.4pt,
                            fill={rgb,255:red,225;green,129;blue,44}]
                coordinates {
                    (0,2.07625) (1,1.95635) (2,2.01302) (3,2.17374)
                    (4,2.43297) (5,4.17219) (6,6.26748) (7,21.20120)};
        \addlegendentry{$p = 3$}

\draw[dashed, red, line width=0.6pt]
    (axis cs:-0.5,2) -- (axis cs:7.5,2);

\node[
        font=\scriptsize,
        red,
    anchor=south east,
    fill=white,
    inner sep=1.5pt
] at (axis cs:7.35,2) {$2\,\mathrm{s}$};

    \end{axis}
    \end{tikzpicture}
        \caption{Cumulative mesh-only regeneration time over 200 evenly spaced configurations of the final rotating-square setup, for polynomial degrees $p = 2$ and $p = 3$ at successive uniform refinement levels $r = 0, \ldots, 7$. Barrier Patch and tangential slip are enabled; the fluid solve is excluded. The dashed red line marks $2\,\mathrm{s}$ and is drawn for reference only.}
    \label{fig:remesh-time}
\end{figure}

\begin{table}[H]
\centering
\footnotesize
\setlength{\tabcolsep}{5pt}
\caption{Cost breakdown of one parameterization-driven ALE time step.
    Measured times are per-step wall-clock averages over 200 time steps
    on the production mesh (4864 elements, $\Delta t = 0.005\,\mathrm{s}$)
    with 8 OpenMP threads, for polynomial degrees $p=2$ and $p=3$.
    The barrier-patch optimization runs on the pre-refinement single-patch
    geometry, so its cost is essentially independent of $r$.}
\label{tab:algorithm-cost-breakdown}
\begin{tabular}{@{}llSSp{4.8cm}@{}}
\toprule
Step in Algorithm~\ref{alg:timestep} & Relative cost
    & \multicolumn{2}{c}{Measured (s/step)} & Comment \\
\cmidrule(lr){3-4}
 & & {$p=2$} & {$p=3$} & \\
\midrule
Tangential reparameterization (line~\ref{algline:tangslip})
    & negligible & {$<10^{-4}$} & {$<10^{-4}$}
    & one-dimensional curve shift \\
Initial ruled/transfinite construction (line~\ref{algline:ruled})
    & negligible & {$<10^{-4}$} & {$<10^{-4}$}
    & explicit geometric construction \\
Barrier-patch optimization (line~\ref{algline:barrier})
    & small & 0.0075 & 0.0087
    & nonlinear optimization with quadrature evaluations \\
Quasi-interpolation transfer (line~\ref{algline:transfer})
    & dominant geometry cost & 0.018  & 0.020
    & local operator, linear in degrees of freedom \\
Mesh-velocity evaluation (line~\ref{algline:meshvel})
    & negligible & {$1.1\times10^{-4}$} & {$1.2\times10^{-4}$}
    & control-point differences \\
Fluid solve (line~\ref{algline:fluid})
    & dominant total cost & 1.57 & 4.04
    & Navier--Stokes solve (Picard iterations) \\
\bottomrule
\end{tabular}
\end{table}

The measurements confirm the qualitative picture. On the production mesh with 4864 elements, the barrier-patch optimization takes approximately $7.5\,\mathrm{ms}$ and $8.7\,\mathrm{ms}$ per time step for $p=2$ and $p=3$, respectively, while the tangential reparameterization remains below $0.1\,\mathrm{ms}$ and the mesh-velocity evaluation costs only about $0.11$–$0.12\,\mathrm{ms}$. The barrier-patch cost remains relatively insensitive to mesh refinement, since the optimization is performed on the pre-refinement single-patch geometry rather than on the refined computational mesh. The quasi-interpolation transfer requires approximately $18\,\mathrm{ms}$ for $p=2$ and $20\,\mathrm{ms}$ for $p=3$, making it the dominant contribution to the geometry-side cost. Nevertheless, this overhead remains small compared with the fluid solve, which requires approximately $1.57\,\mathrm{s}$ and $4.04\,\mathrm{s}$ per time step for $p=2$ and $p=3$, respectively. Thus, the additional cost introduced by the parameterization-driven ALE procedure constitutes only a small fraction of the overall computational cost.

\paragraph{Portability to a finite element solver}
The proposed G+Smo library is stand-alone: it constructs the spline geometry of the physical domain independently of any field solver. Through preCICE, this geometry is passed to \textsc{deal.II}~\citep{dealII95}; after receiving the spline representation, \textsc{deal.II} carries out an independent finite-element flow calculation. The \textsc{deal.II} branch therefore serves only as an interoperability test, not as a component or dependency of the geometry library. This workflow is illustrated in \Cref{fig:fem-export-workflow}.

\begin{figure}[H]
    \centering
    \begin{tikzpicture}[>=Stealth, font=\small,
        box base/.style={draw, line width=0.9pt, rounded corners=3pt,
                         minimum width=2.8cm, minimum height=1.35cm,
                         align=center, inner sep=5pt, font=\small},
        box gismo/.style={box base, draw=cGismo, fill=cGismoBg,
                          text=cGismo!65!black},
        box deal/.style={box base, draw=cDeal, fill=cDealBg,
                         text=cDeal!65!black},
        box remesh/.style={box base, draw=cRemesh, fill=cRemeshBg,
                           text=cRemesh!60!black},
        data arrow/.style={->, >=Stealth, line width=1.1pt},
        force arrow/.style={data arrow, draw=cDeal},
        disp arrow/.style={data arrow, draw=cGismo},
        write arrow/.style={data arrow, draw=cGismo},
        read arrow/.style={data arrow, draw=cDeal},
        ctrl arrow/.style={->, >=Stealth, line width=0.9pt, draw=cRemesh,
                           dash pattern=on 4pt off 2pt},
        col title/.style={font=\small\bfseries, black!85},
        data label/.style={font=\scriptsize, fill=white,
                           inner sep=2pt, rounded corners=1pt},
        side label/.style={font=\scriptsize\itshape, black!65},
        precice label/.style={font=\scriptsize\bfseries, text=black!55,
                              inner sep=3pt}]

      \definecolor{cGismo}{RGB}{230,97,1}
      \definecolor{cGismoBg}{RGB}{253,236,217}
      \definecolor{cDeal}{RGB}{31,119,180}
      \definecolor{cDealBg}{RGB}{222,235,247}
      \definecolor{cRemesh}{RGB}{44,160,44}
      \definecolor{cRemeshBg}{RGB}{233,245,233}
      \definecolor{cPrecice}{RGB}{148,103,189}

      \def\colL{0}
      \def\colR{6.8}

      \node[col title] at (\colL, 3.3) {Geometry library (G+Smo)};
      \node[col title] at (\colR, 3.3) {Optional consumer (\textsc{deal.II})};
      \draw[black!25, line width=0.5pt] (\colL-1.7, 2.95) -- (\colL+1.7, 2.95);
      \draw[black!25, line width=0.5pt] (\colR-1.7, 2.95) -- (\colR+1.7, 2.95);

      \node[box gismo] (strMesh) at (\colL, 1.8)
        {Spline geometry};
      \node[box deal] (cfdL) at (\colL, -0.4)
        {Library API};
      \node[box deal] (cfdR) at (\colR, -0.4)
        {FEM consumer};
      \node[box remesh] (remesh) at (\colR, -2.8)
        {G+Smo geometry};

      \begin{scope}[on background layer]
        \node[draw=cPrecice!70, line width=0.7pt,
              dash pattern=on 4pt off 2pt,
              rounded corners=5pt, inner sep=10pt,
              fill=cPrecice!5,
              fit=(cfdL)(cfdR)] (preciceBox) {};
      \end{scope}
      \node[precice label, text=cPrecice!70!black]
        at (preciceBox.north) [above=5pt] {preCICE};

      \draw[write arrow] ([xshift=-10pt]strMesh.south)
        -- ([xshift=-10pt]cfdL.north);

      \draw[read arrow] ([xshift=10pt]cfdL.north)
        -- ([xshift=10pt]strMesh.south);

      \node[
          side label,
          text=cGismo!70!black,
          align=center,
          anchor=east
      ] at ($(strMesh.south)!0.62!(cfdL.north)+(-0.35,0)$)
          {write\\spline};

      \node[
          side label,
          text=cDeal!70!black,
          align=center,
          anchor=west
      ] at ($(strMesh.south)!0.50!(cfdL.north)+(0.35,0)$)
          {read\\spline};

      \draw[force arrow] ([yshift=4pt]cfdR.west)
        -- node[data label, above=1pt]{Optional request}
           ([yshift=4pt]cfdL.east);
      \draw[disp arrow] ([yshift=-4pt]cfdL.east)
        -- node[data label, below=1pt]{Geometry output}
           ([yshift=-4pt]cfdR.west);

      \draw[ctrl arrow] (remesh.north)
        -- node[data label, right=2pt, text=cRemesh!60!black]
               {provide $\mathcal G^n$}
           (cfdR.south);
      \draw[ctrl arrow] (remesh.west) -|
        node[data label, below=1pt, pos=0.22, text=cRemesh!60!black]
             {regenerate $\mathcal G^n$}
        ([xshift=-6pt]cfdL.south);
    \end{tikzpicture}

    \caption{Stand-alone geometry library and optional finite-element consumer. Through preCICE, G+Smo can operate as an independent geometry library and communicate its spline geometry to a downstream finite-element solver, which constructs its own computational mesh and discrete spaces.}
    \label{fig:fem-export-workflow}
\end{figure}

\begin{figure}[H]
    \centering
    \begin{tikzpicture}
        \definecolor{nbBlue}{RGB}{31,119,180}   
        \definecolor{nbOrange}{RGB}{230,97,1}   
        \begin{axis}[
          width=0.82\linewidth,
          height=0.36\textheight,
          xlabel={Rotation angle $\theta$ [deg]},
          ylabel={$P_y$},
          xlabel style={font=\normalsize, yshift=-3pt},
          ylabel style={font=\normalsize, yshift=6pt},
          tick label style={font=\small},
          xmin=45,
          xmax=160,
          axis line style={line width=1.2pt, black!85},
          tick style={line width=1.0pt, black!85},
          major tick length=5pt,
          minor tick length=3pt,
          grid=both,
          major grid style={line width=0.6pt, draw=gray!30},
          minor grid style={line width=0.4pt, draw=gray!15, dashed},
          minor tick num=1,
          enlarge x limits=false,
          enlarge y limits=0.05,
          legend style={
            at={(0.03,0.97)},
            anchor=north west,
            draw=black!50,
            line width=0.8pt,
            fill=white,
            fill opacity=0.92,
            text opacity=1,
            font=\small,
            row sep=3pt,
            inner sep=7pt,
            rounded corners=2pt
          },
          legend cell align={left},
          legend image post style={line width=2.4pt}
        ]
        \addplot[
          color=nbOrange,
          line width=3.2pt,
          each nth point=4
        ] table [x=angle_deg, y=momentum_y, col sep=comma]
          {original_gismo_energy_momentum.csv};
        \addlegendentry{G+Smo (IGA)}

        \addplot[
          color=nbBlue,
          dash pattern=on 10pt off 5pt,
          line width=3.0pt,
          each nth point=4
        ] table [x=angle_deg, y=Py, col sep=comma]
          {dealii_Py_30_160deg.csv};
        \addlegendentry{\textsc{deal.II} (FEM)}
        \end{axis}
    \end{tikzpicture}
    \caption{Comparison of the vertical linear momentum $P_y$ for flow past the rotating square over $\theta\in[45^\circ,160^\circ]$, computed with G+Smo (IGA) and \textsc{deal.II} (FEM) on the same sequence of physical geometries. The frequencies agree to $0.51\%$ and the peak magnitudes to $2.5\%$, with a phase offset accumulating to $0.54^\circ$ over the interval. The agreement confirms the portability of the stand-alone G+Smo geometry library to an independent FEM solver.}
    \label{fig:gismo-dealii-py}
\end{figure}

For the optional portability test, \Cref{fig:gismo-dealii-py} compares the vertical linear-momentum histories obtained with G+Smo (IGA) and the downstream \textsc{deal.II} consumer (FEM) on the same sequence of physical geometries over $\theta\in[45^\circ,160^\circ]$, matching the angular interval used for the comparison with Schalk et al. in \Cref{fig:graz-py}. On this interval the two computations agree to $0.51\%$ in the frequency of $P_y$ and to $2.5\%$ in its peak magnitude, with a phase offset that grows monotonically to $0.54^\circ$ of rotation, roughly one twelfth of a period, by the end of the interval; the pointwise difference is dominated by this drift rather than by a discrepancy in amplitude. A slow relative phase drift is expected here, since the two solvers share only the geometry and differ in basis, quadrature and linear algebra. Combining this with the comparison against \Cref{fig:graz-py}, and averaging all three over the common interval $\theta\in[45^\circ,160^\circ]$ covered by the \textsc{deal.II} run---hence over a shorter window than the one used in \Cref{fig:graz-freq-table}---the frequency of $P_y$ is $0.1472\,\mathrm{Hz}$ for Schalk et al.~\citep{mesh_regeneration_Graz}, $0.1472\,\mathrm{Hz}$ for \textsc{deal.II} and $0.1465\,\mathrm{Hz}$ for the present isogeometric solver, a spread of $0.51\%$ across three mutually independent flow discretizations driven by the same sequence of regenerated geometries. This confirms that the geometry produced by the stand-alone library is portable across IGA and FEM solvers. We emphasize that \textsc{deal.II} is an optional validation consumer, not a dependency of the proposed library.

\section{Application to 3D rotating-rotor benchmark}
\label{sec:3d-rotor}

To demonstrate that the proposed method extends naturally to three dimensions, we apply the parameterization-driven strategy to axial flow through a rotating rotor enclosed in a cylindrical casing.

\paragraph{Problem setup}
The physical domain, shown in \Cref{fig:rotor-setup}, is the annular region between a fixed cylindrical outer wall of radius $R = 1.5\,\mathrm{m}$ and length $L = 6\,\mathrm{m}$ and a four-bladed rotor centered on the cylinder axis. The rotor cross-section has four-fold symmetry with a bounding diameter of $2\,\mathrm{m}$, and its profile is twisted by $100^\circ$ along the axial ($z$-)direction, producing a helical geometry representative of axial-flow machinery.

\begin{figure}[H]
    \centering
    \includegraphics[width=0.3\linewidth]{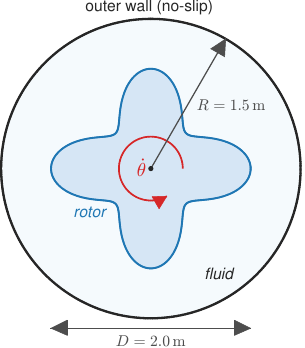}
    \qquad
    \includegraphics[width=0.55\linewidth]{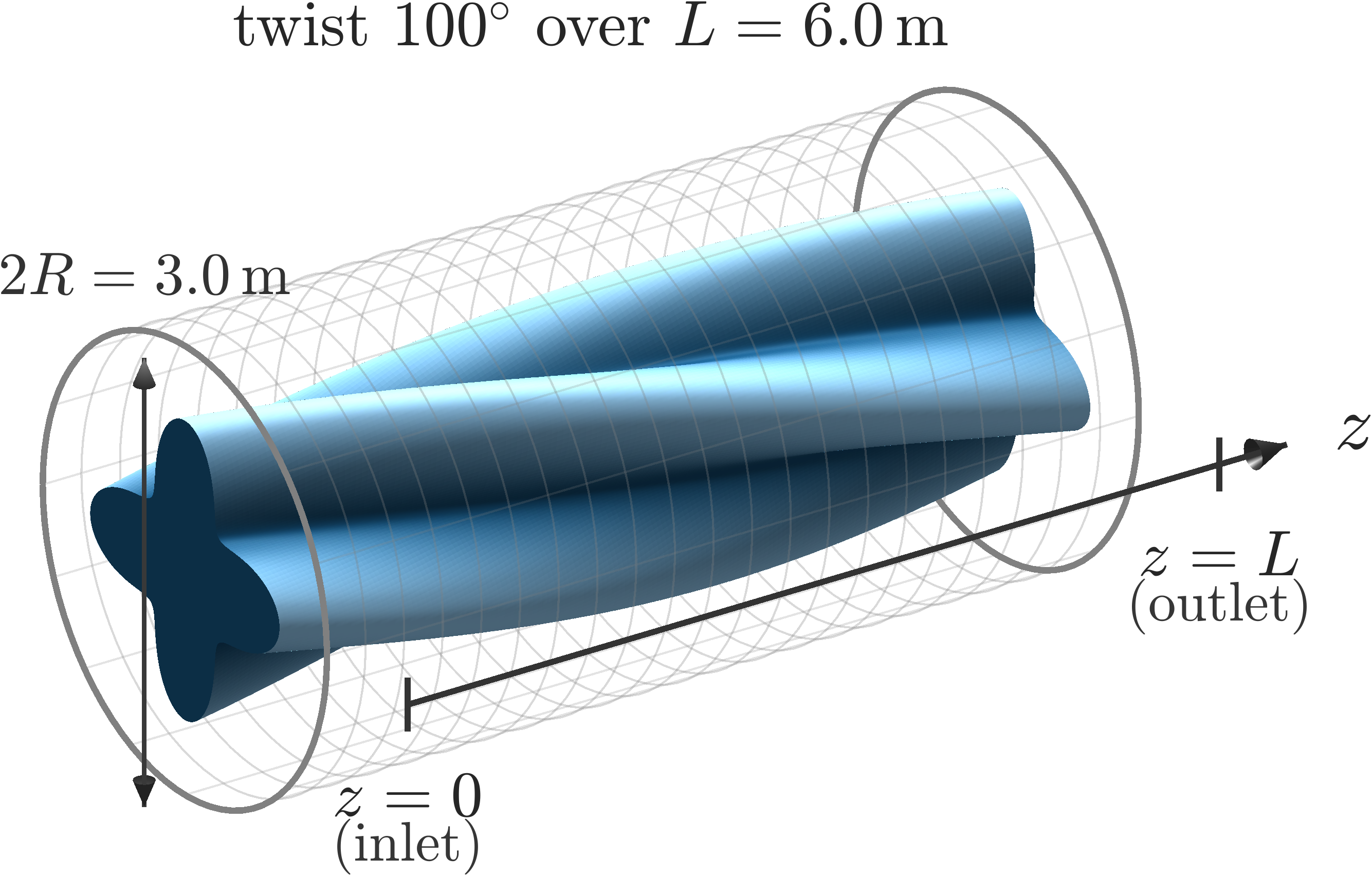}
    \caption{Setup of the 3D rotating-rotor benchmark. (a)~Cross-section through the cylindrical casing showing the four-bladed rotor profile and the annular fluid region. (b)~Longitudinal section along the $z$-axis: the rotor blade envelope is twisted by $100^\circ$ between inlet and outlet, and uniform axial inflow/stress-free outflow conditions are imposed at the two ends.}
\label{fig:rotor-setup}
\end{figure}

A uniform axial inflow $v_{\mathrm{in}} = 20\,\mathrm{m/s}$ is prescribed at $z = 0$, no-slip conditions are imposed on the cylindrical wall, and a stress-free condition is applied at the outlet $z = L$. On the rotor surface, the velocity is set by rigid-body rotation about the $z$-axis,
\begin{equation}
    \mathbf{u}^f\big|_{\Gamma_{\mathrm{rotor}}} = \dot{\theta}\,\mathbf{e}_z \times (\mathbf{x} - \mathbf{x}_c),
\end{equation}
with angular velocity $\dot{\theta} = 30^\circ/\mathrm{s}$. The kinematic viscosity is $\nu^f = 0.6\,\mathrm{m^2/s}$, giving a Reynolds number $\mathrm{Re} = v_{\mathrm{in}} L / \nu^f = 200$. Backward Euler time integration is used with a constant step $\Delta t = 0.05\,\mathrm{s}$ over the interval $[0, 5]\,\mathrm{s}$.

As in the two-dimensional cases, the fluid domain is decomposed into four ring-shaped patches, and the barrier-patch parameterization is reconstructed at each time step to update the volumetric mesh to the current rotor configuration.

\paragraph{Results}
\Cref{fig:rotating-rotor-3d} shows representative snapshots at $t = 2.5\,\mathrm{s}$ (after a $75^\circ$ rotation) and $t = 5.0\,\mathrm{s}$ (after $150^\circ$). At both instants, the four volumetric patches conform to the helical rotor surface (panels~a,\,b), and the velocity field exhibits the expected acceleration through the blade-tip passages (panels~c,\,d). Two complementary cross-sectional views are included: a vertical cut through the cylinder axis (panels~e,\,f), which reveals the annular velocity distribution around the four-lobed rotor profile, and a constant-$z$ cut near the inlet (panels~g,\,h), where the four-fold symmetric pattern is seen to rotate with the rotor.

\begin{figure}[H]
    \centering

    \includegraphics[width=0.8\linewidth]{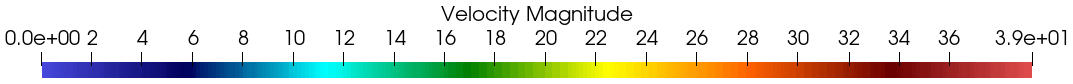}\\

    \vspace{0.4em}

    \subfigure[Parameterization]{
        \includegraphics[width=0.23\linewidth]{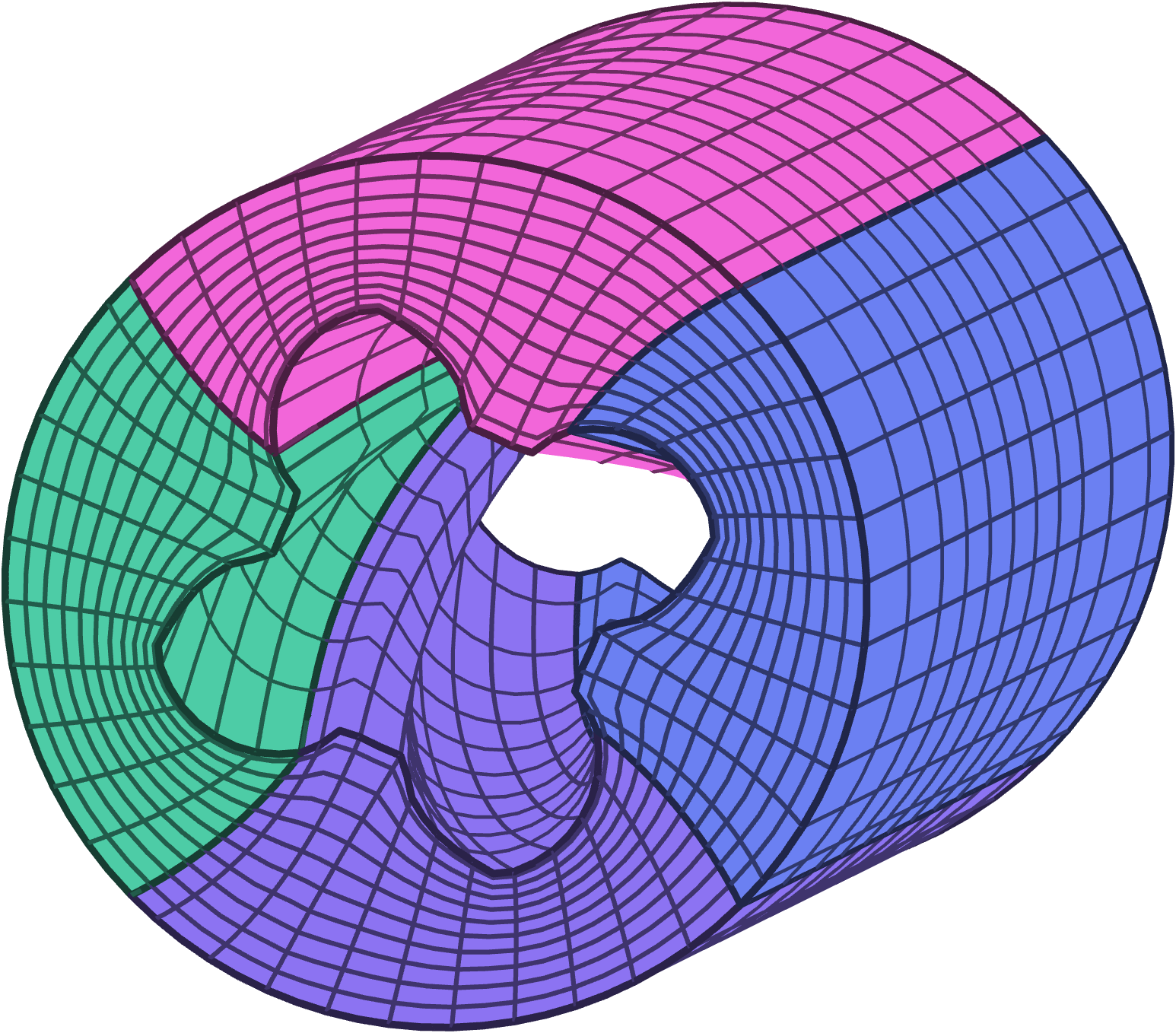}}
    \hfill
    \subfigure[Velocity magnitude]{
        \includegraphics[width=0.23\linewidth]{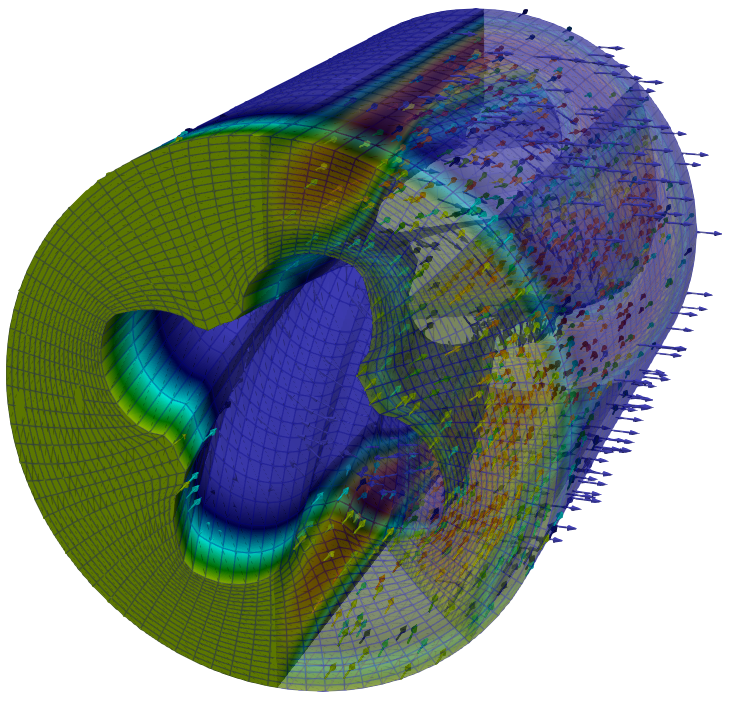}}
    \hfill
    \subfigure[Vertical cut]{
        \includegraphics[width=0.23\linewidth]{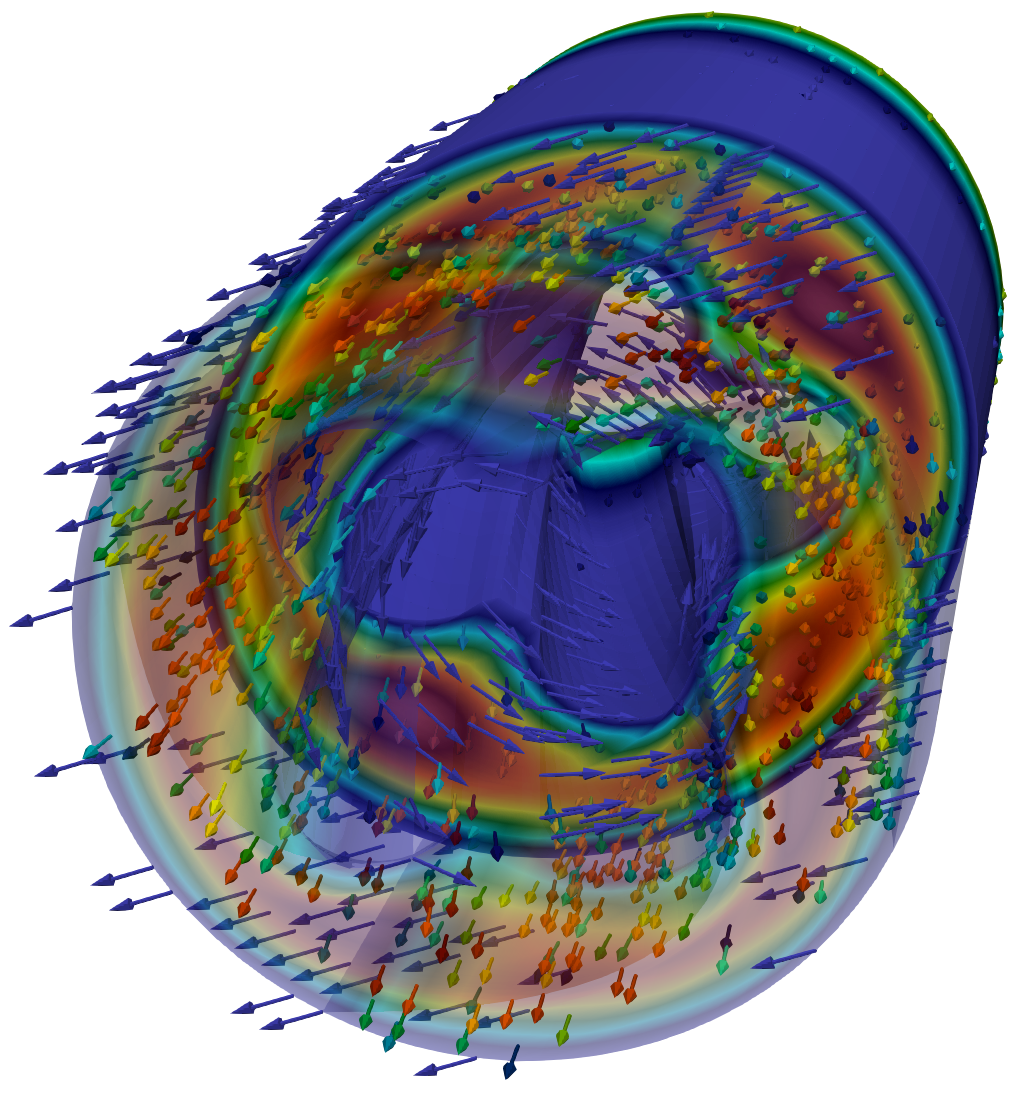}}
    \hfill
    \subfigure[Horizontal cut]{
        \includegraphics[width=0.23\linewidth]{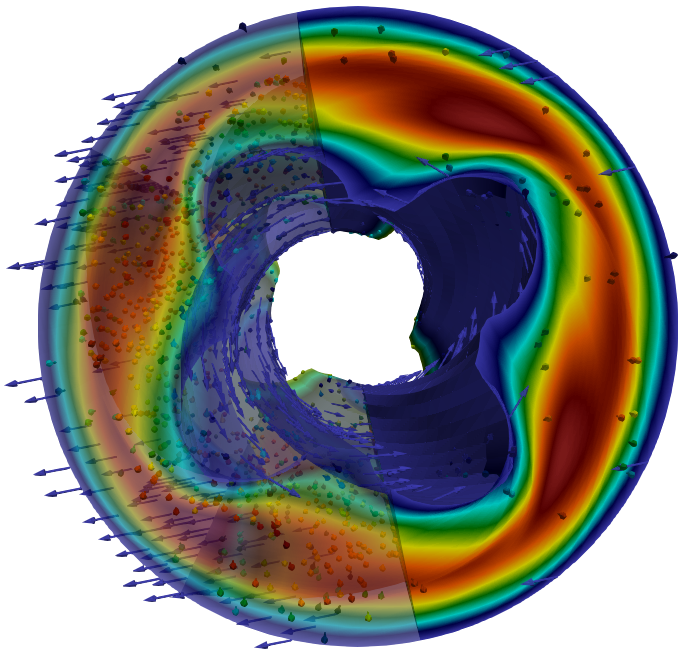}}\\

    \par\smallskip
    {\itshape (a)--(d): $t = 2.5\,\mathrm{s}$}
    \par

    \vspace{0.5em}

    \subfigure[Parameterization]{
        \includegraphics[width=0.23\linewidth]{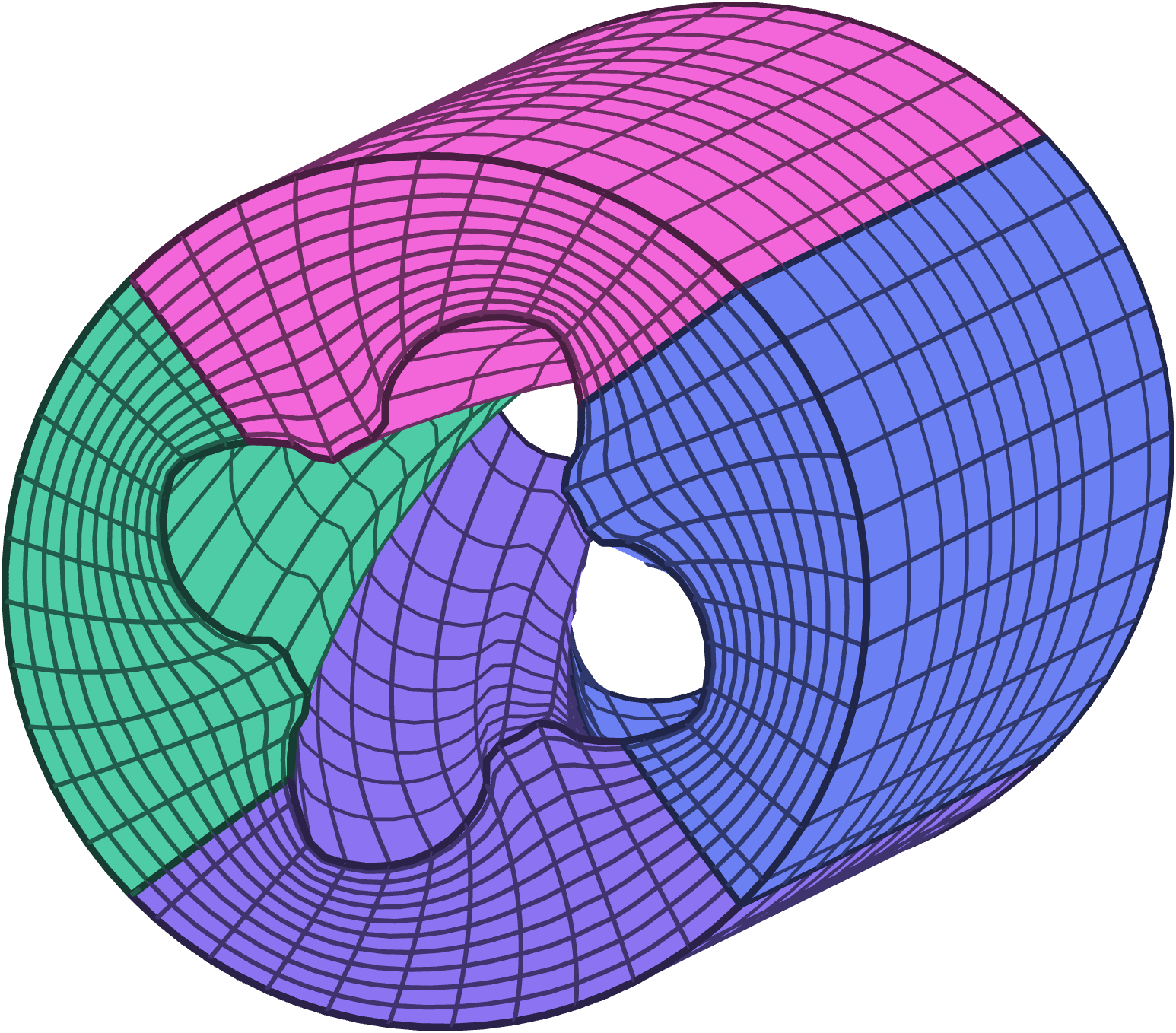}}
    \hfill
    \subfigure[Velocity magnitude]{
        \includegraphics[width=0.23\linewidth]{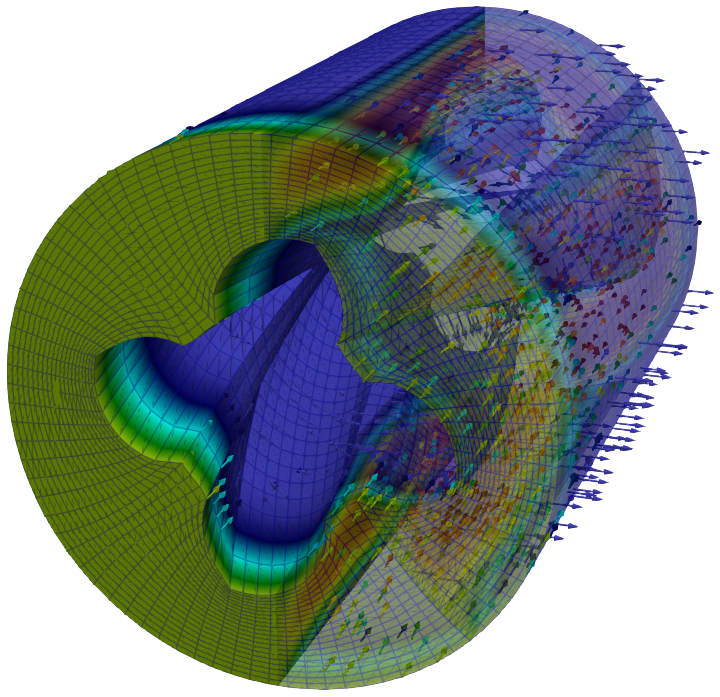}}
    \hfill
    \subfigure[Vertical cut]{
        \includegraphics[width=0.23\linewidth]{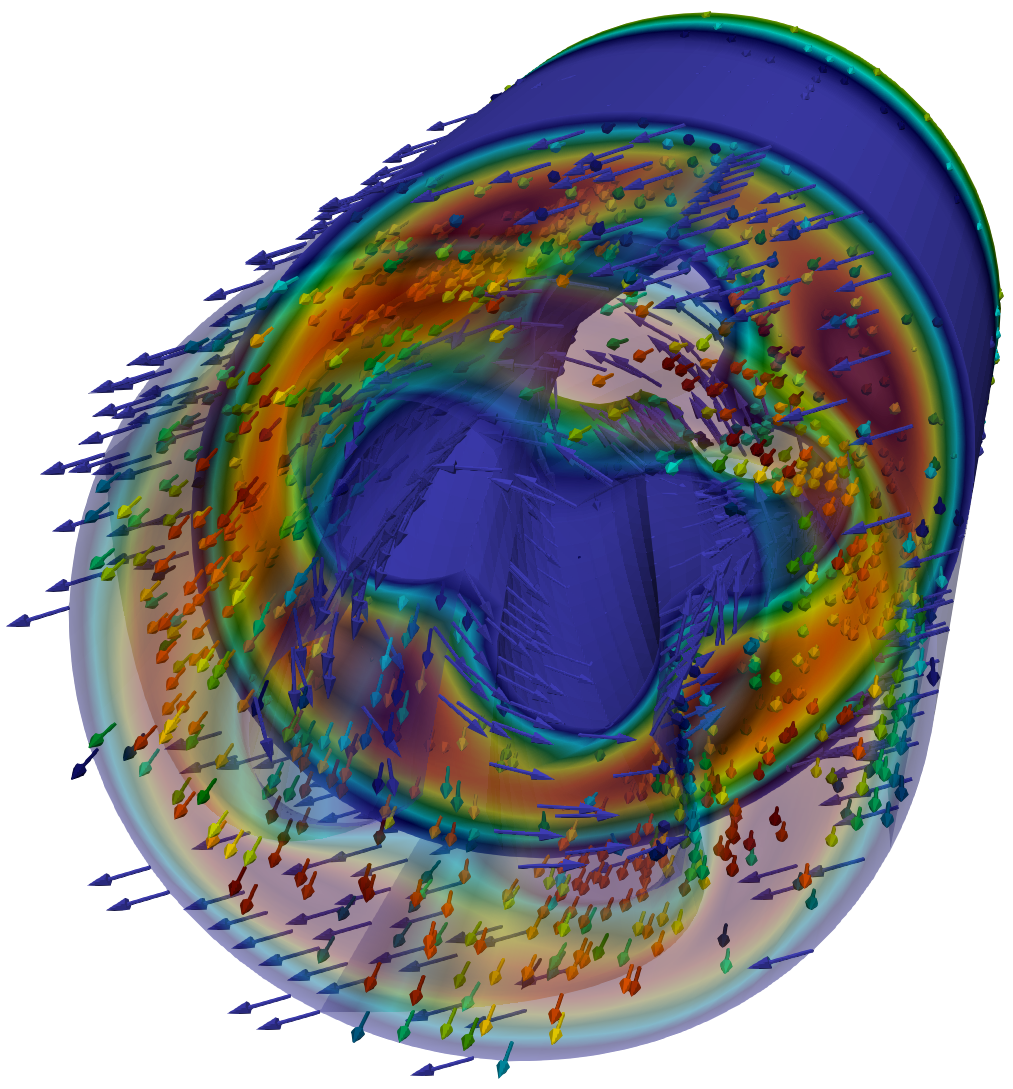}}
    \hfill
    \subfigure[Horizontal cut]{
        \includegraphics[width=0.23\linewidth]{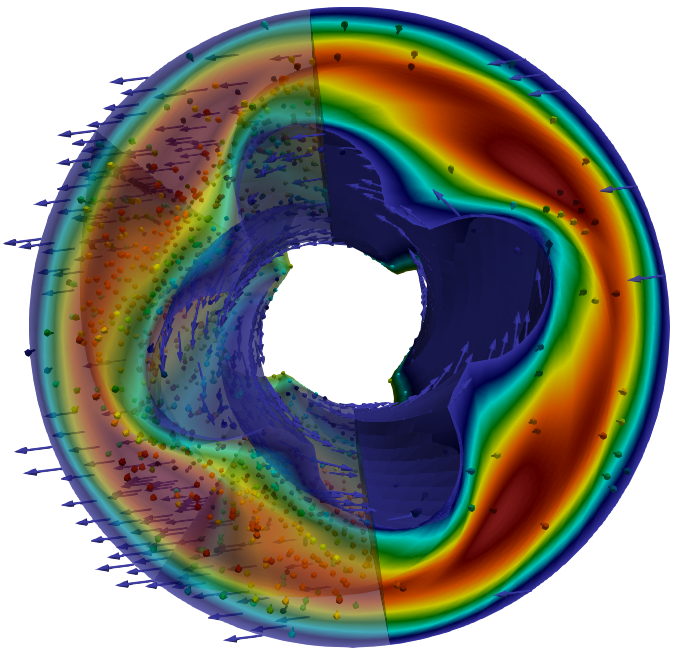}}

    \par\smallskip
    {\itshape (e)--(h): $t = 5.0\,\mathrm{s}$}
    \par

    \caption{Three-dimensional rotating-rotor benchmark at $t = 2.5\,\mathrm{s}$ (top row) and $t = 5.0\,\mathrm{s}$ (bottom row); each column corresponds to one view type. Panels~(a,\,e) show the four conforming volumetric patches of the parameterization adapted to the helically twisted rotor. Because the outer cylinder wall would otherwise occlude the interior flow, the velocity field is shown on three complementary views: panels~(b,\,f) display the magnitude on the rotor surface together with the outer cylinder; panels~(c,\,g) use a vertical cut aligned with the cylinder axis to reveal the velocity distribution around the four-lobed rotor profile; and panels~(d,\,h) use a horizontal cut near the inlet that exposes the four-fold symmetric flow pattern on a constant-$z$ plane. All velocity panels share the common color bar shown at the top.}
    \label{fig:rotating-rotor-3d}
\end{figure}

The tangential-slip reparameterization introduced in \Cref{sec:closed-domain} is applied in the same way as in the two-dimensional rotating-square case, despite the substantially more demanding geometric setting: the four ring-shaped patches must accommodate a $100^\circ$ helical twist along the axis while remaining conforming to the rotor surface as it rotates. Over the simulated interval $[0, 5]\,\mathrm{s}$, the rotor undergoes a cumulative rotation of $150^\circ$ and the parameterization is updated at every one of the 100 time steps without a single failure of the barrier-patch parameterization. This indicates that the proposed parameterization-driven strategy generalizes to 3D rotating-machinery configurations with no algorithmic modifications beyond the natural extension of the patch topology.

\section{Conclusions and outlook}
\label{sec:conclusions}

We have presented a parameterization-driven approach to mesh motion in isogeometric ALE-FSI. Instead of deforming a reference mesh from one time step to the next, the fluid mesh is obtained at each step as a multi-patch spline parameterization built from the current interface geometry. The method relies on three ingredients: a barrier-based parameterization procedure that keeps the Jacobian positive; a tangential-slip reparameterization that allows the parametric pre-images to move along the interface, which is needed when the body rotates through large cumulative angles; and a constant-preserving quasi-interpolation operator that transfers the discrete solution between consecutive parameterizations while satisfying the discrete geometric conservation law at the algebraic level. The proposed method works for meshes defined by elements of any order. However, our analysis employs IGA-native spline-based mesh representations, which provide typically higher resistance to large mesh deformations, as we have demonstrated in \cref{sec:mesh-quality}.

On the perpendicular-flap and Turek--Hron FSI2 benchmarks, the method reproduces established reference results and keeps the Jacobian bounded away from zero for the full duration of the simulations. The rotating-square benchmark, in which classical mesh-deformation ALE fails after a limited rotation, is handled without difficulty by the combination of barrier-based parameterization and tangential slip. The prescribed falling-sphere test completes a $40\,\mathrm{m}$ translation with a positive Jacobian, zero reconnections, and zero mesh-size growth; the corresponding reference's reconnection strategy increases the node and element counts by approximately $10\%$. The coupled free-fall test reproduces the wall-corrected terminal speed with a relative difference of $0.29\%$. The three-dimensional rotating-rotor example shows that the same construction applies to volumetric spline parameterizations. Since the parameterization is a purely geometric object, we also used the sequence of G+Smo-generated parameterizations to drive a deal.II finite element computation and obtained results consistent with the isogeometric solution.

The current framework has three limitations worth stating. The tangential-slip strategy, as implemented here, targets closed boundaries of convex or nearly convex bodies; strongly non-convex or multi-body configurations require a more general criterion for redistributing the pre-images. As demonstrated by the falling-sphere tests, large translational motion can be handled while the topology of the fluid domain and the patch connectivity remain unchanged. In contrast, topology-changing events such as collision, merging, or splitting of bodies are outside the scope of the present formulation, since they require a change of the patch layout itself. Such cases would require repartitioning or a coupling with unfitted methods. The quasi-interpolation operator preserves constants but is not conservative for integral quantities such as mass or momentum, which matters for long-time simulations with strong mass transport and would call for a projection-based transfer instead.

A natural next step is to extend the tangential-slip idea to the non-convex and multi-body cases mentioned above, and to combine the per-step parameterization with adaptive or quality-driven strategies. More broadly, because the parameterization is built independently of the field solver, the same geometry module can drive other field solvers: the \textsc{deal.II} computation of \Cref{sec:flow-past-square} is a first demonstration of this, and we intend to extend it beyond FSI to other multi-physics problems with moving interfaces.

\section*{Acknowledgment}
The authors gratefully acknowledge Hana Horn\'{i}kov\'{a} and Bohum\'{i}r Bastl from the University of West Bohemia in Pilsen, Czech Republic, for generously sharing their IGA fluid solver and for their helpful supports throughout this work. We also thank Teresa Schalk and Thomas-Peter Fries at Graz University of Technology for providing the rotating-square benchmark parameters and reference kinetic-energy and momentum data used in the validation.

\section*{Funding}
This research has been conducted as part of the project FlexFloat with project number 19002 of the research program Open Technology, which is (partly) financed by the Dutch Research Council (NWO), Netherlands.

\section*{CRediT authorship contribution statement}
\textbf{J. Li:} Conceptualization, Formal analysis, Investigation, Methodology, Software, Validation, Visualization, Writing – original draft.
\textbf{Y. Ji:} Methodology, Supervision, Visualization, Software, Writing – original draft, Writing – review \& editing.
\textbf{H. M. Verhelst:} Supervision, Funding acquisition, Software, Writing – review \& editing.
\textbf{J. H. den Besten:} Funding acquisition, Supervision, Writing – review \& editing.
\textbf{M. M\"{o}ller:} Conceptualization, Funding acquisition, Project administration, Supervision, Software, Writing – review \& editing.

\section*{Declaration of competing interest}
The authors declare that they have no known competing financial interests or personal relationships that could have appeared to influence the work reported in this paper.

\section*{Data availability}
Data will be made available upon request.

\section*{References}
\bibliography{cas-refs}

\end{document}